\documentclass[11pt,usenames,dvipsnames]{amsart}
\usepackage{a4wide}
\usepackage{amssymb}
\usepackage{amsmath}
\usepackage{amsthm}
\usepackage{mathrsfs}
\usepackage{yfonts}
\usepackage{euscript}
\usepackage{tikz-cd}
\usepackage{upgreek}
\usepackage[all]{xypic}
\usepackage{xcolor} %{\color{Maroon/ForestGreen}  }
\usepackage{comment}
\usepackage{xr-hyper}

\usepackage{enotez, translations}
\usepackage[unicode,psdextra,dvipdfmx]{hyperref} %dvipdfmx deleted by Otmar
\usepackage{bookmark}
\usepackage[hyperpageref ]{backref}
\usepackage[only,llbracket,rrbracket]{stmaryrd}

\newcommand{\Footnote}[1]{\footnote{{\color{Sepia} #1}}}  % Fussnoten nur f\"{u}r Peter und Otmar farbig in Sepia
\renewcommand{\Footnote}[2][]{\relax}  % nur Fussnoten f\"{u}r Peter und Otmar ausblenden
\theoremstyle{plain}

\newcommand{\id}{\operatorname{id}}

\newcommand{\pr}{\operatorname{pr}}

\newcommand{\Hom}{\operatorname{Hom}}

\newcommand{\Fil}{\operatorname{Fil}}

\newcommand{\Rep}{\operatorname{Rep}}

\newcommand{\Gal}{\operatorname{Gal}}

\newcommand{\cor}{\operatorname{cor}}

 \newcommand{\CC}{\mathbb{C}}

 \newcommand{\QQ}{\mathbb{Q}}
 
 \newcommand{\ZZ}{\mathbb{Z}}

\newcommand{\bA}{\mathbf{A}}
 \newcommand{\bB}{\mathbf{B}}

  \newcommand{\bE}{\mathbf{E}}

 \newcommand{\cO}{\mathcal{O}}

\renewcommand{\cR}{\mathcal{R}}

 \newcommand{\Cp}{\CC_p}

 \newcommand{\be}{\begin{equation}}
\newcommand{\ee}{\end{equation}}

\newcommand{\fpbar}{\ifmmode {\overline{\mathbb{F}_p}}\else$\mathbb{F}_p$\ \fi}

\newcommand{\fp}{\ifmmode {\mathbb{F}_p}\else$\mathbb{F}_p$\ \fi}
\newcommand{\zp}{\ifmmode \mathbb{Z}_p\else$\mathbb{Z}_p$\ \fi}
\newcommand{\zpur}{\ifmmode \widehat{\zp^{ur}}\else $\widehat{\zp^{ur}}$\ \fi}

 \newcommand{\TLT}{T_{\pi}}

   \newcommand{\cyc}{\mathrm{cyc}}
  \newcommand{\inv}{\mathrm{inv}}
  
 \newcommand{\cris}{\mathrm{cris}}
 \newcommand{\dR}{\mathrm{dR}}
 \newcommand{\Iw}{\mathrm{Iw}}
 \newcommand{\rig}{\mathrm{rig}}
  \newcommand{\LT}{\mathrm{LT}}
  \newcommand{\Tate}{\mathrm{Tate}}

\newcommand{\La}{\ifmmode\Lambda\else$\Lambda$\fi}

\newcommand{\q}{\ifmmode {\mathbb Q}\else${\mathbb Q}$\ \fi}
\newcommand{\qp}{\ifmmode {\mathbb Q}_p\else${\mathbb Q}_p$\ \fi}

\newcommand{\Q}{\ifmmode {\mathbb Q}\else${\mathbb Q}$\ \fi}
\newcommand{\ql}{\ifmmode {{\mathbb Q}_l}\else${\mathbb Q}_l$\ \fi}

\newtheorem{theorem}{Theorem}[subsection]
\newtheorem{corollary}[theorem]{Corollary}
\newtheorem{lemma}[theorem]{Lemma}

\newtheorem{proposition}[theorem]{Proposition}

\theoremstyle{definition}

\theoremstyle{remark}
\newtheorem{remark}[theorem]{Remark}

\newtheorem*{acknowledgement}{Acknowledgements}

  1

\newcommand{\cN}{\mathcal{N}}
\newcommand{\GG}{\mathbb{G}}

\title{On Lubin-Tate regulator maps and Kato's explicit reciprocity law}

\author{Takamichi Sano and Otmar Venjakob}

\address{Osaka Metropolitan University,
Department of Mathematics,
3-3-138 Sugimoto\\Sumiyoshi-ku\\Osaka\\558-8585,
Japan}
\email{tsano@omu.ac.jp}
\address{Universit\"{a}t Heidelberg,  Mathematisches Institut,  Im Neuenheimer Feld 288,  69120
Heidelberg,  Germany,}
\urladdr{http://www.mathi.uni-heidelberg.de/~venjakob/}
\email{venjakob@mathi.uni-heidelberg.de}
\begin{document}

\begin{abstract}
We extend the interpolation property of the Lubin-Tate regulator map from \cite{SV24} to Artin characters and show a reciprocity law in the sense of Cherbonnier-Colmez. This allows us to provide a new proof of Kato's explicit reciprocity law for Lubin-Tate formal groups.
\end{abstract}

\maketitle
\tableofcontents

\section{Introduction}

%Notation:
%\begin{itemize}
%\item $L$: a finite extension of $\QQ_p$,
%\item $\pi_L \in o_L$: a prime element,
%\item $({\rm LT}, +_{\rm LT})$: a Lubin-Tate formal group over $o_L$ for $\pi_L$,
%\item $g_{\rm LT}(Z):= \left( \frac{\partial (X+_{\rm LT} Y)}{\partial X} |_{(X,Y)=(0,Z)}\right)^{-1} \in o_L[[Z]]^\times$,
%\item $\log_{\rm LT}(Z) \in L[[Z]]$: the formal logarithm (which satisfies $\frac{d}{dZ}\log_{\rm LT}(Z) = g_{\rm LT}(Z)$),
%\item $\partial_{\rm inv}:= g_{\rm LT}(Z)^{-1}\frac{d}{dZ}$,
%\item ${\rm LT}_n:= {\rm LT}[\pi_L^n]$: the $\pi_L^n$-torsion points,
%\item $L_n:=L({\rm LT}_n)$,
%\item $L_\infty:=\bigcup_n L_n$,
%\item $T:=\varprojlim_n {\rm LT}_n$: the Tate module of ${\rm LT}$,
%\item $V:=L\otimes_{o_L} T$,
%\item $\cN: o_L((Z)) \to o_L((Z))$: Coleman's norm operator,
%\item $\eta =(\eta_n)_n \in \varprojlim_n {\rm LT}_n = T$: an $o_L$-basis,
%\item $g_u=g_{u,\eta} \in (o_L((Z))^\times)^{\cN=1}$: the Coleman power series for $u \in \varprojlim_n L_n^\times$,
%\item $H^i_{\rm Iw}(L_\infty/L,-):=\varprojlim_n H^i(L_n,-)$: the Iwasawa cohomology.
%\end{itemize}

\subsection{Background}

In \cite[Cor.\ 8.7]{SV15} Schneider and the second named author reproved a special case of Kato's explicit reciprocity law via Lubin-Tate $(\varphi,\Gamma)$-modules. It is a natural question - arisen by the first named author - whether one can also give a proof of the remaining cases by such methods.

In order to describe an  answer we first review Kato's explicit reciprocity law starting with  some notations. Let $L$ be a finite extension of $\mathbb{Q}_p$ and fix a Lubin-Tate formal group ${\rm LT}$ over the ring of integers $o_L$ with uniformizer $\pi_L$. Let $\eta=(\eta_n)_n$ be an $o_L$-generator of the Tate module $T=T_\pi:=\varprojlim_n {\rm LT}[\pi_L^n]$ and set $V:=L\otimes_{o_L} T$. The $\pi_L^n$-division points generate a tower of Galois extensions $L_n:=L({\rm LT}[\pi_L^n])$ of $L$. We set $L_\infty:=\bigcup_n L_n$ and $\Gamma_L:=\Gal(L_\infty/L)$. For a norm compatible system $u=(u_n)_n\in \varprojlim_n L_n^\times$, let $g_{u,\eta}\in o_L((Z))^\times$ be the Coleman Laurent series, which satisfies $g_{u,\eta}(\eta_n)=u_n$ for all $n.$ Let $\partial_{\mathrm{inv}}$ denote the invariant derivative operator with respect to ${\rm LT}$  (see (\ref{def:partial}) below).

The {\it equivariant Coates-Wiles homomorphisms} for $j,m \geq 1$ are defined by
\begin{equation*}
   \psi_{{\rm CW},m}^j :  \varprojlim_n L_n^\times  \to
   %L_m(\chi_{\rm LT}^j)
   L_m; \ u \mapsto
    \frac{1}{j!\pi_L^{mj}} \left. \left(\partial_{\mathrm{inv}}^{j}\log(g_{u,\eta})\right)\right|_{Z=\eta_m} .
\end{equation*}
Here $\partial_{\rm inv}^j \log(g_{u,\eta})$ means $\partial_{\rm inv}^{j-1}\left( \partial_{\rm inv}(g_{u,\eta})/g_{u,\eta}\right)$.
%
%
%
%, for $r\geq 1,$ the $r$th Coates-Wiles homomorphism is given by
%\begin{equation*}
%  \psi_{{\rm CW}}^r:\varprojlim_n L_n^\times \to L(\chi_{\rm LT}^r),\;\;\; u \mapsto  \frac{1}{r!}\partial_{\mathrm{inv}}^r \log g_{u,\eta}(Z)|_{Z=0} := \frac{1}{r!}\partial_{\mathrm{inv}}^{r-1}\frac{\partial_{\mathrm{inv}}g_{u,\eta}(Z)}{ g_{u,\eta}(Z)}|_{Z=0} \ ,
%\end{equation*}
%it is Galois invariant and satisfies at least heuristically - setting $t_{\rm LT}=\log_{\rm LT}(Z)$ - the equation
%\begin{equation*}
%  \log g_{u,\eta}(Z)=\sum_r \psi_{{\rm CW}}^r(u)t_{\rm LT}^r
%\end{equation*}

We shall next introduce another map $\lambda_{m,j}: \varprojlim_n L_n^\times \to L_m$.
For $j\geq 1$, the cup product with $\eta^{\otimes(-j)} \in T^{\otimes(-j)}$ induces a ``twisting" map
$${\rm tw}_j:H^1_{\rm Iw}(L_\infty/L, \ZZ_p(1)) \to H^1_{\rm Iw}(L_\infty/L, T^{\otimes(-j)}(1)),$$ where $H^i_{\rm Iw}(L_\infty/L,-):=\varprojlim_n H^i(L_n,-)$ denotes Iwasawa cohomology.
If, for $m\geq 1$, we write ${\rm pr}_m: H^1_{\rm Iw}(L_\infty/L,-)\to H^1(L_m, -)$ for the projection map,  we define the second map
$$\lambda_{m,j}: \varprojlim_n L_n^\times \to L_m,$$
as the composition
$$\varprojlim_n L_n^\times \xrightarrow{{\rm Kum}} H^1_{\rm Iw}(L_\infty/L,\ZZ_p(1)) \xrightarrow{{\rm tw}_j} H^1_{\rm Iw}(L_\infty/L, T^{\otimes(-j)}(1)) \xrightarrow{{\rm pr}_m} H^1(L_m, T^{\otimes(-j)}(1)) \xrightarrow{\exp_j^\ast} L_m,$$
where the first map is the Kummer map and $\exp_j^\ast := \exp_{L_m,V^{\otimes (-j)}(1)}^\ast$ is the Bloch-Kato dual exponential map (upon identifying $D^0_{{\rm dR}, L_m}(V^{\otimes(-j)} (1))$ with $L_m$ by choosing a basis, see (\ref{def dr}) below).

Then Kato's explicit reciprocity law for Lubin-Tate formal groups is stated as follows.

\begin{theorem}[{Kato's explicit reciprocity law, \cite[Thm. II.2.1.7]{katolecture}}]\label{kato erl}
For any $j,m\geq 1$ and $u=(u_n)_n \in \varprojlim_n L_n^\times$, we have
$$\lambda_{m,j}(u)= j\cdot \psi_{{\rm CW},m}^j(u). $$
\end{theorem}

Next, we recall the explicit reciprocity law due to Schneider and the second named author.
Let $\bA_L$ be the $\pi_L$-adic completion of $o_L((Z))$, i.e.,
$$\bA_L:=\left\{ \sum_{n \in \ZZ} a_n Z^n \ \middle| \  a_n\in o_L, \ \underset{n\to -\infty}{\lim} a_n =0\right\}.$$
(See \cite[\S 4]{SV15} for the correct choice of the variable $Z=\omega_\LT$ contained in some ring of Witt vectors.)
%We regard $\bA_L$ as a subring of $W(o_{\CC_p^\flat})_L$ as in \cite[\S 4]{SV15}.
In \cite[Thm. 5.13]{SV15} an isomorphism
$${\rm Exp}^\ast : H^1_{\rm Iw}(L_\infty/L, T^{\otimes(-1)} (1)) \xrightarrow{\sim} \bA_L^{\psi=1} $$
is constructed, where $\psi=\psi_L$ denotes a trace operator (see (\ref{def psi}) below).
%In the same way as \cite[Prop. 7.1.1(i)]{colmez}, one proves that
%$$\bA_L^{\psi =1} = o_L\cdot \frac 1Z \oplus o_L[[Z]]^{\psi=1}.$$
%In particular, we have $\bA^{\psi=1}\subset \frac 1Z o_L[[Z]]$.

\begin{theorem}[{\cite[Thm. 6.2]{SV15}}]\label{thm sv}
The composition map
$$\varprojlim_n L_n^\times \xrightarrow{{\rm Kum}} H^1_{\rm Iw}(L_\infty/L,\ZZ_p(1))\xrightarrow{{\rm tw}_1} H^1_{\rm Iw}(L_\infty/L,T^{\otimes(-1)}(1)) \xrightarrow{{\rm Exp}^\ast} \bA_L^{\psi=1}$$
is explicitly given by $u \mapsto \partial_{\rm inv}\log(g_{u,\eta}):=\partial_{\rm inv}(g_{u,\eta})/g_{u,\eta}$.
\end{theorem}

\subsection{Main results}

{\it In this article, we deduce Theorem \ref{kato erl} from Theorem \ref{thm sv} using Theorem \ref{key claim} below and thus give a new proof of Kato's explicit reciprocity law}.

Our argument is as follows.
%Let $\varphi$ denote the Frobenius operator on $\bA_L$, which sends $Z$ to $[\pi_L](Z)$.
We set
$$t_{\rm LT}:=\log_{\rm LT}(Z) \in L[[Z]]$$
and for $m\geq 1$ we define
$$\varphi^{-m}=\varphi_L^{-m}: \bA_L^{\psi=1} \to L_m[[t_{\rm LT}]]; \ f \mapsto f\left(\eta_m+_{\rm LT} \exp_{\rm LT}\left(\frac{t_{\rm LT}}{\pi_L^m}\right)\right)$$
(this is well-defined: see Corollary \ref{corpsi}).
Then for $j,m \geq 1$ we define an ``evaluation map"
$${\rm ev}_{m,j}: \bA_L^{\psi =1} \to L_m$$
by
$${\rm ev}_{m,j}(f) :=\text{the coefficient of $t_{\rm LT}^{j-1}$ in $\pi_L^{-m}\varphi^{-m}(f)$}.$$
By explicit calculation, one sees that this map is explicitly given by
\begin{equation}\label{evmj}
{\rm ev}_{m,j}(f) = \frac{1}{(j-1)! \pi_L^{mj}} \bigg(\partial_{\rm inv}^{j-1} f\bigg)|_{Z=\eta_m}.
\end{equation}
(See Proposition \ref{prop:explicit}.)

The following result, which relates ${\rm Exp}^\ast$ with Bloch-Kato dual exponential maps, is one of the main results of this article.

\begin{theorem}\label{key claim}
For any $m\geq 1$ and $j\geq 1$ such that $\pi_L^j\neq q$, the following diagram is commutative:
$$\xymatrix{
H^1_{\rm Iw}(L_\infty/L, T^{\otimes(-1)}(1)) \ar[d]_-{{\rm tw}_{j-1}} \ar[rr]^-{{\rm Exp}^\ast}&  & \bA_L^{\psi=1}\ar[dd]^-{{\rm ev}_{m,j}}   \\
H^1_{\rm Iw}(L_\infty/L, T^{\otimes(-j)}(1))  \ar[d]_-{{\rm pr}_m}& & \\
H^1(L_m, T^{\otimes(-j)}(1))  \ar[rr]_-{\exp^\ast_j}& &L_m.
}$$
%$$\xymatrix{
%H^1_{\rm Iw}(L_\infty/L, T^{\otimes(-1)}(1)) \ar[r]^{{\rm tw}_{j-1}}\ar[dd]_{{\rm Exp}^\ast} & H^1_{\rm Iw}(L_\infty/L, T^{\otimes(-j)}(1)) \ar[d]^{{\rm pr}_m} \\
% & H^1(L_m, T^{\otimes(-j)}(1)) \ar[d]^{\exp^\ast_j} \\
% \bA_L^{\psi=1} \ar[r]_{{\rm ev}_{m,j}} & L_m.
%}$$
\end{theorem}

\begin{remark}\label{rem:CC}
In the cyclotomic case (i.e., when $L=\QQ_p$, $\pi_L=p$, ${\rm LT}=\widehat \GG_m$, and $T=\ZZ_p(1)$), Theorem \ref{key claim} is proved by Cherbonnier-Colmez \cite[Thm. IV.2.1]{CC}. (Note that their definition of ${\rm Exp}^\ast$ is different from that of \cite{SV15}, but we can check that they coincide. See Appendix \ref{sec:bigexp}.) In fact, using their notation, Theorem \ref{key claim} is expressed by the formula
$$p^{-m} \varphi^{-m}({\rm Exp}^\ast(\mu)) = \sum_{j \in \ZZ} \exp^\ast_{L_m, \QQ_p(1-j)}\left(\int_{\Gamma_{L_m}} \chi_{\rm cyc}(x)^{1-j} \mu(x)\right).$$
Thus Theorem \ref{key claim} is a generalization of this result to the Lubin-Tate setting.
\end{remark}

Based on Theorem \ref{key claim} we can give a new proof of Kato's explicit reciprocity law (Theorem \ref{kato erl}) by using Theorem \ref{thm sv}  (under the assumption $\pi_L^j \neq q$, but see Remark \ref{rem:katoproof} below).

\begin{proof}[Proof of Theorem \ref{kato erl}]
Theorem \ref{key claim} implies that the map $\lambda_{m,j}$ coincides with the composition
$$\lambda_{m,j}: \varprojlim_n L_n^\times \xrightarrow{{\rm Kum}} H^1_{\rm Iw}(L_\infty/L,\ZZ_p(1))\xrightarrow{{\rm tw}_1} H^1_{\rm Iw}(L_\infty/L,T^{\otimes(-1)}(1)) \xrightarrow{{\rm Exp}^\ast} \bA_L^{\psi=1} \xrightarrow{{\rm ev}_{m,j}} L_m.$$
By Theorem \ref{thm sv} and the explicit description (\ref{evmj}) of the map ${\rm ev}_{m,j}$, we have for any $u \in \varprojlim_n L_n^\times$
$$\lambda_{m,j}(u) = {\rm ev}_{m,j}(\partial_{\rm inv} \log(g_{u,\eta}))  = \frac{1}{(j-1)!\pi_L^{mj}} \left. \left(\partial_{\mathrm{inv}}^{j}\log (g_{u,\eta})\right)\right|_{Z=\eta_m}=j\cdot \psi_{{\rm CW},m}^j(u) .$$
This is Kato's explicit reciprocity law.
\end{proof}
\begin{remark}\label{rem:katoproof}
Even when $\pi_L^j =q$, we can still give a proof of Kato's explicit reciprocity law by our method. See \S \ref{sec second} for details.
\end{remark}

In this article, we prove Theorem \ref{key claim} by using the ``Lubin-Tate regulator map" introduced by Schneider and the second named author in \cite{SV24}. One of the key ingredients is the ``reciprocity formula" established in \cite[Cor. 5.2.2]{SV24} (see Theorem \ref{cor:adjoint}). Combining this result with the interpolation property of Berger-Fourquaux's big exponential map (see \cite[Thm. 3.5.3]{BF}, which we recall in Theorem \ref{thm:BF}), we deduce the interpolation formula for the regulator map (see Theorem \ref{thm:adjointformulan}), which generalizes \cite[Thm.~A.2.3]{LVZ15} and \cite[Thm.\ B.5]{LZ}  from the cyclotomic case as well as \cite[Thm.\ 5.2.26]{SV24} in the Lubin-Tate case. Theorem \ref{thm:adjointformulan} is another main result of this article. Using this interpolation formula, we prove Theorem \ref{key claim} in \S \ref{sec:claim}.

In other words, we establish the general case of  Kato's explicit reciprocity law as a consequence of the reciprocity formula in \cite{SV24}. Although desirable we are not aware of any more direct proof of Theorem \ref{key claim}. (In the cyclotomic case, Theorem \ref{key claim} is proved directly without using regulator maps: see Theorem \ref{thm:cyc colmez}.)

%Hence the main task in this article will consist of providing a proof of Theorem \ref{key claim}. In \S \ref{sec:claim} we shall first reformulate it slightly into the commutativity of diagram \eqref{f:interpolation}. To verify the latter we shall apply the interpolation property of an attached regulator map as a consequence of  Theorem \ref{thm:adjointformulan} below.

As a byproduct we show a version of Colmez' reciprocity law in the Lubin-Tate setting (see Theorem \ref{thm:Colmez}), which generalizes the formula in Remark \ref{rem:CC}. Moreover, as another consequence we are able to extend the interpolation property of the regulator map to Artin characters in Corollary \ref{cor:adjointformula}.

\subsection{Notation}\label{sec:notation}

We shall summarize notations which will be used throughout this article.

Let $L/\QQ_p$ be a finite   extension of degree $d$, $o_L$ the ring of integers of $L$, $\pi_L \in o_L$ a fixed prime element, $k_L := o_L/\pi_L o_L$ the residue field with cardinality $q := \# k_L$,
% $e$ the absolute ramification index of $L$
and $G_L:=\Gal(\overline{L}/L).$ We always use the absolute value $|\cdot |_p$ on $\mathbb{C}_p$ which is normalized by $|\pi_L|_p = q^{-1}$.
%We warn the reader, though, that we will repeatedly use the references \cite{BSX}, \cite{FX}, \cite{La}, \cite{Sc1}, \cite{Sc}, \cite{ST}, and \cite{ST2} in which the absolute value is normalized differently from this paper by $|p| = p^{-1}$. Our absolute value is the $d$th power of the one in these references. The transcription of certain formulas to our convention will usually be done silently.

We fix a Lubin-Tate formal group ${\rm LT} = {\rm LT}_{\pi_L}$ over $o_L$ corresponding to the prime element $\pi_L$. We always identify ${\rm LT}$ with the open unit disk around zero, which gives us a global coordinate $Z$ on ${\rm LT}$. The $o_L$-action then is given by formal power series $[a](Z) \in o_L[[Z]]$ for $a\in o_L$. For simplicity the formal group law will be denoted by $+_{\LT}$. We set
$$g_{\rm LT}(Z):= \left(\left. \frac{\partial (X +_{\LT} Y)}{\partial Y}\right|_{(X,Y) = (Z,0)} \right)^{-1} \in o_L[[Z]]^\times.$$
We define $\log_{\rm LT}(Z) \in Z L[[Z]]$ to be the unique formal power series such that
$$\frac{d}{dZ} \log_{\rm LT}(Z)=g_{\rm LT}(Z).$$
The invariant derivative operator is defined by
\begin{equation}\label{def:partial}
\partial_{\rm inv}:=g_{\rm LT}(Z)^{-1}\frac{d}{dZ}.
\end{equation}
For basic properties, see \cite{SV15}.

%The power series $\frac{\partial (X +_{\LT} Y)}{\partial Y}_{|(X,Y) = (Z,0)}$ is a unit in $o_L[[Z]]$ and we let $g_{\LT}(Z)$ denote its inverse. Then $g_{\LT}(Z) dZ$ is, up to scalars, the unique invariant differential form on ${\rm LT}$ (\cite[\S5.8]{Haz}). We also let \begin{equation}\label{f:tLT}
%  \log_{\LT}(Z) = Z + \ldots
%\end{equation}
%denote the unique formal power series in $L[[Z]]$ whose formal derivative is $g_{\LT}$. This $\log_{\LT}$ is the logarithm of ${\rm LT}$ (\cite[8.6]{Lan}). In particular, $g_{\LT}dZ = d\log_{\LT}$. The invariant derivation $\partial_\inv$ corresponding to the form $d\log_{\LT}$ is determined by
%\begin{equation*}
%  f' dZ = df = \partial_\inv(f) d\log_{\LT} = \partial_\inv(f) g_{\LT} dZ
%\end{equation*}
%and hence is given by
%\begin{equation}\label{f:inv}
%  \partial_\inv(f) = g_{\LT}^{-1} f' \ .
%\end{equation}
%For any $a \in o_L$ we have
%\begin{equation}\label{f:dlog}
%  \log_{\LT} ([a](Z)) = a \cdot \log_{\LT} \qquad\text{and hence}\qquad ag_{\LT}(Z) = g_{\LT}([a](Z))\cdot [a]'(Z)
%\end{equation}
%(\cite[8.6 Lemma 2]{Lan}).

Let $\TLT$ be the Tate module of ${\rm LT}$. Then $\TLT$ is
a free $o_L$-module of rank one, say with generator $\eta=(\eta_n)_n$, and the action of
$G_L$ on $\TLT$ is given by a continuous character
$$\chi_{\LT} : G_L \longrightarrow o_L^\times.$$
Let $\TLT'$ denote the Tate module of the $p$-divisible group Cartier dual to ${\rm LT}$ with period $\Omega$ (depending on the choice of a generator of $\TLT'$, see \cite{ST2}), which again is a free $o_L$-module of rank one. The Galois action on $\TLT'\simeq \TLT^\ast(1):=\Hom_{o_L}(\TLT,o_L(1))$ is given by the continuous character
\begin{equation*}\label{def tau}
\tau := \chi_{\cyc}\cdot\chi_{\LT}^{-1},
\end{equation*}
where $\chi_{\cyc}$ is the cyclotomic character.

Let $K\subseteq \mathbb{C}_p$ be any complete subfield containing the period $\Omega$  and the maximal abelian extension $L^{\rm ab}$ of $L$.

For $n \geq 0$ we let $L_n/L$ denote the extension (in $\mathbb{C}_p$) generated by the $\pi_L^n$-torsion points of ${\rm LT}$, and set $G_n:=\Gal(L_n/L)$. Also, we set $L_\infty := \bigcup_n L_n$, $\Gamma_L := \Gal(L_\infty/L)$, and $H_L := \Gal(\overline{L}/L_\infty)$. The Lubin-Tate character $\chi_{\LT}$ induces an isomorphism $\Gamma_L \xrightarrow{\sim} o_L^\times$.

For the following standard facts we refer the reader to \cite{SV15,SV24}.
The ring endomorphisms induced by sending $Z$ to $[\pi_L](Z)$ are called $\varphi_L$ where applicable; e.g.\  for the ring $\bA_L$ defined to be the $\pi_L$-adic completion of $o_L((Z)):=o_L[[Z]][Z^{-1}]$ or $\bB_L := \bA_L[\pi_L^{-1}]$ which denotes the field of fractions of $\bA_L$.  Recall that we also have introduced the unique additive endomorphism $\psi_L$ of $\bB_L$ (and then  $\bA_L$) which satisfies
\begin{equation}\label{def psi}
  \varphi_L \circ \psi_L = \pi_L^{-1} \cdot \mathrm{Tr}_{\bB_L/\varphi_L(\bB_L)}  .
\end{equation}
Moreover,  projection formula
\begin{equation}\label{proj formula}
  \psi_L(\varphi_L(f_1)f_2) = f_1 \psi_L(f_2) \qquad\text{for any $f_1, f_2 \in \bB_L$}
\end{equation}
as well as the formula
\begin{equation*}
  \psi_L \circ \varphi_L = \frac{q}{\pi_L} \cdot \id \
\end{equation*}
hold. An  \'{e}tale $(\varphi_L,\Gamma_L)$-module $M$ comes with a Frobenius operator $\varphi_M$ semilinear with respect to $\varphi_L$ and an induced operator  denoted by $\psi_M$.

We set $t_{\rm LT}:=\log_{\rm LT}(Z) \in L[[Z]]$, so that
\begin{align*}
 \varphi_L(t_{\rm LT})=\pi_L\cdot t_{\rm LT}\ \ \ \text{ and }\ \ \ \gamma(t_{\rm LT})=\chi_{\rm LT}(\gamma)\cdot t_{\rm LT}
 \text{ for all }\gamma\in\Gamma_L.
\end{align*}

 Let $\cR_L$ be the Robba ring over $L$ and set
$$\cR_L^+:= \cR_L \cap L[[Z]],$$
which is the ring of power series with coefficients in $L$ that converge on the open unit disk.

Let $\Rep_{L}(G_L)$ denote the category of finitely dimensional $L$-vector spaces equipped with a continuous linear $G_L$-action.
For $V\in \Rep_L(G_L)$ we put
\[D_{\cris,L}(V) := (B_{\cris} \otimes_{L_0} V)^{G_L}\;\; \mbox{ and } \;\; D_{\dR,L}(V) := ( B_{\dR} \otimes_{\qp} V)^{G_{L}},\]
where $L_0$ denotes the maximal unramified extension of $\QQ_p$ inside $L$.
The second definition will be used also in the more general form
\[D_{\dR,L'}(V) := D_{\dR,L'}(V_{|G_{L'}})=( B_{\dR} \otimes_{\qp} V)^{G_{L'}}\]
for $L'$ any finite extension of $L$. Moreover, we abbreviate $D^0_{\dR,L'}(V):=\Fil^0(  D_{\dR,L'}(V)).$
Furthermore we write $D_{\dR,L'}^{\id}(V)$    for the direct summand $( B_{\dR} \otimes_L V)^{G_{L'}}$   of $D_{\dR,L'}(V).$

Recall that $T_\pi=o_L\eta$ denotes the Tate module of the Lubin-Tate formal group ${\rm LT}$. We set $V_\pi:=L\otimes_{o_L} T_\pi \in\Rep_L(G_L)$. We write
$t_{\mathbb{Q}_p} := \log_{\widehat{\mathbb{G}}_m}(\omega_{\widehat{\mathbb{G}}_m})$ for the multiplicative formal group $\widehat{\mathbb{G}}_m$. Then, for $j\geq 1,$
\begin{equation}\label{def dr}
\mathbf{d}_j := t_{\LT}^j t_{\mathbb{Q}_p}^{-1} \otimes (\eta^{\otimes -j} \otimes \eta^{\cyc}),
\end{equation}
where $\eta^{\cyc}$ is a generator of the cyclotomic Tate module $\mathbb{Z}_p(1)$, is an $L$-basis of  $D^0_{\dR,L}(V_\pi^{\otimes -j}(1))$.
We also set $e_j:=t_{\LT}^{-j}\otimes \eta^{\otimes j}\in D_{\cris,L}(L(\chi_{\LT}^j)).$

Let $\Rep_{o_L}(G_L)$ denote the category of free $o_L$-modules of finite rank  equipped with a continuous linear $G_L$-action.
The Iwasawa cohomology of $T\in \Rep_{o_L}(G_L)$ is defined by
\begin{equation*}
    H^i_{\Iw}(L_\infty/L,T) := \varprojlim_{L'} H^i(L',T)
\end{equation*}
where $L'$ runs through the finite Galois extensions of $L$ contained in $L_\infty$ and
the transition maps in the projective system are the cohomological corestriction maps.
For $V:=T\otimes_{o_L}L\in \Rep_L(G_L)$ we define
\begin{equation*}
    H^i_{\Iw}(L_\infty/L,V) := H^i_{\Iw}(L_\infty/L,T)\otimes_{o_L} L,
\end{equation*}
which is independent of the choice of $T$.

\subsection{A remark about signs regarding the dual Bloch-Kato exponential map}\label{rem sign}

Let
$$\exp_V: D_{\dR,L}(V) \to H^1(L,V)$$
be the Bloch-Kato exponential map. Using the local Tate duality isomorphism
\begin{equation}\label{local tate}
H^1(L,V) \simeq H^1(L,V^\ast(1))^\ast,
\end{equation}
the dual exponential map is usually defined by the composition map
$$\exp_V^\ast: H^1(L,V) \stackrel{(\ref{local tate})}{\simeq} H^1(L,V^\ast(1))^\ast \xrightarrow{\text{dual of $\exp_{V^\ast(1)}$}} D_{{\rm dR},L}(V^\ast(1))^\ast \simeq D_{{\rm dR},L}(V).$$
However, there is sign ambiguity in this definition, since there are two ways to define (\ref{local tate}): for $x\in H^1(L,V)$ and $y\in H^1(L,V^\ast(1))$,
\begin{itemize}
\item[(a)] $x \mapsto (y\mapsto {\rm inv}(x\cup y))$;
\item[(b)] $x \mapsto (y\mapsto {\rm inv}(y \cup x))$.
\end{itemize}
These definitions differ by sign since $x\cup y = -y\cup x$ (see \cite[Prop. 1.4.4]{NSW}).
%\begin{remark}\label{rem sign}
%We shall explain why a sign appears in \cite[Cor. 8.7]{SV15} whereas not in (\ref{specialkato}).
%Note that there is sign ambiguity in the local Tate duality isomorphism:
%
%In fact, letting ${\rm inv}: H^2(L,L(1))\xrightarrow{\sim} L$ be the invariant map in local class field theory, there are two ways to define (\ref{local tate}): for $x\in H^1(L,V)$ and $y\in H^1(L,V^\ast(1))$,
%\begin{itemize}
%\item[(a)] $x \mapsto (y\mapsto {\rm inv}(x\cup y))$;
%\item[(b)] $x \mapsto (y\mapsto {\rm inv}(y \cup x))$.
%\end{itemize}
%These definitions differ by sign since $x\cup y = -y\cup x$ (see \cite[Prop. 1.4.4]{NSW}). So the definition of the dual exponential map using (\ref{local tate})
%$$\exp_V^\ast: H^1(L,V) \stackrel{(\ref{local tate})}{\simeq} H^1(L,V^\ast(1))^\ast \xrightarrow{\text{dual of $\exp_{V^\ast(1)}$}} D_{{\rm dR},L}(V^\ast(1))^\ast \simeq D_{{\rm dR},L}(V)$$
%has sign ambiguity.

On the other hand, Kato's definition of the dual exponential map (see \cite[\S II.1.2.4]{katolecture}) has no sign ambiguity: it is defined by the composite
$$\exp^\ast_{V}: H^1(L,V)\to H^1(L, V\otimes_{\QQ_p} B_{\rm dR}^+) \simeq D_{{\rm dR},L}^0(V),$$
where the last isomorphism is the inverse of
$$D_{{\rm dR},L}^0(V)=H^0(L,V\otimes_{\QQ_p} B_{\rm dR}^+) \xrightarrow{\sim} H^1(L,V\otimes_{\QQ_p} B_{\rm dR}^+); \ d \mapsto d\cup \log \chi_{\rm cyc}. $$
(Note that we regard $\log \chi_{\rm cyc}\in H^1(L,\ZZ_p)$.) {\it In this article, we use this definition of the dual exponential map.}

Using the normalization of the reciprocity map as in \cite{serre}, we claim that our dual exponential map coincides with the definition using (b), i.e., we have the commutative diagram
$$
\xymatrix{
H^1(L,V) & \times & H^1(L,V^\ast(1)) \ar[d]^{\exp^\ast_{V^\ast(1)}}  \ar[r]^\cup& H^2(L,L(1)) \ar[r]^-{\rm inv}& L \ar@{=}[d]\\
D_{{\rm dR},L}(V)\ar[u]_{\exp_V} &\times & D_{{\rm dR},L}(V^\ast(1)) \ar[rr]& & L.
}$$
In fact, according to \cite[Cor. 7.2.13]{NSW} and the fact that the composite
$$\ZZ_p^\times \xrightarrow{{\rm rec}} G_{\QQ_p}^{\rm ab} \xrightarrow{\chi_{\rm cyc}} \ZZ_p^\times$$
is given by $u\mapsto u^{-1}$ (see \cite[Rem. in \S XIV.7]{serre}), we have
\begin{equation}\label{kummer normalize}
{\rm inv}(\kappa(a) \cup \log \chi_{\rm cyc}) = \log(a)
\end{equation}
for any $a \in \QQ_p^\times$, where $\kappa: \QQ_p^\times \to H^1(\QQ_p,\QQ_p(1))$ denotes the Kummer map. Our claim then follows from the argument of \cite[Thm. II.1.4.1(4)]{katolecture} if we note that the sign of \cite[Lem. II.1.4.5]{katolecture} is different from (\ref{kummer normalize}).

%
%In \cite[Thm. II.1.4.1(4)]{katolecture}, it is claimed that Kato's definition of $\exp^\ast_V$ coincides with the definition using (a). However, it seems to the authors that its proof has a sign error: according to \cite[Cor. 7.2.13]{NSW} and the fact that the composite
%$$\ZZ_p^\times \xrightarrow{{\rm rec}} G_{\QQ_p}^{\rm ab} \xrightarrow{\chi_{\rm cyc}} \ZZ_p^\times$$
%is given by $u\mapsto u^{-1}$ (see \cite[Rem. in \S XIV.7]{serre}), \cite[Lem. II.1.4.5]{katolecture} is not correct. Thus we conclude that {\it our $\exp_V^\ast$ coincides with the definition using (b)}.

%\end{remark}

\section{Lubin-Tate regulator maps}

The aim of this section is to establish the interpolation property of the Lubin-Tate regulator map (see Theorem \ref{thm:adjointformulan}, Corollaries \ref{cor:adjointformula} and \ref{f:intchi-j}).

\subsection{Definition of regulator maps}\label{sec:reg}

We recall regulator maps defined in \cite[\S 5.1]{SV24}.  We use the same notation in loc. cit. In particular, we write
$$\mathfrak{M}: D(\Gamma_L,\Cp)\xrightarrow{\sim} {\mathcal{O}_{\Cp}(\mathbf{B})}^{\psi_L=0}; \ \lambda\mapsto \lambda\cdot \eta(1,Z)$$
for the Mellin transform, where we write $D(\Gamma_L, \Cp)$ for the locally analytic distribution algebra on $\Gamma_L$, $\cO_{\CC_p}(\mathbf{B})$ for the ring of rigid analytic functions on the open unit disk $\mathbf{B}$ over $\Cp$, and we set
$$\eta(1,Z):=\exp(\Omega {\rm log}_{\rm LT}(Z))\in \mathcal{O}_{\Cp}(\mathbf{B}).$$
The ring of rigid analytic functions on the open unit disk $\mathbf{B}$ over $L$ is simply denoted by $\cO$. For a subfield $L'\subseteq \CC_p$ we often abbreviate $\cO_{L'}(\mathbf{B})$ to $\cO_{L'}$. Also, we write $\Rep_{o_L,f}^{\cris}(G_L)$ and $\Rep_{o_L,f}^{\cris,\mathrm{an}}(G_L)$ for the full subcategories of $\Rep_{o_L,f}(G_L)$, the    category of finitely generated free $o_L$-modules  with a continuous linear $G_L$-action, consisting    of those $T$ which are free over $o_L$ and such that the representation $V := L \otimes_{o_L} T$ is crystalline and in addition   analytic, respectively. Here, $V$ is called {  analytic},   if   the filtration on $D_{\dR}(V)_\mathfrak{m}$ is trivial for each maximal ideal $\mathfrak{m}$ of $L \otimes_{\mathbb{Q}_p} L$ which does not correspond to the identity $\id : L \to L$. Moreover, $D_{\LT}(T)$ denotes the Lubin-Tate $(\varphi_L,\Gamma_L)$-module attached to any $T\in \Rep_{o_L,f}^{\cris}(G_L)$ (see \cite[Thm.\ 1.6]{KR}, also recalled in \cite[Thm.\ 4.2]{SV15}), while   $N(T)$ denotes the attached Wach module \`{a} la Kisin-Ren  for $T$ in $\Rep_{o_L,f}^{\cris,\mathrm{an}}(G_L)$, see \cite[Ch.\ 3, especially before Rem.\ 3.1.6]{SV24}.

Recall that we set $\tau := \chi_{\cyc}\cdot\chi_{\LT}^{-1}$.
 Let   $T$ be in $\Rep_{o_L,f}^{\cris}(G_L)$    such that $T(\tau^{-1})$ belongs to  $\Rep_{o_L,f}^{\cris,\mathrm{an}}(G_L)$ with all   Hodge-Tate weights   in $[0,r],$ and such that $V:=L\otimes_{o_L} T$ does not have any quotient isomorphic to $L(\tau).$ Under these assumptions Schneider and Venjakob \cite[\S 5.1]{SV24} define the regulator maps
\begin{align*}
\mathbf{L}_V: & H^1_{\Iw}(L_\infty/L,T)\to  D(\Gamma_L,\Cp)\otimes_L D_{\cris,L}(V(\tau^{-1})), \\
  %\mathcal{L}_V^0: & H^1_{\Iw}(L_\infty/L,T)\to   \mathcal{O}_L(\mathbf{B})^{\psi_L=0}\otimes_L D_{\cris,L}(V(\tau^{-1})), \\
  \mathcal{L}_V: & H^1_{\Iw}(L_\infty/L,T)\to D(\Gamma_L,\Cp)\otimes_L D_{\cris,L}(V)
\end{align*}
as (part of) the composite \small
\begin{align}\label{f:defregulator}
 \notag H^1_{\Iw}(L_\infty{/L},T)&\stackrel{{\rm Exp}^\ast}{\simeq} D_{\LT}(T(\tau^{-1}))^{\psi_L=1}=N(T(\tau^{-1}))^{\psi_{D_{\LT}(T(\tau^{-1}))}=1} \xrightarrow{(1-\frac{\pi_L}{q}\varphi_L)}\varphi^*_L(N(V(\tau^{-1})))^{\psi_L=0}\\
  &\hookrightarrow \mathcal{O}^{\psi_L=0} \otimes_L D_{\cris,L}(V(\tau^{-1})) \subseteq {\mathcal{O}_{\Cp}(\mathbf{B})}^{\psi_L=0}\otimes_L D_{\cris,L}(V(\tau^{-1}))\\
\notag &\xrightarrow{\mathfrak{M}^{-1}\otimes \id}D(\Gamma_L,\Cp)\otimes_L D_{\cris,L}(V(\tau^{-1}))\to D(\Gamma_L,\Cp)\otimes_L D_{\cris,L}(V),
\end{align}\normalsize
 where the last map sends $\mu\otimes d \in D(\Gamma_L,\Cp)\otimes_L D_{\cris,L}(V(\tau^{-1}))$ to $\mu\otimes d \otimes \mathbf{d}_1\in D(\Gamma_L,\Cp)\otimes_L D_{\cris,L}(V(\tau^{-1}))\otimes_L D_{\cris,L}(L(\tau))\simeq D(\Gamma_L,\Cp)\otimes_L D_{\cris,L}(V). $ Note that
 $$D:=D_{\cris,L}(L(\tau))=D^0_{\dR,L}(L(\tau)) = L \mathbf{d}_1$$
with $\mathbf{d}_1= t_{\LT} t_{\mathbb{Q}_p}^{-1} \otimes (\eta^{\otimes -1} \otimes \eta^{\cyc})$, where $L(\chi_{\LT})=L\eta$ and $L(\chi_{\cyc})=L\eta^{\cyc}.$

Alternatively, in order to stress that the regulator is essentially the map $1- \varphi_L, $ one can rewrite this as \small
 \begin{equation*}
  \begin{aligned}
  	H^1_{\Iw}&(L_\infty/L,T) \stackrel{{\rm Exp}^\ast}{\simeq} D_{\LT}(V(\tau^{-1}))^{\psi_L=1}=N(T(\tau^{-1}))^{\psi_{D_{\LT}(T(\tau^{-1}))}=1} \\
  &\hookrightarrow \Big(N(V(\tau^{-1}))\otimes_L D\Big) ^{\Psi =1}\xrightarrow{1- \varphi_L}\varphi^*_L(N(V(\tau^{-1})))^{\Psi_L=0}\otimes_L D \\
  &\hookrightarrow \mathcal{O}^{\Psi_L=0} \otimes_L D_{\cris,L}(V(\tau^{-1}))\otimes_L D \subseteq {\mathcal{O}_{\Cp}(\mathbf{B})}^{\Psi_L=0}\otimes_L D_{\cris,L}(V)\\
 &\xrightarrow{\mathfrak{M}^{-1}\otimes \id}D(\Gamma_L,\Cp)\otimes_L D_{\cris,L}(V)
  \end{aligned}
\end{equation*}\normalsize
 where the $ \hookrightarrow $ in the second line sends $n$ to $n\otimes \mathbf{d}_1$ and the $\varphi_L$ as well as $\Psi$ now acts diagonally.

 As mentioned after \cite[Prop.\ 5.1.1]{SV24} one can extend the definition of the regulator map to the case of all  $T$   in $\Rep_{o_L,f}^{\cris}(G_L)$    such that $T(\tau^{-1})$ belongs to  $\Rep_{o_L,f}^{\cris,\mathrm{an}}(G_L)$, if one is willing to accept denominators in the target, e.g. by replacing $D(\Gamma_L,\mathbb{C}_p)$   by its  total ring of quotients. For example, we obtain for $T=o_L(\chi_\LT^{-j})(1)$ a regulator map
 \[\mathcal{L}_{L(\chi_\LT^{-j})(1)}:  H^1_{\Iw}(L_\infty/L,T)\to \frac{1}{\mathfrak{l}'_{L(\chi_\LT^{j+1})}}D(\Gamma_L,\Cp)\otimes_L D_{\cris,L}(L(\chi_\LT^{-j})(1)),\]
where $\mathfrak{l}'_{L(\chi_\LT^{j+1})}\in D(\Gamma_L,\Cp)$ is defined in Appendix \ref{sec:Lie}. Inspection of the construction shows that we have a commutative diagram

\footnotesize
\begin{equation}\label{f:regulatortwist}
 \xymatrix@C=0.1cm{
\mathcal{L}_{L(\chi_{\rm LT})(1)}: H^1_{\Iw}(L_\infty{/L},o_L(\chi_\LT(1)))   \ar[rr]^-{\mathrm{Exp}^*} && \mathbf{A}_L^{\psi=1}(\chi_{\rm LT}^{2})    \ar[rr]^-{\Xi_{L(\chi_{\rm LT}^{2})}}   & &  D(\Gamma_L,K)\otimes_L D_{\mathrm{cris},L}(L(\chi_{\rm LT} )(1)) \\
  \mathcal{L}_{L(\chi_\LT^{-j})(1)}:  H^1_{\Iw}(L_\infty{/L},\TLT^{\otimes -j}(1))\ar[u]_{\otimes \eta^{\otimes j+1}}  \ar[rr]^-{\mathrm{Exp}^*} && \mathbf{A}_L^{\psi=1}(\chi_{\rm LT}^{1-j}) \ar[u]_{\otimes \eta^{\otimes j+1}}  \ar[rr]^-{\Xi_{L(\chi_{\rm LT}^{1-j})}}   & & \frac{1}{\frak{l}'_{L(\chi_\LT^{j+1})}}D(\Gamma_L,K)\otimes_L D_{\mathrm{cris},L}(L(\chi_{\rm LT}^{-j})(1)) \ar[u]_-{\frac{\frak{l}_{L(\chi_\LT^{j+1})}\mathrm{Tw}_{\chi_\LT^{-(j+1)}}}{\Omega^{j+1}}\otimes \id \otimes e_{j+1}},    }
\end{equation}\normalsize
in which - according to \cite[Figure 5.1]{SV24} - the map $\Xi_{L(\chi_{\rm LT}^{2})}$ is given as the composite
\begin{align*}
 \mathbf{A}_L^{\psi=1}(\chi_{\rm LT}^{2})  \simeq &(Z^{-2}o_L[[Z]])^{\psi_L=1}\otimes_{o_L} o_L \eta^{\otimes 2}\xrightarrow{(1-\frac{\pi_L}{q}\varphi_L)\otimes \id } \varphi_L(Z)^{-2}o_L[[Z]]^{\psi_L=0}\otimes_{o_L} o_L\eta^{\otimes 2}\\
 &\xrightarrow{t_\LT^2\otimes t_\LT^{-2}} (\cR_K^+)^{\psi_L=0}\otimes_L D_{\cris,L}(L(\chi_\LT^2))\xrightarrow{\mathfrak{M}^{-1}\otimes \id \otimes \mathbf{d}_1} D(\Gamma_L,K)\otimes_L D_{\cris,L}(L(\chi_\LT)(1))
\end{align*}
or equivalently as
\begin{align}\label{f:Xichi2}\footnotesize
 \mathbf{A}_L^{\psi=1}(\chi_{\rm LT}^{2})&\xrightarrow{x\mapsto x\otimes \mathbf{d}_1}   \bigg(\frac{t_\LT^2}{Z^2}o_L[[Z]]\otimes_{o_L}D_{\cris,L}(L(\chi_\LT)(1)) \bigg)^{\Psi_L=1} \\
 &  \xrightarrow{1- \varphi_L\otimes \varphi_L }  (\cR_K^+)^{\psi_L=0}\otimes_L D_{\cris,L}(L(\chi_\LT)(1)) \xrightarrow{\mathfrak{M}^{-1}\otimes \id } D(\Gamma_L,K)\otimes_L D_{\cris,L}(L(\chi_\LT)(1)).\notag
\end{align}\normalsize

\begin{remark}\label{rem:Xi}
 $\Xi_{L(\chi_{\rm LT}^{1-j})}$ can be described as composite of
\begin{align}\label{f:Achid1}
  \mathbf{A}_L^{\psi=1}(\chi_{\rm LT}^{1-j})\xrightarrow{x\mapsto x\otimes \mathbf{d}_1} \mathbf{A}_L^{\psi=1}&(\chi_{\rm LT}^{1-j})\otimes_{o_L} D\subseteq \bigg(\frac{1}{t^{j+1}}\mathcal{O}_K\otimes D_{\cris,L}(L(\chi_\LT^{-j})(1))\bigg)^{\Psi_L=1},
\end{align}
\begin{align}\label{f:XiexplicitGamma}
  \bigg(\frac{1}{t^{j+1}}\mathcal{O}_K\otimes D_{\cris,L}(L(\chi_\LT^{-j})(1))\bigg)^{\Psi_L=1}  & \xrightarrow{\partial_\mathrm{inv}^{j+1}t_\LT^{j+1}\otimes\id} \bigg(\mathcal{O}_K\otimes D_{\cris,L}(L(\chi_\LT^{-j})(1))\bigg)^{\Psi_L=1} ,
\end{align}\footnotesize
\begin{align}\label{f:Xiexplicit}
  \bigg(\mathcal{O}_K\otimes D_{\cris,L}(L(\chi_\LT^{-j})(1))\bigg)^{\Psi_L=1}  & \xrightarrow{1-\varphi_L\otimes \varphi_L} ({\mathcal{O}_K)^{\psi_L=0}\otimes D_{\cris,L}(L(\chi_\LT^{-j})(1)) }\xrightarrow{\mathfrak{M}^{-1}\otimes \id} D(\Gamma_L,K)\otimes_L D_{\cris,L}(L(\chi_\LT^{-j})(1))
\end{align}\normalsize
and the inverse of
\begin{align}\label{f:Gammainverse}
 \frac{1}{\mathfrak{l}'_{L(\chi_\LT^{j+1})}} D(\Gamma_L,K)\otimes_L D_{\cris,L}(L(\chi_\LT^{-j})(1)) & \xrightarrow{\cdot \mathfrak{l}'_{L(\chi_\LT^{j+1})}} D(\Gamma_L,K)\otimes_L D_{\cris,L}(L(\chi_\LT^{-j})(1)).
\end{align}
\end{remark}
\begin{proof}
  Note that we have a commutative diagram
\[\xymatrix{
  \mathbf{A}_L^{\psi=1}(\chi_{\rm LT}^{1-j}) \ar[d]_-{\otimes \eta^{\otimes j+1}} \ar[r]^-{\otimes \mathbf{d}_1 } &  \bigg(\frac{1}{t^{j+1}}\mathcal{O}_K\otimes D_{\cris,L}(L(\chi_\LT^{-j})(1))\bigg)^{\Psi_L=1} \ar[d]^{t^{j+1}\otimes e_{j+1}} \\
  \mathbf{A}_L^{\psi=1}(\chi_{\rm LT}^{2}) \ar[r]^-{\otimes \mathbf{d}_1} & \bigg(\frac{t_\LT^2}{Z^2}o_L[[Z]]\otimes_{o_L}D_{\cris,L}(L(\chi_\LT)(1)) \bigg)^{\Psi_L=1}.   }\]
Now use the fact from  \cite[\S 5.1]{SV24} that under the Mellin transform $\partial_\mathrm{inv}^{j+1}$ corresponds to $\Omega^{j+1}\mathrm{Tw}_{\chi_\LT^{j+1}}$ while $\partial_\mathrm{inv}^{j+1}t_\LT^{j+1}$ to $\mathfrak{l}'_{L(\chi_\LT^{j+1})}$ in order to transform \eqref{f:Xichi2} into the claimed composite using the recipe from \eqref{f:regulatortwist}.
\end{proof}

\subsection{The reciprocity formula}\label{sec:Berger-Four}

 Let  $W$ denote an $L$-analytic crystalline representation of $G_L$ (recall from \cite{BF}\footnote{Laurent Berger confirmed to us that the condition on the extension $F/\mathbb{Q}_p$, i.e., $L/\mathbb{Q}_p$ in this paper, being {\it Galois}  is really used nowhere in that article, see also \url{https://perso.ens-lyon.fr/laurent.berger/articles/errata.pdf}. } that an $L$-linear representation $V$ of $G_L$ is said to be $L$-analytic if $\CC_p\otimes_{L,\sigma}V$ is trivial as a semilinear $\CC_p$-representation of  $G_{\widetilde L}$ for any embedding $\sigma: L\hookrightarrow \CC_p$ with $\sigma \neq \id$, where $\widetilde L$ denotes the Galois closure of $L/\QQ_p$). Take an integer $h\geq 1$   such that $\mathrm{Fil}^{-h}D_{\cris,L}(W)=D_{\cris,L}(W)$ and such that $D_{\cris,L}(W)^{\varphi_L=\pi_L^{-h}}=0$ holds. Under these conditions in \cite{BF} a big exponential map \`{a} la Perrin-Riou
\begin{align*}
\Omega_{W,h}:\left(\mathcal{O}^{\psi_L=0}\otimes_L D_{\cris,L}(W)\right)^{\Delta=0}\to D_{\mathrm{rig}}^\dagger(W)^{\psi_L=\frac{q}{\pi_L}}
\end{align*}
is constructed as follows: According to \cite[Lem.\ 3.5.1]{BF}  there is an exact sequence
\begin{align*}
0\to \bigoplus_{k=0}^h t_{\LT}^kD_{\cris,L}(W)&^{\varphi_L=\pi_L^{-k}}\to \left( \mathcal{O}\otimes_{o_L} D_{\cris,L}(W)\right)^{\psi_L=\frac{q}{\pi_L}}\xrightarrow{1-\varphi_L} \\ & \mathcal{O}^{\psi_L=0}\otimes_{L} D_{\cris,L}(W)\xrightarrow{\Delta}\bigoplus_{k=0}^hD_{\cris,L}(W)/(1-\pi_L^k\varphi_L)D_{\cris,L}(W)\to 0,
\end{align*}
where, for $f\in \mathcal{O}\otimes_L D_{\cris,L}(W)$,  $\Delta(f)$ denotes the image of $\bigoplus_{k=0}^h(\partial_\inv^k\otimes \id_{D_{\cris,L}(W)})(f)(0)$ in $\bigoplus_{k=0}^hD_{\cris,L}(W)/(1-\pi_L^k\varphi_L)D_{\cris,L}(W).$ Hence, if $f\in \left(\mathcal{O}^{\psi_L=0}\otimes_L D_{\cris,L}(W)\right)^{\Delta=0}$ there exists $y\in  \left( \mathcal{O}\otimes_{o_L} D_{\cris,L}(W)\right)^{\psi_L=\frac{q}{\pi_L}}$ such that $f=(1-\varphi_L)y.$ Setting $\nabla_i:=\nabla-i$ for any integer $i$, one observes that $\nabla_{h-1}\circ \ldots \circ \nabla_0$ annihilates $ \bigoplus_{k=0}^{h-1} t_{\LT}^kD_{\cris,L}(W)^{\varphi_L=\pi_L^{-k}}$ whence
$$\Omega_{W,h}(f):=\nabla_{h-1}\circ \ldots \circ \nabla_0(y)$$
is well-defined and belongs under the comparison isomorphism  to $D_{\mathrm{rig}}^\dagger(W)^{\psi_L=\frac{q}{\pi_L}}$.

Note that $\left(\mathcal{O}^{\psi_L=0}\otimes_L D_{\cris,L}(W)\right)^{\Delta=0}=\mathcal{O}^{\psi_L=0}\otimes_L D_{\cris,L}(W) $ if $D_{\cris,L}(W)^{\varphi_L=\pi_L^{-k}}=0$ for all $0\leq k\leq h.$ If this does not hold for $W$ itself, it does hold for $W(\chi_{\LT}^{-r})$ for $r$ sufficiently large (with respect to the same $h$).

We also could consider the following variant of the big exponential map (under the assumptions of the theorem below, in particular with $W=V^*(1)$)
\begin{align*}
\mathbf{\Omega}_{V^*(1),h}: D(\Gamma_L,\mathbb{C}_p)\otimes_L D_{\cris,L}(V^*(1))\to D_{\mathrm{rig}}^\dagger(V)^{\psi_L=\frac{q}{\pi_L}}
\end{align*}
by extending scalars from $L$ to $\mathbb{C}_p$ and composing the original one with $\Omega^{-h}$
%\footnote{This means to replace $\nabla$ by $\frac{\nabla}{\Omega}$ in order to achieve twist invariance of the big exponential map, see Remark \ref{rem:BigExpTwist} below.}
times
\[
 D(\Gamma_L,\mathbb{C}_p)\otimes_L D_{\cris,L}(V^*(1))\xrightarrow{ \mathfrak{M}\otimes \id}({\mathcal{O}_{\CC_p}(\mathbf{B})})^{\psi_L=0}\otimes_L D_{\cris,L}(V^*(1)).\]
Let $\upiota_\ast$ be the involution on $D(\Gamma_L, \CC_p)$ induced by $\Gamma_L \to \Gamma_L; \ g\mapsto g^{-1}$.
\begin{theorem}[Reciprocity formula, {\cite[Cor.\ 5.2.2]{SV24}}] \label{cor:adjoint}
Assume that $V^*(1)$ is $L$-analytic with $\mathrm{Fil}^{-1}D_{\cris,L}(V^*(1))=D_{\cris,L}(V^*(1))$ and   $D_{\cris,L}(V^*(1))^{\varphi_L=\pi_L^{-1}}=D_{\cris,L}(V^*(1))^{\varphi_L=1}=0$. Then the following diagram  of $D(\Gamma_L,K)$-$\upiota_*$-sesquilinear pairings  commutes:\Footnote{According to \cite[Prop.\ A.2.2]{LVZ15} in the cyclotomic case we have the relation
\[<\mathcal{L}_V(x),\omega>_{\Iw,cris}=\gamma_{-1}<x,\Omega_{V^*(1),1,\xi}(\omega)>_{\Iw}\]
for $x\in H^1_{\Iw}(\qp.V)$ and $\omega\in D(\Gamma,\qp)\otimes D_{\cris}(V^*(1))$ (partly in the notation of that article). Note that the factor $\gamma_{-1}$ is already contained in the definition \eqref{f:def[]0} of the pairing $[,]_{\cris}.$ It is derived from the equation (12) in (loc.\ cit.)
\[\Omega_{V,1,\xi}(\mathcal{L}_V(x))=\nabla(x)  \mbox{ modulo torsion}\]
for all $x\in D(\Gamma,\qp)\otimes_{\Lambda(\Gamma)} H^1_{\Iw}(\qp,V)$ and the reciprocity formula from Thm.\ 3.3.7 in (loc.\ cit.)
\[<\mathcal{L}_V(x),\mathcal{L}_{V^*(1)}(y)>_{\Iw,cris}=\gamma_{-1}\nabla<x,y>_{\Iw},\]
(In (loc.\ cit.) there appears a sign on the right hand side, which  I do not understand and which according to David Loeffler is a typo!) {We assume that the HT weights of $V$ are in $[0,r].$ Then $\mathcal{L}_{V^*(1)} $ is defined as follows:  choose $h$ sufficiently large that $V^*(1+h)$ has HT weights $\geq0$ and set \[\mathcal{L}_{V^*(1)}(y)=(\nabla_{-1}\cdots \nabla_{-h})^{-1}\mathrm{Tw}_{\chi_{\cyc}^{-h}}(\mathcal{L}_{V^*(1+h)}(\mathrm{tw}_{\chi_{\cyc}^{-h}}(y)))\otimes t^he_{-h}\in FracD(\Gamma,\qp)\otimes D_{\cris}.\] }}
\begin{equation}\label{f:adjointcor}
\xymatrix{
    {D_{\mathrm{rig}}^\dagger(V^*(1))^{\psi_L=\frac{q}{\pi_L}}}\ar@{}[r]|{\times} &  {D(V(\tau^{-1}))^{\psi_L=1}}\ar[d]_{\frac{\sigma_{-1}\mathbf{L}_V}{{{\Omega}}} } \ar[r]^-{\frac{q-1}{q}\{,\}_{\Iw}} & D(\Gamma_L,\Cp) \ar@{=}[d] \\
    D(\Gamma_L,\mathbb{C}_p)\otimes_L D_{\cris,L}(V^*(1)) \ar[u]_{\mathbf{\Omega}_{V^*(1),1}}   \ar@{}[r]|{\times} & D(\Gamma_L,\mathbb{C}_p)\otimes_L D_{\cris,L}(V(\tau^{-1})) \ar[r]^-{[,]^0} & D(\Gamma_L,\Cp),}
\end{equation}
where $[-,-]^0=[ \mathfrak{M}\otimes \id (-)  , \sigma_{-1}\mathfrak{M}\otimes \id (-)   ]   ,$ i.e.,
\begin{equation}\label{f:def[]0}
[\lambda\otimes\check{d},\mu \otimes d]^0\cdot \eta(1,Z)\otimes (t_{\LT}^{-1}\otimes \eta)=\lambda\upiota_*(\mu)\cdot  \eta(1,Z)\otimes[\check{d}, d]_{\cris},
\end{equation}  where $D_{\cris,L}(V^*(1))\times D_{\cris,L}(V(\tau^{-1}))\xrightarrow{[\;,\;]_{\cris}} D_{\cris,L}(L(\chi_{\LT}))$ is the canonical pairing.
\end{theorem}
We denote by
\begin{equation}\label{def cris prime}
[\;,\;]_{\cris}' :D_{\cris,L}(V^*(1))\times D_{\cris,L}(V(\tau^{-1}))\to L
\end{equation}
the pairing such that $[\;,\;]_{\cris}=[\;,\;]_{\cris}'\cdot(t_{\LT}^{-1}\otimes \eta).$
\begin{remark}\label{rem:BigExpTwist}
By \cite[Cor.\ 3.5.4]{BF} we have $\Omega_{V,h}(x)\otimes \eta^{\otimes j}=\Omega_{V(\chi_{\LT}^j), h+j}(\partial_\inv^{-j}x\otimes t_{\LT}^{-j}\eta^{\otimes j})$ and $\mathfrak{l}_h\circ \Omega_{V,h}=\Omega_{V,h+1},$ whence we obtain $\mathbf{\Omega}_{V,h}(x)\otimes \eta^{\otimes j}=\mathbf{\Omega}_{V(\chi_{\LT}^j),h+j}(\mathrm{Tw}_{\chi_{\LT}^{-j}}(x)\otimes t_{\LT}^{-j}\eta^{\otimes j})$ and $\mathfrak{l}_h\circ \mathbf{\Omega}_{V,h}=\mathbf{\Omega}_{V,h+1}.$
\end{remark}

\subsection{Evaluation maps}\label{sec:evaluation}

Let $W$ be an $L$-analytic, crystalline  $L$-linear representation of $G_L$.
For $n\geq 1$, we write
$$\mathrm{Ev}_{W,n}:\mathcal{O}_L \otimes_L D_{\cris,L}(W)\to L_n\otimes_L D_{\cris,L}(W)$$
for the composite $\partial_{D_{\cris,L}(W)}\circ \varphi_q^{-n}$ from the introduction of \cite{BF}, which actually sends $f(Z)\otimes d$ to $f(\eta_n)\otimes \varphi_L^{-n}(d).$  For $n=0$, define ${\rm Ev}_{W,0}$ by
$${\rm Ev}_{W,0}(f(Z)\otimes d):= f(0) \otimes (1-q^{-1}\varphi_L^{-1})d.$$
Then by \cite[Lem.\  2.4.3]{BF} we have
\begin{equation}\label{f:Bergerrel}
   \mathrm{Tr}_{L_n/L_m}(q^{-n}\mathrm{Ev}_{W,n}(x) )=q^{-m}\mathrm{Ev}_{W,m}(x)
\end{equation}
for any $n\geq m\geq 0$.

We extend it to a map
\begin{align*}
  \mathrm{Ev}_{W,n}:\mathcal{O}_K \otimes_L D_{\cris,L}(W) & \to K\otimes_L L_n\otimes_L D_{\cris,L}(W), \\
  f(Z)\otimes d & \mapsto (f(\sigma_a\eta_n)\otimes \varphi_L^{-n}(d))_a,
\end{align*}
using the identification $K\otimes_L L_n=\prod_{(o_L/\pi_L^n)^\times}K$ mapping $k\otimes_L l$ to $(\sigma_a(l)\cdot k)_{a\in  (o_L/\pi_L^n)^\times }$; we have the maps
\[\mathrm{Tr}_{K_n/K}:\prod_{(o_L/\pi_L^n)^\times}K\to K,\; (l_a)_{a\in  (o_L/\pi_L^n)^\times }\mapsto \sum_{a\in  (o_L/\pi_L^n)^\times } l_a.\]
%By abuse of notation we  also use $\mathrm{Ev}_{W,0}$ for the analogous map   $ \mathcal{O}_K \otimes_L D_{\cris,L}(W)\to K\otimes_L D_{\cris,L}(W)$.
%From the definitions we obtain
%\begin{equation}\label{f:Ev}
%  \mathrm{Ev}_{W,k}\bigg((1-\varphi_L\otimes \varphi_L)(x)\bigg)=\left\{
%                                                                   \begin{array}{ll}
%                                                                     \mathrm{Ev}_{W,k}(x)-\mathrm{Ev}_{W,k-1}(x), & \hbox{if $k\geq 1$;} \\
%                                                                    (1-\varphi_L)\mathrm{Ev}_{W,0}(x), & \hbox{if $k=0$,}
%                                                                   \end{array}
%                                                                 \right.
%\end{equation}
%{\color{red} should this not change due to the new definition of $Ev_0$?!This is needed in \eqref{f:phi}below.}

For $\rho$  running to the characters of $G_n:=\Gal(L_n/L)$ with values in  $L^{\rm ab}\subseteq K$, we denote by
$$\mathfrak{e}_\rho:=\frac{1}{[L_n:L]}\sum_{g\in G_n} \rho(g^{-1})g\in K[G_n]$$
the idempotent attached to  $\rho$ satisfying $g\mathfrak{e}_\rho=\rho(g)\mathfrak{e}_\rho$ for all $g\in G_n$.
For the purpose of this article we say that a character
$\rho: G_n\to K^\times$ has {\em conductor} $0\leq a(\rho)\leq n,$ if it factorizes over $G_{a(\rho)}$, but not over $G_{a(\rho)-1}.$ Furthermore, $\tau(\rho)$ denotes the Gau{\ss}  sum
$$\tau(\rho):= [L_{a(\rho)}:L]\mathfrak{e}_\rho \eta(1,\eta_{a(\rho)})=\sum_{g\in G_{a(\rho)}}\rho(g^{-1})g \cdot \eta(1,\eta_{a(\rho)}).$$

Assuming $D_{\cris,L}(W)^{\varphi_L=q^{-1}}=0$, we consider the map
\begin{align*}
  \Theta_{W,n}:  K\otimes_L L_n\otimes_L D_{\cris,L}(W)   \to K[G_n]\otimes_L D_{\cris,L}(W)
\end{align*}
which is characterized by the following property: for a character $\rho$ of $G_n$, we have
\begin{equation}\label{def Theta}
({\rm pr}_\rho\otimes \id) \circ \Theta_{W,n} (x) = \begin{cases}
(1-\varphi_L)(1-q^{-1}\varphi_L^{-1})^{-1}{\rm Tr}_{L_n/L}(x), &\text{if $a(\rho)=0$,}\\
\tau(\rho)^{-1} q^{a(\rho)}\varphi_L^{a(\rho)} \sum_{g\in G_{a(\rho)}} \rho(g)g^{-1}({\rm Tr}_{L_n/L_{a(\rho)}}(x)) &\text{if $a(\rho)\geq 1$}
\end{cases}
\end{equation}
for any $x\in  L_n\otimes_L D_{\cris,L}(W)$, where ${\rm pr}_\rho: K[G_n] \to K$ denotes the ring homomorphism induced by $\rho$ and ${\rm Tr}_{L_n/L_{a(\rho)}}: L_n \otimes_L D_{\cris,L}(W)\to L_{a(\rho)} \otimes_L D_{\cris,L}(W)$ the map induced by the trace map $L_n\to L_{a(\rho)}$.  By definition, we have the following commutative diagram:
\begin{equation}\label{Theta compatible}
\xymatrix{
 K\otimes_L L_n\otimes_L D_{\cris,L}(W)  \ar[d]_{{\rm Tr}_{L_n/L_m}} \ar[r]^-{\Theta_{W,n}}&  K[G_n] \otimes_L D_{\cris,L}(W) \ar[d]^{{\rm pr}\otimes \id}\\
K\otimes_L L_m\otimes_L D_{\cris,L}(W) \ar[r]_-{\Theta_{W,m}} &  K[G_m]\otimes_L D_{\cris,L}(W).
}
\end{equation}
Note that $\Theta_{W,n}$
%the map
%$$ \Theta_{W,n}: K\otimes_L L_n\otimes_L D_{\cris,L}(W)   \to K[G_n]\otimes_L D_{\cris,L}(W)$$
is an isomorphism if $D_{\cris,L}(W)^{\varphi_L=q^{-1}}= D_{\cris,L}(W)^{\varphi_L=1}=0$.

For $x\in  D(\Gamma_L,K)\otimes_L D_{\cris,L}(W)$ we denote by $x(\chi_{\LT}^{j}) $ the image under the map $ D(\Gamma_L,K)\otimes_L D_{\cris,L}(W)\to K\otimes_L D_{\cris,L}(W),$ $\lambda\otimes d\mapsto \lambda(\chi_{\LT}^{j})\otimes d$ and similarly for other characters $\rho.$

\begin{lemma}\label{lem:Evn} Let $W$ be an $L$-analytic, crystalline  $L$-linear representation of $G_L$. Assume that $D_{\cris,L}(W)^{\varphi_L=q^{-1}}=0$. Then, for $n\geq 0,$  we have the following  commutative diagram:\small
\begin{equation*}
\xymatrix{
  D(\Gamma_L,K)\otimes_L D_{\cris,L}(W) \ar[d]_{\mathrm{pr}_{G_n}\otimes \id} \ar[r]^-{\mathfrak{M}\otimes \id} & \bigg(\mathcal{O}_K \otimes_L D_{\cris,L}(W)\bigg)^{\Psi=0}  &&  \bigg(\mathcal{O}_K \otimes_L D_{\cris,L}(W) \bigg)^{\Psi=1} \ar[ll]_-{1-\varphi_L\otimes \varphi_L} \ar[d]^{q^{-n}\mathrm{Ev}_{W,n}}\\
 \phantom{D(\Gamma_L,)} K[G_n]\otimes_L D_{\cris,L}(W)   &&& K\otimes_L L_n\otimes_L D_{\cris,L}(W) .\ar[lll]_-{\Theta_{W,n}}    }
\end{equation*}\normalsize
\end{lemma}
\begin{proof}
The case $n=0$ is easy, so we assume $n\geq 1$.
%\eqref{f:Ev} and the fact that  $\eta_0=0$ and $\left(\delta_g\cdot\eta(1,Z)\right)_{|Z=0}=1,$ whence $\mathrm{Ev}_{W,0}\circ(\mathfrak{M}\otimes \id) = \mathrm{pr}_{G_0}\otimes\id.$
%
%Now

We will check the commutativity $\rho$-componentwise, i.e., after applying $\mathfrak{e}_\rho$ for all characters $\rho$ of $G_n.$ We first consider the trivial character. In this case, using the case $m=0$ of the relation (\ref{f:Bergerrel}) and (\ref{Theta compatible}), the claim follows from the case $n=0$.

Now let $\rho$ be a character of conductor $a(\rho)\geq 1$. By (\ref{f:Bergerrel}) and (\ref{Theta compatible}) we may assume $a(\rho)=n$.
Then, noting that $\mathfrak{e}_\rho L_{n-1}=0$, we have
%, whence $\mathfrak{e}_\rho \mathrm{Ev}_{W,n-1}= 0$, and   \eqref{f:Ev} implies
\begin{equation}\label{f:phi}
 \mathfrak{e}_\rho \mathrm{Ev}_{W,n}(x)= \mathfrak{e}_\rho \mathrm{Ev}_{W,n}\bigg((1-\varphi_L\otimes \varphi_L)(x)\bigg).
\end{equation}
If $\lambda\in  D(\Gamma_L,K)\otimes_L D_{\cris,L}(W)$ satisfies $(\mathfrak{M}\otimes \id)(\lambda)=(1-\varphi_L\otimes \varphi_L)(x),$ then we see that
\begin{align*}
  \mathfrak{e}_\rho q^{-n}\mathrm{Ev}_{W,n}(x) & =\mathfrak{e}_\rho  q^{-n}\mathrm{Ev}_{W,n}\bigg((1-\varphi_L\otimes \varphi_L)(x)\bigg) \\
    & = \mathfrak{e}_\rho  q^{-n}\mathrm{Ev}_{W,n}\bigg((\mathfrak{M}\otimes \id)(\lambda)\bigg)\\
    & = q^{-n}\varphi_L^{-n}(\lambda )\mathfrak{e}_\rho  \mathrm{Ev}_{L,n}(\eta(1,Z))\\
    & =  q^{-n}\varphi_L^{-n}(\lambda )(\rho)\mathfrak{e}_\rho (\eta(1,\eta_n)) \\
    & = [L_n:L]^{-1}\tau(\rho) q^{-n}\varphi_L^{-n}(\lambda (\rho)),
\end{align*}
where we used for the third equation $D(\Gamma_L,K)$-linearity of the homomorphism $\mathrm{Ev}_{W,n}.$ Hence we have
$$\lambda(\rho)= \tau(\rho)^{-1}q^n\varphi_L^n \sum_{g\in G_n} \rho(g) g^{-1}\cdot q^{-n}{\rm Ev}_{W,n}(x)=({\rm pr}_\rho\otimes \id) \circ \Theta_{W,n} (q^{-n}{\rm Ev}_{W,n}(x)),$$
where the second equality follows from the definition of $\Theta_{W,n}$ (see (\ref{def Theta})). Since ${\rm pr}_\rho: K[G_n]\to K$ induces an isomorphism
$${\rm pr}_\rho: \mathfrak{e}_\rho K[G_n]\xrightarrow{\sim} K; \ \mathfrak{e}_\rho y \mapsto {\rm pr}_\rho(y), $$
we obtain
$$  \mathfrak{e}_\rho ({\rm pr}_{G_n}\otimes \id)(\lambda)= \mathfrak{e}_\rho \Theta_{W,n}(q^{-n}{\rm Ev}_{W,n}(x)),$$
which proves the claim.
%
%For the remaining case
%$n>a(\rho)\geq 1$ note that we can decompose $\mathfrak{e}_\rho^{G_n}$ as $\mathfrak{e}_\rho^{G_{a(\rho)}}q^{a(\rho)-n}\mathrm{Tr}_{L_n/L_{a(\rho)}}$ whence by \eqref{f:Bergerrel} this case is reduced to the previous one:
%\begin{align*}
% \mathfrak{e}_\rho^{G_n} q^{-n}\mathrm{Ev}_{W,n}(x)   & =\mathfrak{e}_\rho^{G_{a(\rho)}}q^{a(\rho)-n}\mathrm{Tr}_{L_n/L_{a(\rho)}} q^{-n}\mathrm{Ev}_{W,n}(x)  \\
%    & =q^{a(\rho)-n}
%\mathfrak{e}_\rho^{G_{a(\rho)}} q^{-a(\rho)}\mathrm{Ev}_{W,a(\rho)}(x) \\
%    & = q^{a(\rho)-n}[L_{a(\rho)}:L]^{-1}\tau(\rho) q^{-{a(\rho)}}\varphi_L^{-{a(\rho)}}(\lambda (\rho))\\
%    & =[L_{n}:L]^{-1}\tau(\rho) q^{-{a(\rho)}}\varphi_L^{-{a(\rho)}}(\lambda (\rho)).
%\end{align*}
\end{proof}
\begin{remark}\label{rem:EvDn}
Recall that $D:=D_{\cris,L}(L(\tau))=D^0_{\dR,L}(L(\tau)) = L \mathbf{d}_1$, on which $\varphi_L$ acts via $\pi_L/q$. We define
$$\mathrm{Ev}_{W,D,n}:\mathcal{O}_K \otimes_L \bigg(D_{\cris,L}(W)\otimes D\bigg) \to K\otimes_L L_n \otimes_L \bigg(D_{\cris,L}(W)\otimes D\bigg)$$
analogously to $\mathrm{Ev}_{W,n}$ taking now  the $\varphi_L$-action on $D_{\cris,L}(W)\otimes_L D$. Note that we have
$$\mathrm{Ev}_{W,D,n}(f\otimes d\otimes\mathbf{d}_1)=\begin{cases}
f(0) \otimes (1-\pi_L^{-1}\varphi_L^{-1}) (d)\otimes \mathbf{d}_1 &\text{if $n=0$,}\\
(\frac{q}{\pi_L})^n\mathrm{Ev}_{W,n}(f\otimes d)\otimes\mathbf{d}_1 &\text{if $n\geq 1$.}
\end{cases}$$
The variant of $\Theta_{W,n}$
$$\Theta_{W,D,n}: K\otimes_L L_n\otimes_L \bigg(D_{\cris,L}(W)\otimes_L D \bigg)  \to {K[G_n]}\otimes_L \bigg(D_{\cris,L}(W)\otimes_L D \bigg)$$
%\begin{align*}
%\Theta_{W,D,n}: K\otimes_L L_n\otimes_L \bigg(D_{\cris,L}(W)\otimes_L D \bigg) & \to {K[G_n]}\otimes_L \bigg(D_{\cris,L}(W)\otimes_L D \bigg),\\
%  \mathfrak{C}_{W,D,n}: K\otimes_L L_n\otimes_L \bigg(D_{\cris,L}(W)\otimes_L D \bigg)  & \to K\otimes_L L_n\otimes_L \bigg(D_{\cris,L}(W)\otimes_L D \bigg),\\
%\end{align*}
is defined similarly by extending the $\varphi_L$-action onto $D_{\cris,L}(W)\otimes_L D$ diagonally, under the assumption $(D_{\cris,L}(W)\otimes_L D)^{\varphi_L=q^{-1}}=0$ (or equivalently, $D_{\cris,L}(W)^{\varphi_L=\pi_L^{-1}}=0$).
Then we have
\begin{equation}\label{Thetad1}
\Theta_{W,D,n}(x\otimes d\otimes\mathbf{d}_1)= \Theta_{W,n}^\ast(x\otimes d) \otimes\mathbf{d}_1,
\end{equation}
where we define
$$\Theta_{W,n}^\ast:  K\otimes_L L_n\otimes_L D_{\cris,L}(W)\to  K[G_n]\otimes_L D_{\cris,L}(W)$$
to be the map characterized by
\begin{equation}\label{def Theta ast}
({\rm pr}_\rho\otimes \id)\circ \Theta_{W,n}^\ast (x) =\begin{cases}
(1-\pi_L^{-1}\varphi_L^{-1})^{-1}(1-\frac{\pi_L}{q}\varphi_L) {\rm Tr}_{L_n/L}(x) &\text{if $a(\rho)=0$,}\\
\tau(\rho)^{-1}\pi_L^{a(\rho)}\varphi_L^{a(\rho)}\sum_{g \in G_{a(\rho)}} \rho(g)g^{-1}({\rm Tr}_{L_n/L_{a(\rho)}}(x))&\text{if $a(\rho)\geq 1$}
\end{cases}
\end{equation}
for a character $\rho$ of $G_n$ and $x\in  L_n\otimes_L D_{\cris,L}(W) $  (the notation $\Theta_{W,n}^\ast $   indicating some adjointness property compared to $\Theta_{W^\ast(1)(\tau^{-1}),n}$ will be justified by Lemma \ref{lem:dR-pairing-cris} below).

There is the following variant of Lemma \ref{lem:Evn}, which we also will need for our decent calculation:
{\it Let $W$ be an $L$-analytic, crystalline  $L$-linear representation of $G_L$. Assume that $\Omega$ is contained in $K$ and that  $(D_{\cris,L}(W)\otimes_L D)^{\varphi_L=q^{-1}}=0$. Then, for $n\geq 0,$  there is the following  commutative diagram:}
\tiny
\begin{equation*}
\xymatrix{
  D(\Gamma_L,K)\otimes  \bigg(D_{\cris}(W)\otimes D\bigg) \ar[d]_{\mathrm{pr}_{G_n}\otimes \id} \ar[r]^-{\mathfrak{M}\otimes \id} & \bigg(\mathcal{O}_K \otimes  \bigg(D_{\cris}(W)\otimes D\bigg)\bigg)^{\Psi=0}  &&  \bigg(\mathcal{O}_K \otimes  \bigg(D_{\cris}(W)\otimes D\bigg) \bigg)^{\Psi=1} \ar[ll]_-{1-\varphi_L\otimes \varphi_L} \ar[d]^{q^{-n}\mathrm{Ev}_{W,D,n}}\\
 \phantom{D(\Gamma_L,)} K[G_n]\otimes  \bigg(D_{\cris}(W)\otimes D\bigg)  &&& K\otimes  L_n\otimes  \bigg(D_{\cris}(W)\otimes D\bigg) .\ar[lll]_-{\Theta_{W,D,n}}    }
\end{equation*}\normalsize
%{\color{red}$(\mathfrak{C}_{n,0}^*)^{-1}$ should be $\mathfrak{C}_{n,0}^*$?\color{blue} I guess it is rather $ [L_n:L]^2(\mathfrak{C}_{n,0}^*)^{-1}$}
\end{remark}

\Footnote{ Alternative approach: (with old definition of $\mathrm{Ev}_{W,0}!$)
For $n\geq k \geq  0,$ we define the maps
\begin{align*}
  \Upsilon_k^n:=\Upsilon_{W,k}^n:  K[G_n]&\otimes_L D_{\cris,L}(W)  \to  K\otimes_L L_n\otimes_L D_{\cris,L}(W) \\
  \lambda\otimes d & \mapsto (q^{-n}\lambda\cdot\varphi^{-k}(( \eta(1,\sigma_a\eta_k))\otimes d))_a=(q^{-n}(\lambda\cdot \eta(1,\sigma_a\eta_k))\otimes \varphi^{-k}(d))_a.
\end{align*}
We have
\begin{align}\label{f:rho}
  \mathfrak{e}_\rho \Upsilon_{a(\rho)}^n (x)=\left\{
\begin{array}{ll}
   q^{-n}\Upsilon_{W}^n(\chi_{\mathrm{triv}})(x) , & \hbox{if $\rho=\chi_{\mathrm{triv}}$;} \\
   {[L_{n}:L]}^{-1}  \Upsilon_{W}^n(\rho)(x), & \hbox{otherwise.}
\end{array}
\right.
%\in K\otimes_L D_{\dR,L}(W(\rho^{\ast}))= \mathfrak{e}_\rho (K\otimes_LL_n\otimes_L D_{\cris,L}(W)) ,
\end{align}
We set
\[\Xi_{W,n}:= (1-\varphi)^{-1}\mathrm{Ev}_{W,0} +\sum_{k=1}^n \mathrm{Ev}_{W,k}:\mathcal{O}_K \otimes_L D_{\cris,L}(W)\to K\otimes_L L_n\otimes_L D_{\cris,L}(W),\]
  supposed that $D_{\cris,L}(W)^{\varphi_L=1}=0$,
and
\[\beth_{W,n}:=(1-\varphi)^{-1}\Upsilon_{W,0}^n +\sum_{k=1}^n  \Upsilon_{W,k}^n:K[G_n]\otimes_L D_{\cris,L}(W)  \to  K\otimes_L L_n\otimes_L D_{\cris,L}(W).\]

\begin{lemma}\label{lem:Evnalt} Let $W$ be an $L$-analytic, crystalline  $L$-linear representation of $G_L$ and assume that $\Omega$ is contained in $K$. Then, for $n\geq 0,$  there is the following  commutative diagram:
%\begin{equation*}
%\xymatrix{
%  D(\Gamma_L,K)\otimes_L D_{\cris,L}(W) \ar[d]_{\mathrm{pr}_{G_n}\otimes \id} \ar[r]^-{\mathfrak{M}\otimes \id} & \bigg(\mathcal{O}_K \otimes_L D_{\cris,L}(W)\bigg)^{\Psi=0}\ar[d]^{{\color{red}q^{-n}\sum_{k=0}^n\mathrm{Ev}_{W,k} }} \ar[drr]^{{\color{red} q^{-n}\Xi_{W,n}}} &&  \bigg(\mathcal{O}_K \otimes_L D_{\cris,L}(W) \bigg)^{\Psi=1} \ar[ll]_-{1-\varphi_L\otimes \varphi_L} \ar[d]^{{\color{red}q^{-n}}\mathrm{Ev}_{W,n}}\\
% \phantom{D(\Gamma_L,)} K[G_n]\otimes_L D_{\cris,L}(W) \ar[r]^-{q^{-n}\sum_{k=0}^n\Upsilon_k}_\simeq &  \phantom{O} K\otimes_L L_n\otimes_L D_{\cris,L}(W)   && K\otimes_L L_n\otimes_L D_{\cris,L}(W) .\ar[ll]_-{\mathfrak{C}_n{\color{red} ???}}    }
%\end{equation*}

\begin{equation*}
\xymatrix{
  D(\Gamma_L,K)\otimes_L D_{\cris,L}(W) \ar[d]_{\mathrm{pr}_{G_n}\otimes \id} \ar[r]^-{\mathfrak{M}\otimes \id} & \bigg(\mathcal{O}_K \otimes_L D_{\cris,L}(W)\bigg)^{\Psi=0} \ar[d]^{{ q^{-n}\Xi_{W,n}}} &&  \bigg(\mathcal{O}_K \otimes_L D_{\cris,L}(W) \bigg)^{\Psi=1} \ar[ll]_-{1-\varphi_L\otimes \varphi_L} \ar[d]^{{q^{-n}}\mathrm{Ev}_{W,n}}\\
 \phantom{D(\Gamma_L,)} K[G_n]\otimes_L D_{\cris,L}(W) \ar[r]^-{\beth_{W,n}}_\simeq &  \phantom{O} K\otimes_L L_n\otimes_L D_{\cris,L}(W)   && K\otimes_L L_n\otimes_L D_{\cris,L}(W) .\ar@{=}[ll]   }
\end{equation*}

%and
%\begin{equation*}\scriptsize
%\xymatrix{
%  D(\Gamma_L,K)\otimes_L D_{\cris,L}(W) \ar[d]_{\mathrm{Tw}_{\chi_{\LT}^{-j}}\otimes e_j} \ar[r]^-{\mathfrak{M}\otimes \id} & \mathcal{O}_K \otimes_L D_{\cris,L}(W)\ar[d]^{(\frac{\partial}{\Omega})^{-j}\otimes e_j} &  \mathcal{O}_K \otimes_L D_{\cris,L}(W) \ar[l]_{1-\varphi_L\otimes \varphi_L}\ar[d]^{(\frac{\partial}{\Omega})^{-j}\otimes e_j}\\
%D(\Gamma_L,K)\otimes_L D_{\cris,L}(W(\chi_{\LT}^j))  \ar[r]^-{\mathfrak{M}\otimes \id} & \mathcal{O}_K \otimes_L D_{\cris,L}(W(\chi_{\LT}^j)) &  \mathcal{O}_K \otimes_L D_{\cris,L}(W(\chi_{\LT}^j)) \ar[l]_{1-\varphi_L\otimes \varphi_L}  .   }
%\end{equation*}
%In the latter we follow the (for $j>0$) abusive   notation $\partial^{-j}$  from \cite[Rem.\ 3.5.5.]{BF}. {\color{red} we need an analogue for $n>0$!}
\end{lemma}

\begin{proof}  We first deal with the case $n=0:$ For the left diagram note that $\eta_0=0$ and $\left(\delta_g\cdot\eta(1,Z)\right)_{|Z=0}=1,$ from which the claim follows for Dirac distributions, whence in general. For the right square we observe that $\varphi_L(g(Z))_{|Z=0}=g(0).$ Now assume $n>0.$ The commutativity of the left square follows from the relation
\begin{align*}
  (\gamma_a \eta(1,Z))_{|Z=\eta_n} = & \eta(1,[a](Z))_{|Z=\eta_n} \\
  =  & \eta(1,[a](\eta_n)) \\
  =  & \gamma_a \eta(1,\eta_n).
\end{align*}
and the definition of the maps. From \eqref{f:Ev} we have
%\begin{equation*}
%  \mathrm{Ev}_{W,k}\bigg((1-\varphi_L\otimes \varphi_L)(x)\bigg)=\left\{
%                                                                   \begin{array}{ll}
%                                                                     \mathrm{Ev}_{W,k}(x)-\mathrm{Ev}_{W,k-1}(x), & \hbox{if $k\geq 1$;} \\
%                                                                    (1-\varphi)\mathrm{Ev}_{W,0}(x), & \hbox{if $k=0$,}
%                                                                   \end{array}
%                                                                 \right.
%\end{equation*}
%whence
\begin{equation*}
   \mathrm{Ev}_{W,n}(x)=(1-\varphi)^{-1}\mathrm{Ev}_{W,0}\bigg((1-\varphi_L\otimes \varphi_L)(x)\bigg)+\sum_{k=1}^n\mathrm{Ev}_{W,k}\bigg((1-\varphi_L\otimes \varphi_L)(x)\bigg),
\end{equation*}
supposed that $D_{\cris,L}(W)^{\varphi_L=1}=0$ (or modulo $D_{\cris,L}(W)^{\varphi_L=1} $).
This shows the commutativity of the right square.
%
% can be checked after applying the   idempotents $\mathfrak{e}_\rho$ for $\rho$  running through the characters of $G_n:=\Gamma_L/\Gamma_n$. Let $x$ be in $\bigg(\mathcal{O}_K \otimes_L D_{\cris,L}(W) \bigg)^{\Psi=1}.$ The trivial character being dealt with already  we may assume without loss of generality - due to the following relation for $n\geq m\geq 0:$
%\[ q^{-(n-m)} \mathrm{Tr}_{L_n/L_m}(\mathrm{Ev}_{W,n}(x) )=\left\{
%                                                          \begin{array}{ll}
%                                                            \mathrm{Ev}_{W,m}(x), & \hbox{if $m\geq 1$;} \\
%                                                            (1-q^{-1}\varphi_L^{-1})(\mathrm{Ev}_{W,0}(x)), & \hbox{if $m=0$.}
%                                                          \end{array}
%                                                        \right.
% \]
%by \cite[Lem.\  2.4.3]{BF} -
% that $\rho$ does not factorize over $G_{n-1}.$ Then $\mathfrak{e}_\rho L_{n-1}=0$ and a straight forward calculation shows that, for $x=\sum_{i\geq 0, j} a_iZ^i\otimes d_j \in \mathcal{O}_K \otimes_L D_{\cris,L}(W),$ we have
%\begin{equation}\label{f:phi}
% \mathfrak{e}_\rho \mathrm{Ev}_{W,n}(x)= \mathfrak{e}_\rho \mathrm{Ev}_{W,n}\bigg((1-\varphi_L\otimes \varphi_L)(x)\bigg).
%\end{equation}
\end{proof}
%
%\begin{remark}\label{rem:EvDn}
%There is the following variant of the previous lemma, which we also will need for our decent calculation:\tiny
%\begin{equation*}
%\xymatrix{
%  D(\Gamma_L,K)\otimes_L \bigg(D_{\cris,L}(W)\otimes D\bigg) \ar[d]_{\mathrm{pr}_{G_n}\otimes \id} \ar[r]^-{\mathfrak{M}\otimes \id} & \bigg(\mathcal{O}_K \otimes_L \bigg(D_{\cris,L}(W)\otimes D\bigg)\bigg)^{\Psi=0}\ar[d]^{\mathrm{Ev}'_{W,n}} &&  \bigg(\mathcal{O}_K \otimes_L \bigg(D_{\cris,L}(W)\otimes D\bigg) \bigg)^{\Psi=1} \ar[ll]_-{1-\varphi_L\otimes \varphi_L} \ar[d]^{\mathrm{Ev}'_{W,n}}\\
% \phantom{D(\Gamma_L,)} K[G_n]\otimes_L \bigg(D_{\cris,L}(W)\otimes D\bigg) \ar[r]^{\Upsilon'_n}_\simeq &  \phantom{O} K\otimes_L L_n\otimes_L \bigg(D_{\cris,L}(W)\otimes D\bigg)   && K\otimes_L L_n\otimes_L \bigg(D_{\cris,L}(W)\otimes D\bigg) .\ar[ll]_-{\mathfrak{C}'_n}    }
%\end{equation*}\normalsize
%Here, $\mathrm{Ev}'_{W,n}$ is defined analogously to $\mathrm{Ev}_{W,n}$ taking now   $\varphi_L$-action on $D_{\cris,L}(W)\otimes_L D$. Note that we have $\mathrm{Ev}'_{W,n}(f\otimes d\otimes\mathbf{d}_1)=(\frac{q}{\pi_L})^n\mathrm{Ev}_{W,n}(f\otimes d)\otimes\mathbf{d}_1$ and analogously for $\Upsilon_n',$ while $ \mathfrak{C}'_n$ corresponds to $ \mathfrak{C}_n$ with $  1-\frac{
%\pi_L}{q}\varphi_L$ in the component of the trivial character.
%\end{remark}
}

\subsection{Twisting}

In this subsection, we review some basic properties of twisting maps. The following result is proved in \cite[Lem.\ 5.2.24]{SV24}. (We abbreviate $\partial_{\rm inv}$ to $\partial$.)

\begin{lemma}\label{lem:Tw}
There is the following commutative diagram:
\begin{equation*}\small
\xymatrix{
  D(\Gamma_L,K)\otimes_L D_{\cris,L}(W) \ar[d]_{\mathrm{Tw}_{\chi_{\LT}^{-j}}\otimes e_j} \ar[r]^-{\mathfrak{M}\otimes \id} & \mathcal{O}_K \otimes_L D_{\cris,L}(W)\ar[d]^{(\frac{\partial}{\Omega})^{-j}\otimes e_j} &  \mathcal{O}_K \otimes_L D_{\cris,L}(W) \ar[l]_{1-\varphi_L\otimes \varphi_L}\ar[d]^{(\frac{\partial}{\Omega})^{-j}\otimes e_j}\\
D(\Gamma_L,K)\otimes_L D_{\cris,L}(W(\chi_{\LT}^j))  \ar[r]^-{\mathfrak{M}\otimes \id} & \mathcal{O}_K \otimes_L D_{\cris,L}(W(\chi_{\LT}^j)) &  \mathcal{O}_K \otimes_L D_{\cris,L}(W(\chi_{\LT}^j)) \ar[l]_{1-\varphi_L\otimes \varphi_L}  .   }
\end{equation*}
Here we follow the (for $j>0$) abusive   notation $\partial^{-j}$  from \cite[Rem.\ 3.5.5]{BF}.
\end{lemma}

Consider the following pairing induced by local Tate duality:
\begin{align*}
  \xymatrix{
H^1_{\dagger}(L_n,V^*(1)(\chi_{\LT}^{j}))\ar@{}[r]|{\times} &H^1_{/\dagger}(L_n,V(\chi_{\LT}^{-j}))  \ar[rr]^-{\langle,\rangle_{\Tate,L_n,\dagger}} && H^2(L_n,L(1))\simeq L\subseteq\Cp . }
\end{align*}
It induces the next pairing
\begin{align*}
  \xymatrix{
H^1_{\dagger}(L_n,V^*(1)(\chi_{\LT}^{j}))\ar@{}[r]|{\times} &H^1_{/\dagger}(L_n,V(\chi_{\LT}^{-j}))  \ar[rr]^-{\langle\langle,\rangle\rangle_{\Tate,L_n,\dagger}} && \Cp[G_n], }
\end{align*}
by defining $\langle\langle x,y \rangle\rangle_{\Tate,L_n,\dagger}:=\sum_{g\in G_n} \langle g^{-1}x, y\rangle_{\Tate,L_n,\dagger}g.$
By Frobenius reciprocity we obtain from \cite[Prop.\ 5.2.18]{SV24} and its proof (in particular equation (5.27)) combined with remark 4.5.21 in (loc.\ cit.) the following variant of Cor.\ 5.2.20 of (loc.\ cit.):
\begin{proposition}\label{prop:Taten}
For a representation $V$ of $G_L$ such that $V^*(1)$ is $L$-analytic, the following diagram is commutative:
\begin{equation*}
\xymatrix{
H^1_{\dagger}(L_n,V^*(1)(\chi_{\LT}^{j}))\ar@{}[r]|{\times} &H^1_{/\dagger}(L_n,V(\chi_{\LT}^{-j}))  \ar[rr]^-{\langle\langle,\rangle\rangle_{\Tate,L_n,\dagger}} && \Cp[G_n]\\
    {D_{\mathrm{rig}}^\dagger(V^*(1))^{\psi_L=\frac{q}{\pi_L}}}\ar[u]_{h^1_{L_n,V^*(1)}\circ \mathrm{tw}_{\chi_{\LT}^{j}}} \ar@{}[r]|{\times} &  {D_{\mathrm{rig}}^\dagger(V(\tau^{-1}))^{\psi_L=1}}\ar[u]_{\pr_{L_n}\circ \mathrm{tw}_{\chi_{\LT}^{-j}} } \ar[rr]^-{ \frac{q-1}{q}\{,\}_{\Iw}} && D(\Gamma_L,\Cp).\ar[u]_{ \mathrm{pr}_{G_n}\circ\mathrm{Tw}_{\chi_{\LT}^{-j}} } }
\end{equation*}
\end{proposition}
With respect to evaluating at a character we have the following analogue of Proposition \ref{prop:Taten}:

\begin{proposition}\label{prop:crisn}
For a $G_L$-representation $V$ such that $V^*(1)$ is $L$-analytic the following diagram is commutative, in which the middle pairing $[[,]]'_{\cris,n}$ is the $\Cp[G_n]$-$\upiota_*$-sesquilinear extension of $[,]'_{\cris}: D_{\cris,L}(V^*(1)(\chi_{\LT}^{j}))\times D_{\cris,L}(V(\tau^{-1})(\chi_{\LT}^{-j}))\to L $. \small
\begin{equation*}\small
\xymatrix{
\Cp[G_n]\otimes_L D_{\cris,L}(V^*(1)(\chi_{\LT}^{j}))\ar@{}[r]|{\times}  & \Cp[G_n]\otimes_L D_{\cris,L}(V(\tau^{-1})(\chi_{\LT}^{-j})) \ar[r]^-{[[,]]'_{\cris,n}} &  \Cp[G_n] \\
    {D(\Gamma_L,\mathbb{C}_p)\otimes_L D_{\cris,L}(V^*(1))}\ar[u]_{\mathrm{pr}_{G_n}\circ\mathrm{Tw}_{\chi_{\LT}^{-j}} \otimes t_{\LT}^{-j} \eta^{\otimes j} } \ar@{}[r]|{\times} &  {D(\Gamma_L,\mathbb{C}_p)\otimes_L D_{\cris,L}(V(\tau^{-1})) }\ar[u]_{\mathrm{pr}_{G_n}\circ\mathrm{Tw}_{\chi_{\LT}^{j}} \otimes t_{\LT}^{j} \eta^{\otimes -j} } \ar[r]^-{[,]^0}& D(\Gamma_L,\Cp).\ar[u]_{ \mathrm{pr}_{G_n}\circ\mathrm{Tw}_{\chi_{\LT}^{-j}} } }
\end{equation*}\normalsize
\end{proposition}

\begin{proof}
Using \cite[Lem.\ 5.2.23]{SV24}  the statement is reduced to $j=0.$ Evaluation of \eqref{f:def[]0} implies the claim in this case.
\end{proof}

\subsection{Dual exponential and logarithm maps}

Recall that we set
$$D_{\dR,L'}(V) := ( B_{\dR} \otimes_{\qp} V)^{G_{L'}}$$
for an algebraic extension $L'/L$.
Since $B_{\dR}$ contains the algebraic closure $\overline{L}$ of $L$ we have the isomorphism
\begin{equation*}
  B_{\dR} \otimes _{\mathbb{Q}_p} V = (B_{\dR} \otimes _{\mathbb{Q}_p} L) \otimes_L V  \xrightarrow{\;\simeq\;} \prod_{\sigma \in G_{\mathbb{Q}_p}/G_L} B_{\dR} \otimes_{\sigma,L} V
\end{equation*}
which sends $b \otimes v$ to $(b \otimes v)_\sigma$. The tensor product in the factor $B_{\dR} \otimes_{\sigma,L} V$ is formed with respect to $L$ acting on $B_{\dR}$ through $\sigma$. With respect to the $G_L$-action the right hand side decomposes according to the double cosets in $G_L \backslash G_{\mathbb{Q}_p}/G_L$. It follows, in particular, that
$$D_{\dR,L}^{\id}(V) := ( B_{\dR} \otimes_L V)^{G_L}$$
is a direct summand of $D_{\dR,L}(V)$.
% and we denote by $\pr^{\id}$ the corresponding projection.
Similarly,
$${\rm tan}_{L,\id}(V) := ( B_{\dR}/B^+_{\dR} \otimes_L V)^{G_L}  $$
is a direct summand of  ${\rm tan}_L(V):= ( B_{\dR}/B_{\dR}^+ \otimes_{\QQ_p} V)^{G_L} $.
 More generally, also the filtration $D_{\dR,L}^i(V)$ decomposes into direct summands.

In the Lubin-Tate setting we can   consider the dual of the identity component
$$\exp_{L',V^*(1),\id} : {\rm tan}_{L',\id}(V^\ast(1)) \to H^1(L',V^\ast(1))$$
of $\exp_{L',V^*(1)}, $ which is given by the commutativity of the following diagram - note that in this paper in contrast to \cite{SV24} we use Kato's sign convention for the dual exponential map of $V$, i.e., variant (b) in \S \ref{rem sign}:
\begin{equation}
\label{f:dualexpId}
\xymatrix{
K\otimes_L H^{1} (L_n,V^*(1)) \ar@{}[r]|(0.45){\times} & { \phantom{} }K\otimes_L H^1 (L_n,V)\ar[d]^{\widetilde{\exp}_{L_n,V,\id}^*}  \ar[rr]^-{\langle\langle,\rangle\rangle_{\Tate,L_n}} &&  K[G_n]\ar@{=}[dd]\\
  K\otimes_L {\rm tan}_{L_n,\id}(V^\ast(1)) \ar[u]_{\exp_{L_n,V^*(1),\id}}\ar@{}[r]|(0.45){\times} & K\otimes_L D_{\dR,L_n}^{\id,0}(V(\tau^{-1})) \ar@{^(->}[d]  & &  \\
 K\otimes_L   D_{\dR,L_n}^{\id}(V^*(1))\ar@{->>}[u]^{\pr}
\ar@{}[r]|(0.45){\times} & K\otimes_L D_{\dR,L_n}^{\id}(V(\tau^{-1}))\ar[rr]^-{[[,]]_{\mathrm{dR},L_n}} && K[G_n].}
\end{equation}
Here,  the  $K[G_n]$-$\upiota_*$-sesquilinear bottom pairing is induced by the $K$-linear base change $[, ]_{\mathrm{dR},L_n}$ of the pairing
%\begin{align*}
%  \xymatrix{
%K\otimes_L D_{\dR,L_n}^{\id}(V(\tau^{-1}))\ar@{}[r]|{\times} & K\otimes_L  D_{\dR,L_n}^{\id}(V^*(1)) \ar[rr]^-{[[,]]_{\mathrm{dR},L_n}} && K[G_n], }
%\end{align*}
\begin{align*}
  \xymatrix{
   D_{\dR,L_n}^{\id}(V^*(1))\ar@{}[r]|{\times} &    D_{\dR,L_n}^{\id}(V(\tau^{-1}))\ar[r]^{} & D_{\dR,L_n}^{\id}(L(\chi_{\LT}))\ar[r]^-{\simeq} & L_n\ar[r]^{\mathrm{Tr}_{L_n/L}} & L. }
\end{align*}
by defining $[[x,y ]]_{\mathrm{dR},L_n}:=\sum_{g\in G_n} [g^{-1} x, y]_{\mathrm{dR},L_n}\ g.$

Upon noting that under the identifications $D_{\dR,L'}(\mathbb{Q}_p(1)) \simeq L'$ and $D_{\dR,L'}^{\id}(\mathbb{Q}_p(1)){\simeq}  L'$ the elements $t_{\mathbb{Q}_p}\otimes \eta^{\cyc}$  and $t_{\LT}\otimes \eta$  are sent to $1,$ one easily checks that, if $V^*(1)$ is $L$-analytic, whence the inclusion ${\rm tan}_{L',\id}(V^*(1)) \subseteq {\rm tan}_{L'}(V^*(1)) $ is an equality and  $\exp_{L',V^*(1),\id} =\exp_{L',V^*(1)} $, then it holds
\begin{equation}\label{f:dualexpcomp}
  \mathbb{T}_{\tau^{-1}}\circ \exp_{L',V}^*=\widetilde{\exp}_{L',V,\id}^*,
\end{equation}
where $\mathbb{T}_{\tau^{-1}}: D_{\dR,L'}^0(V) \to D_{\dR,L'}^{\id,0}(V(\tau^{-1})) $ is the isomorphism, which sends $b\otimes v$ to $b\frac{t_{\mathbb{Q}_p}}{t_{\LT}}\otimes v\otimes \eta\otimes\eta_{\cyc}^{\otimes -1};$ note that $\frac{t_{\mathbb{Q}_p}}{t_{\LT}}\in (B_{\dR}^+)^\times,$ whence the filtration is preserved.

Now let $W(=V^*(1))$ be an $L$-analytic, crystalline  $L$-linear representation of $G_L.$
%Recall that $\eta=(\eta_n)_n$ denotes a fixed generator of $T_\pi$ and
Consider the map $\mathrm{tw}_{\chi_{\LT}^j}:D^\dagger_{\rig}(W)\to D^\dagger_{\rig}(W(\chi_{\LT}^j));$ $d\mapsto d\otimes \eta^{\otimes j}$. For $D_{\cris}$ twisting $D_{\cris,L}(W)\xrightarrow{-\otimes e_j} D_{\cris,L}(W(\chi_{\LT}^j)) $ maps $d$ to $d\otimes e_j$ with $e_j:=t_{\LT}^{-j}\otimes \eta^{\otimes j}\in D_{\cris,L}(L(\chi_{\LT}^j)).$

If we assume, in addition, that
%{\color{blue} how are these conditions compatible with $h\geq 1$ such that $\mathrm{Fil}^{-h}D_{\cris,L}(W)=D_{\cris,L}(W)$ below in Theorems \ref{thm:BF} and \ref{thm:adjointformulan}?}
\begin{enumerate}
\item $W$ has Hodge-Tate weights   $\leq 0$,   whence $W^*(1)$ has Hodge-Tate weights $\geq 1$ and $D_{\dR,L}^0(W^*(1))=0$, and
\item $D_{\cris,L}(W^*(\chi_{\LT}))^{\varphi_L=\frac{q}{\pi_L}}=0,$
\end{enumerate}
then  $\exp_{L,W^*(1)}:D_{\dR,L}(W^*(1))\hookrightarrow H^1(L,W^*(1))$ is   injective    with image $H^1_e(L,W^*(1))=H^1_f(L,W^*(1))$ by our assumption (see \cite[Cor.\ 3.8.4]{BK}). We denote its inverse by
\[\log_{L,W^*(1)}:H^1_f(L,W^*(1))\to D_{\dR,L}(W^*(1))\] and define
\[\widetilde{\log}_{L,W^*(1)}:H^1_f(L,W^*(1))\to D_{\dR,L}(W^*(1))\xrightarrow{\mathbb{T}_{\tau^{-1}}} D_{\cris,L}(W^*(\chi_{\LT}))\]
where  (by abuse of notation) we also write $\mathbb{T}_{\tau^{-1}}: D_{\dR,L}(W^*(1)) \to D_{\dR,L}^{\id}(W^*(\chi_{\LT}))=D_{\cris,L}(W^*(\chi_{\LT})) $ for the isomorphism, which sends $b\otimes v$ to $b\frac{t_{\mathbb{Q}_p}}{t_{\LT}}\otimes v\otimes \eta\otimes\eta_{\cyc}^{\otimes -1}.$
We obtain the following commutative diagram, which defines the dual map $\log_{L,W}^*$ being inverse to $\exp_{L,W}^*$ (up to factorisation over $H^{1} (L,W)/H^1_f(L,W) $):
%{\color{red} here we use  sign   (a) for dual log for (analytic) $W$, right? This does not seem to be compatible with [BF] where Kato's convention is used for $\exp_{L,W}^*$.}
\begin{equation*}
\label{f:duallog}
\begin{gathered}
	\xymatrix{
H^{1} (L,W)/H^1_f(L,W) \ar@{}[r]|(0.5){\times} & { \phantom{} }H^1_f (L,W^*(1)) \ar[d]_{{  -}\log_{L,W^*(1)}}\ar[rr]^-{\langle ,\rangle_{\Tate,L}} &&   L\ar@{=}[d]\\
D_{\dR,L}(W)
\ar@{}[r]|(0.45){\times}\ar[u]^{\log_{L,W}^*}  & D_{\dR,L}(W^*(1))  \ar[r]^(0.5){} & D_{\dR,L}(\mathbb{Q}_p(1))\ar[r]^-{\simeq} & L.}
\end{gathered}
\end{equation*}
  Note that in this and the next diagram a sign in front of $\log$ shows up due to the skew-symmetry of the cup-product as we  changed the order compared to \eqref{f:dualexpId}.
Similarly as above there is a commutative diagram
\begin{equation}
\label{f:duallogId}
\xymatrix{
 K\otimes_L H^{1} (L_n,V^*(1))/ K\otimes_L H^1_f(L,V^*(1)) \ar@{}[r]|-{\times} & { \phantom{} } K\otimes_L H^1_f (L_n,V) \ar[d]_{{  -}\widetilde{\log}_{L_n,V,\id}} \ar[rr]^-{\langle\langle,\rangle\rangle_{\Tate,L_n}} &&  K[G_n]\ar@{=}[d]\\
%D_{\dR,L}^{\id,0}(W(\tau^{-1}))\ar@{^(->}[d]\ar@{}[r]|(0.45){\times} & {\rm tan}_{L,\id}(W^*(1)) \ar[u]_{\exp_{L,W^*(1),\id}} \ar[r]^(0.45){} & D_{\dR,L}^{\id}(L(\chi_{\LT}))\ar@{=}[d]\ar[r]^(0.65){\simeq} & L\ar@{=}[d]\\
  K\otimes_L  D_{\dR,L_n}^{\id}(V^*(1))\ar[u]^{ {\log}_{L_n,V^*(1),\id}^*}
\ar@{}[r]|-{\times} &  K\otimes_LD_{\dR,L_n}^{\id}(V(\tau^{-1}))  \ar[rr]^-{[[,]]_{\mathrm{dR},L_n}} && K[G_n].}
\end{equation}

\subsection{Interpolation properties}

In this subsection we are going to prove the   interpolation property regarding Artin representations for the regulator map $\mathbf{L}_V$. %and $\mathcal{L}_V$, respectively.

We first recall the interpolation property of Berger's and Fourquaux' big exponential map:

\begin{theorem}[{Berger-Fourquaux \cite[Thm.\ 3.5.3]{BF}}]\label{thm:BF}
 Let $W$ be an $L$-analytic crystalline  $L$-linear representation of $G_L$ and $h\geq 1$ such that $\mathrm{Fil}^{-h}D_{\cris,L}(W)=D_{\cris,L}(W)$. For any  $f\in \left(\mathcal{O}^{\psi=0}\otimes_L D_{\cris,L}(W)\right)^{\Delta=0}$ and $y\in \left(\mathcal{O} \otimes_L D_{\cris,L}(W)\right)^{\psi=\frac{q}{\pi_L}}$ with $f=(1-\varphi_L)y$ we have the following: if $h+j\geq 1$, then
 \begin{equation}\label{f:interjgeq1}
 h^1_{L_n,W(\chi_{\LT}^j)}  (\mathrm{tw}_{\chi_{\LT}^j}(\Omega_{W,h}(f)))= (-1)^{h+j-1}(h+j-1)!  \exp_{L_n,W(\chi_{\LT}^j)}\Big(q^{-n}\mathrm{Ev}_{W(\chi_{\LT}^j),n}( \partial_{\inv}^{-j}y\otimes e_j)\Big),
 \end{equation}
%
%\begin{align}\label{f:interjgeq1}\notag
%  &h^1_{L_n,W(\chi_{\LT}^j)}  (\mathrm{tw}_{\chi_{\LT}^j}(\Omega_{W,h}(f)))= \\
%   & (-1)^{h+j-1}(h+j-1)!\left\{
%                           \begin{array}{ll}
%\exp_{L,W(\chi_{\LT}^j)}\Big((1-q^{-1}\varphi_L^{-1})\mathrm{Ev}_{W(\chi_{\LT}^j),0}   ( \partial_{\inv}^{-j}y\otimes e_j)\Big)         , & \hbox{if $n=0$;}\\
%                             \exp_{L_n,W(\chi_{\LT}^j)}\Big(q^{-n}\mathrm{Ev}_{W(\chi_{\LT}^j),n}( \partial_{\inv}^{-j}y\otimes e_j)\Big) & \hbox{if $n\geq 1$.}
%                           \end{array}
%                         \right.
%\end{align}
and if $h+j\leq  0$, then
\begin{equation}\label{f:interjleq0}
\exp_{L_n,W(\chi_{\LT}^j)}^* \Big(  h^1_{L_n,W(\chi_{\LT}^j)}  (\mathrm{tw}_{\chi_{\LT}^j}(\Omega_{W,h}(f)))\Big)=\frac{1}{(-h-j)!} q^{-n}\mathrm{Ev}_{W(\chi_{\LT}^j),n}(\partial_{\inv}^{-j}y\otimes e_j) .
\end{equation}
%
%\begin{align}\label{f:interjleq0}\notag
%\exp_{L_n,W(\chi_{\LT}^j)}^* &\Big(  h^1_{L_n,W(\chi_{\LT}^j)}  (\mathrm{tw}_{\chi_{\LT}^j}(\Omega_{W,h}(f)))\Big)= \\
%   & \frac{1}{(-h-j)!}\left\{
%                           \begin{array}{ll}
%(1-q^{-1}\varphi_L^{-1})\mathrm{Ev}_{W(\chi_{\LT}^j),0}   (\partial_{\inv}^{-j}y\otimes e_j)         , & \hbox{if $n=0$;}\\
%                            q^{-n}\mathrm{Ev}_{W(\chi_{\LT}^j),n}(\partial_{\inv}^{-j}y\otimes e_j) & \hbox{if $n\geq 1$.}
%                                               \end{array}
%                         \right.
%\end{align}
\end{theorem}
By abuse of notation we shall denote the base change $K\otimes_L-$ of the (dual) Bloch-Kato exponential map by the same expression. Also, we abbreviate \[x(\chi_{\LT}^{-j},n):=\mathrm{pr}_{G_n}(\mathrm{Tw}_{\chi_{\LT}^{-j}}(x)).\] Using Lemmata  \ref{lem:Evn} and \ref{lem:Tw} we deduce the following interpolation property for the modified big exponential map with $x\in D(\Gamma_L,K)\otimes_L D_{\cris,L}(W)$:
if $j\geq 0$,  then
\begin{equation}\label{f:interjgeq1mod}
  h^1_{L_n,W(\chi_{\LT}^j)}  (\mathrm{tw}_{\chi_{\LT}^j}(\mathbf{\Omega}_{W,1}(x)))=(-1)^{j}j! \Omega^{-j-1} \exp_{L_n,W(\chi_{\LT}^j)}\Big(\Theta_{W(\chi_{\rm LT}^j),n}^{-1}
          \left(x(\chi_{\LT}^{-j},n)  \otimes e_j\right)\Big),
\end{equation}
%
%\begin{align}\label{f:interjgeq1mod}\notag
%  h^1_{L_n,W(\chi_{\LT}^j)} & (\mathrm{tw}_{\chi_{\LT}^j}(\mathbf{\Omega}_{W,1}(x)))= \\
%   &  (-1)^{j}j! \Omega^{-j-1}\left\{
%       \begin{array}{ll}
%         \exp_{L,W(\chi_{\LT}^j)}\Big((1-q^{-1}\varphi_L^{-1})(1-\varphi_L)^{-1}\left(x(\chi_{\LT}^{-j})  \otimes e_j\right)\Big), & \hbox{if $n=0$;} \\
%          \exp_{L_n,W(\chi_{\LT}^j)}\Big(
%          %\mathfrak{C}_n^{-1}\circ\Upsilon^n
%          {\color{red}\Theta_{W(\chi_{\rm LT}^j),n}^{-1}}
%          \left(x(\chi_{\LT}^{-j},n)  \otimes e_j\right)\Big), & \hbox{if $n>0$.}
%       \end{array}
%     \right.
%\end{align}
and if $ j< 0$, then
\begin{equation}\label{f:interjleq0mod}\notag
 h^1_{L_n,W(\chi_{\LT}^j)}   (\mathrm{tw}_{\chi_{\LT}^j}(\mathbf{\Omega}_{W,1}(f)))=\frac{1}{(-1-j)!}\Omega^{-j-1}  \log_{L_n,W(\chi_{\LT}^j)}^*\Big(\Theta_{W(\chi_{\rm LT}^j),n}^{-1}
 \left(x(\chi_{\LT}^{-j},n)  \otimes e_j\right)\Big),
\end{equation}
%
%\begin{align}\label{f:interjleq0mod}\notag
%  h^1_{L_n,W(\chi_{\LT}^j)}  & (\mathrm{tw}_{\chi_{\LT}^j}(\mathbf{\Omega}_{W,1}(f)))=  \\
%   & \frac{1}{(-1-j)!}\Omega^{-j-1} \left\{
%                         \begin{array}{ll}
%                           \log_{L,W(\chi_{\LT}^j)}^*\Big(
% (1-q^{-1}\varphi_L^{-1})(1-\varphi_L)^{-1}\left(x(\chi_{\LT}^{-j})  \otimes e_j\right)\Big), & \hbox{if $n=0$;} \\
%                          \log_{L_n,W(\chi_{\LT}^j)}^*\Big(
% %\mathfrak{C}_n^{-1}\circ\Upsilon^n
%  {\color{red}\Theta_{W(\chi_{\rm LT}^j),n}^{-1}}
% \left(x(\chi_{\LT}^{-j},n)  \otimes e_j\right)\Big), & \hbox{if $n>0$.}
%                         \end{array}
%                       \right.
%\end{align}
assuming in both cases that  the operators $1-\varphi_L$ and $1-q^{-1}\varphi_L^{-1}$ are invertible on $D_{\cris,L}(W(\chi_{\LT}^j))$ (so that $\Theta_{W(\chi_{LT}^j),n}$ is bijective and $\log_{L_n,W(\chi_{\LT}^j)}$ is defined).
% the operator $1-\varphi_L$ is invertible on $D_{\cris,L}(W(\chi_{\LT}^j))$ and for $j<0$ also that the operator $ 1-q^{-1}\varphi_L^{-1}$ is invertible on $D_{\cris,L}(W(\chi_{\LT}^j))$ (in order to grant the existence of $\log_{L,W(\chi_{\LT}^j)}$).

For a character $\rho$ of $G_n$, recall that we write ${\rm pr}_\rho: K[G_n]\to K$ for the ring homomorphism induced by $\rho$. The aim of the next Lemma is to determine the adjoint of the map $\Theta_{W,n}$. Recall that $\Theta_{W,n}^\ast$ is defined in (\ref{def Theta ast}).

\begin{lemma}\label{lem:dR-pairing-cris} Assume that the $L$-analytic representation $V^*(1)$ is crystalline  and also that $D_{\cris,L}(V^\ast(1))^{\varphi_L=q^{-1}} = D_{\cris,L}(V^\ast(1))^{\varphi_L=1}=0$. Then
$[[,]]_{\mathrm{dR},L_n} $ and $[[,]]_{\cris} $ can be compared by the commutativity of the  following diagram:
%\small
%\begin{equation*}\scriptsize
%\xymatrix{
%K\otimes_L D_{\dR,L_n}^{\id}(V^*(1))\ar@{}[r]|{\times} & K\otimes_L  D_{\dR,L_n}^{\id}(V(\tau^{-1})) \ar[r]^-{[[,]]_{\mathrm{dR},L_n}} & K[G_n]\ar@{=}[dd]\\
%K\otimes_L L_n\otimes_L D_{\cris,L}(V^*(1)) \ar[u]_{\simeq}\ar@{}[r]|{\times}& K\otimes_L L_n\otimes_L D_{\cris,L}(V(\tau^{-1}))  \ar[u]_{\simeq} & \\
%K[G_n]\otimes_L D_{\cris,L}(V^*(1))\ar@{}[r]|{\times}\ar[u]_{\Upsilon^n_{V^*(1)}} & K[G_n]\otimes_L D_{\cris,L}(V(\tau^{-1})) \ar[u]_{(\frac{q}{\pi_L})^{a(\cdot)}\sigma_{-1}\Upsilon^n_{V(\tau^{-1})}} \ar[r]^-{[[,]]'_{\cris,n}} &K[G_n],  \\
% }
%\end{equation*}\normalsize
%where $(\frac{q}{\pi_L})^{a(\cdot)}$ denotes multiplication of the $\rho$-component with $(\frac{q}{\pi_L})^{a(\rho)}.$
%I think the correct statement is
\begin{equation*}
\xymatrix{
K\otimes_L D_{\dR,L_n}^{\id}(V^*(1))\ar@{}[r]|{\times} & K\otimes_L  D_{\dR,L_n}^{\id}(V(\tau^{-1})) \ar[r]^-{[[,]]_{\mathrm{dR},L_n}} & K[G_n]\ar@{=}[dd]\\
K\otimes_L L_n\otimes_L D_{\cris,L}(V^*(1)) \ar[u]^{\simeq}\ar@{}[r]|{\times}  \ar[d]_{\Theta_{V^*(1),n}}& K\otimes_L L_n\otimes_L D_{\cris,L}(V(\tau^{-1}))  \ar[u]_{\simeq} \ar[d]^{\sigma_{-1}\Theta_{V(\tau^{-1}),n}^\ast} & \\
K[G_n]\otimes_L D_{\cris,L}(V^*(1))\ar@{}[r]|{\times} & K[G_n]\otimes_L D_{\cris,L}(V(\tau^{-1})) \ar[r]_-{[[,]]'_{\cris,n}} &K[G_n],
 }
\end{equation*}
where $\sigma_{-1} \in \Gamma_L$ denotes the element such that $\chi_{\rm LT}(\sigma_{-1})=-1$.
%$\Theta_{W,n}$ is the map defined in (\ref{def Theta}), and
\end{lemma}
\begin{proof}

It is sufficient to prove the commutativity of the following diagram:
\begin{equation*}
\xymatrix{
 L_n\otimes_L D_{\cris,L}(V^*(1)) \ar[d]_{\Theta_{V^*(1),n}} \ar@{}[r]|{\times} & L_n\otimes_L D_{\cris,L}(V(\tau^{-1}))  \ar[d]^{\sigma_{-1}\Theta_{V(\tau^{-1}),n}^\ast} \ar[r]& L[G_n]\ar@{^{(}->}[d] \\
K[G_n]\otimes_L D_{\cris,L}(V^*(1))\ar@{}[r]|{\times} & K[G_n]\otimes_L D_{\cris,L}(V(\tau^{-1})) \ar[r]_-{[[,]]'_{\cris,n}} &K[G_n],
 }
\end{equation*}
where the upper pairing is induced by the trace pairing
$$L_n \times L_n \to L[G_n]; \ (a,b)\mapsto \sum_{g\in G_n} {\rm Tr}_{L_n/L}(g(a)b)g^{-1}$$
and the natural pairing (see (\ref{def cris prime}))
$$[\cdot ,\cdot]_{\rm cris}': D_{{\rm cris},L}(V^\ast(1))\times D_{{\rm cris},L}(V(\tau^{-1})) \to D_{{\rm cris},L}(L(\chi_{\rm LT}))\simeq  L.$$
Recall that $[[\cdot,\cdot]]_{{\rm cris},n}'$ is defined by
$$[[\lambda \otimes x , \mu \otimes y]]_{{\rm cris},n}' = \lambda \iota(\mu) [x,y]_{{\rm cris}}',$$
where $\iota: K[G_n]\to K[G_n]; \ g\mapsto g^{-1}$ denotes the involution.

Take any $a\otimes x \in  L_n\otimes_L D_{\cris,L}(V^*(1))$ and $b\otimes y \in L_n\otimes_L D_{\cris,L}(V(\tau^{-1}))$. Let $\rho$ be a character of $G_n$. We first assume that $a(\rho)=0$, i.e., $\rho$ is the trivial character. In this case, we compute
\begin{align*}
& {\rm pr}_\rho \left( [[ \Theta_{V^\ast(1),n} (a\otimes x), \Theta^\ast_{V(\tau^{-1}),n} (b\otimes y) ]]_{{\rm cris},n}' \right)\\
&= [(1-\varphi_L)(1-q^{-1}\varphi_L^{-1})^{-1}x, (1-\pi_L^{-1}\varphi_L^{-1})^{-1}(1-\frac{\pi_L}{q}\varphi_L)y]_{{\rm cris}}' {\rm Tr}_{L_n/L}(a){\rm Tr}_{L_n/L}(b) \\
&= [x,y]_{\rm cris}' \sum_{g \in G_n}{\rm Tr}_{L_n/L}(g(a)b),
\end{align*}
where the second equality follows by noting that the adjoint of $\varphi_L$ is $\pi_L^{-1}\varphi_L^{-1}$ under $[\cdot,\cdot]_{\rm cris}'$
(note that the assumption $D_{\cris,L}(V^\ast(1))^{\varphi_L=1}=0$ is equivalent to $D_{\cris,L}(V(\tau^{-1}))^{\varphi_L=\pi_L^{-1}}=0$ and so $1-\pi_L^{-1}\varphi_L^{-1}$ is invertible on $D_{\cris,L}(V(\tau^{-1}))$).
This proves the commutativity of the diagram on the trivial character component.

Next, suppose that $a(\rho)\geq 1$. In this case we compute
\begin{align*}
& {\rm pr}_\rho \left( [[ \Theta_{V^\ast(1),n} (a\otimes x), \sigma_{-1}\Theta^\ast_{V(\tau^{-1}),n} (b\otimes y) ]]_{{\rm cris},n}' \right)\\
&= [\varphi_L^{a(\rho)}x,\varphi_L^{a(\rho)} y]_{{\rm cris}}' \tau(\rho)^{-1} q^{a(\rho)} \sum_{g \in G_{a(\rho)}}\rho(g)g^{-1}( {\rm Tr}_{L_n/L_{a(\rho)}}(a)) \\
&\quad  \times   \rho(\sigma_{-1})\tau(\rho^{-1})^{-1}\pi_L^{a(\rho)}\sum_{g\in G_{a(\rho)}} \rho(g)g({\rm Tr}_{L_n/L_{a(\rho)}}(b)) \\
&=[x,y]_{{\rm cris}}' \left(\sum_{g\in G_{a(\rho)}}\rho(g)g^{-1}({\rm Tr}_{L_n/L_{a(\rho)}}(a))\right) \left(\sum_{g\in G_{a(\rho)}}\rho(g)g({\rm Tr}_{L_n/L_{a(\rho)}}(b))\right) \\
&=[x,y]_{\rm cris}'  {\rm pr}_\rho \left( \sum_{g\in G_n} {\rm Tr}_{L_n/L}(g(a)b)g^{-1}\right),
\end{align*}
where the first equality follows from the definitions of $\Theta$ and $\Theta^\ast$ (see (\ref{def Theta}) and (\ref{def Theta ast}) respectively),
%and the pairing $[[\cdot,\cdot]]_{{\rm cris},n}'$,
the second from
$$[\varphi_L^{a(\rho)}x,\varphi_L^{a(\rho)} y]_{{\rm cris}}' = \pi_L^{-a(\rho)}[x,y]_{\rm cris}'$$
and
\begin{align*}
  \tau(\rho)\tau(\rho^{-1})\rho(\sigma_{-1})=q^{a(\rho)}
\end{align*}
(see \cite[Rem.\ 6.2, (61), (62)]{MSVW}\footnote{Let $\delta:(o_L/\pi_L^n)^\times\to K^\times$ be the character corresponding to $\rho.$ Then
$\tau(\rho)= \left(\sum_{i\in (o_L/\pi_L^n)^\times}\delta(i)\eta(i,\eta_n)\right)=  \epsilon_K(L,K(\delta^{-1}),\psi_{\eta},dx), $ where $\psi_\eta$ is defined in (loc.\ cit.) and satisfies $n(\psi_\eta)=0.$ Combine this with the duality relation
\begin{equation*}\label{f:epsduality2}
  \epsilon(L,\delta,\psi,dx)\epsilon(L,\delta^{-1}|-|,\psi(x),dx)=\delta(-1)
\end{equation*}
and twist relation
 \begin{equation*}\label{f:epsabsolutevalue}
   \epsilon(L,\delta^{-1}|-|,\psi(x),dx)=q^{-a(\delta) } \epsilon(L,\delta^{-1},\psi(x),dx)=|\pi_L^{a(\delta) }| \epsilon(L,\delta^{-1},\psi(x),dx).
\end{equation*}
 }),
and the last from the following calculation:
\begin{align*}
& \left(\sum_{g\in G_{a(\rho)}}\rho(g)g^{-1}({\rm Tr}_{L_n/L_{a(\rho)}}(a))\right) \left(\sum_{g\in G_{a(\rho)}}\rho(g)g({\rm Tr}_{L_n/L_{a(\rho)}}(b))\right)\\
&= \left(\sum_{g\in G_{n}}\rho(g)g^{-1}(a)\right) \left(\sum_{g\in G_{n}}\rho(g)g(b)\right)=  \sum_{g,h \in G_n} \rho(gh)g^{-1}(a)h(b) \\
&= \sum_{g,h\in G_n} \rho(g^{-1}) g h(a) h(b) = {\rm pr}_\rho \left( \sum_{g \in G_n} \Big( \sum_{h\in G_n} h (g(a)b)  \Big)g^{-1}\right)\\
&= {\rm pr}_\rho \left( \sum_{g\in G_n} {\rm Tr}_{L_n/L}(g(a)b)g^{-1}\right).
\end{align*}

%a simple computation {\color{blue} can this be reduced to some orthogonality relation or how does this work?}.
Hence we have
$$[[ \Theta_{V^\ast(1),n} (a\otimes x),\sigma_{-1} \Theta^\ast_{V(\tau^{-1}),n} (b\otimes y) ]]_{{\rm cris},n}' = [x,y]_{\rm cris}' \sum_{g \in G_n} {\rm Tr}_{L_n/L}(g(a)b)g^{-1}.$$
This shows the commutativity of the diagram.
\end{proof}

Note that under the assumption that $V( \tau^{-1})$    does not have any quotient isomorphic to the trivial representation $L$, there
is a commutative diagram
\begin{equation}
\label{f:corpr}\xymatrix{
  H^1_{\Iw}(L_\infty/L,T) \ar[d]_{\cor} \ar[r]_{{\rm Exp}^\ast}^-{\simeq } &
  D_{\LT}(T(\tau^{-1}))^{\psi=1}\ar[d]_{\pr_{\Gamma_n}} \ar@{^(->}[r]^{ } &
  D_{\rig}^\dagger(V(\tau^{-1}))^{\psi=1} \ar[d]^{\pr_{\Gamma_n}} \\
  H^1(L_n,V) \ar@{=}[r]^{ } & H^1(L_n,V) \ar@{->>}[r]^{ } & H^1_{/\dagger}(L_n,V), }
\end{equation}
as explained before \cite[Thm.\ 5.2.26]{SV24}.
%}

Let $y_{\chi_{\LT}^{-j},n}$ denote the image of $y \in  H^1_{\Iw}(L_\infty/L,T)$ under the map
 \begin{equation}\label{f:twistprojection}
 \mathrm{pr}_{n,-j}: H^1_{\Iw}(L_\infty/L,T)\xrightarrow{ \cdot \otimes \eta^{\otimes -j}
 }H^1_{\Iw}(L_\infty/L,T(\chi_{\LT}^{-j}))\xrightarrow{\mathrm{cor}}H^1 (
 L_n,T(\chi_{\LT}^{-j}))\to H^1 ( L_n,V(\chi_{\LT}^{-j})).
 \end{equation}
 Recall that the map
 $$\widetilde{\exp}^*_{L_n,V(\chi_{\LT}^{-j}),\id}: H^1 ( L_n,V(\chi_{\LT}^{-j})) \to L_n \otimes_L D_{{\rm cris},L}(V(\tau^{-1}\chi_{\rm LT}^{-j}))$$
 satisfies $\widetilde{\exp}^*_{L_n,V(\chi_{\LT}^{-j}),\id} \otimes \mathbf{d}_1 =\exp_{L_n,V(\chi_{\rm LT}^{-j})}^\ast$. Similarly, the map
$$\widetilde{\log}_{L_n,V(\chi_{\LT}^{-j}),\id}: H_f^1 ( L_n,V(\chi_{\LT}^{-j})) \to L_n \otimes_L D_{{\rm cris},L}(V(\tau^{-1}\chi_{\rm LT}^{-j}))$$
 satisfies $\widetilde{\log}_{L_n,V(\chi_{\LT}^{-j}),\id} \otimes \mathbf{d}_1 =\log_{L_n,V(\chi_{\rm LT}^{-j})}$. Recall also that the map
$$\Theta_{W,n}^\ast:  K\otimes_L L_n\otimes_L D_{\cris,L}(W)\to  K[G_n]\otimes_L D_{\cris,L}(W)$$
is defined in (\ref{def Theta ast}).

The following result generalizes \cite[Thm.~A.2.3]{LVZ15} and \cite[Thm.\ B.5]{LZ}  from the cyclotomic  case as well as \cite[Thm.\ 5.2.26]{SV24} in the Lubin-Tate case:
\begin{theorem}\label{thm:adjointformulan} Assume that $V^*(1)$ is $L$-analytic crystalline  $L$-linear representation of $G_L$ with $\mathrm{Fil}^{-1}D_{\cris,L}(V^*(1))=D_{\cris,L}(V^*(1))$ and $D_{\cris,L}(V^*(1))^{\varphi_L=\pi_L^{-1}}=D_{\cris,L}(V^*(1))^{\varphi_L=1}=0$. %\linebreak and $D_{\cris,L}(V^*(1))^{\varphi_L=\pi_L^{-1}}=D_{\cris,L}(V^*(1))^{\varphi_L=1}=0$.
Then it holds  for $y\in H^1_{\rm Iw}(L_\infty/L,T)$ and $j\in \ZZ$ that
\begin{align*}
 \Omega^{j} \mathbf{L}_V(y)(\chi_{\LT}^j,n)=\begin{cases}
 \displaystyle  j!\Theta^\ast_{V(\tau^{-1}\chi_{\rm LT}^{-j}),n}
  \Big( \widetilde{\exp}^*_{L_n,V(\chi_{\LT}^{-j}),\id}(y_{\chi_{\LT}^{-j},n})\Big) \otimes e_{ j} &\text{if $j\geq 0$,}\\
 \displaystyle   \frac{{ (-1)^{j+1}}}{(-1-j)!}  \Theta^\ast_{V(\tau^{-1}\chi_{\rm LT}^{-j}),n}
  \Big( \widetilde{\log}_{L_n,V(\chi_{\LT}^{-j}),\id}(y_{\chi_{\LT}^{-j},n})\Big) \otimes e_{ j} &\text{if $j\leq -1$,}
  \end{cases}
\end{align*}
if the operators $1-\pi_L^{-1}\varphi_L^{-1}, 1-\frac{\pi_L}{q}\varphi_L$ are invertible on  $D_{\cris,L}(V(\tau^{-1}\chi_{\LT}^{-j}))$ (or equivalently, $1-\pi_L^{-1-j}\varphi_L^{-1}, 1-\frac{\pi_L^{j+1}}{q}\varphi_L $ are invertible on $D_{\cris,L}(V(\tau^{-1}))$).
\end{theorem}
\begin{proof}
%The case $n=0$ being proven in \cite[Thm.\ 5.2.26]{SV24} we assume $n\geq 1.$
From the reciprocity formula in Theorem \ref{cor:adjoint} (for the second equation),  Propositions \ref{prop:Taten} (for the third equation) and \ref{prop:crisn} (for the first equation) we obtain for $x\in D(\Gamma_L,\mathbb{C}_p)\otimes_L D_{\cris,L}(V^*(1)),$ $y\in D(V(\tau^{-1}))^{\psi_L=1}  $ and $j\geq 0$ using \eqref{f:corpr}
\begin{align*}
\notag[[x(\chi_{\LT}^{-j},n)&\otimes e_{j} , (-1)^{j}{\sigma_{-1}}\mathbf{L}_V(y)(\chi_{\LT}^{j},n)\otimes e_{-j}]]'_{\cris,n}\\
&={\Omega}
[x,\frac{\sigma_{-1}\mathbf{L}_V(y)}{\Omega}]^0(\chi_{\LT}^{-j},n)\\
&={\Omega} { \frac{q-1}{q}}\{\mathbf{\Omega}_{{V^*(1)},1}(x),y\}_{\Iw}(\chi_{\LT}^{-j},n)\\
&={ \Omega} \langle\langle h^1_{L_n}\circ \notag \mathrm{tw}_{\chi_{\LT}^{j}}\left(\mathbf{\Omega}_{{V^*(1)},1}(x)\right),y_{\chi_{\LT}^{-j},n}\rangle\rangle_{\Tate}\\
\notag&={ \Omega} \langle\langle(-1)^{j}j! \Omega^{-j-1}
 \exp_{L_n,V^*(1)(\chi_{\LT}^j)}\Big(
 %\mathfrak{C}_n^{-1}\circ\Upsilon^n
 {\Theta_{V^\ast(1)(\chi_{\rm LT}^{j}),n}^{-1}}
 \left(x(\chi_{\LT}^{-j},n)  \otimes e_j\right)\Big),y_{\chi_{\LT}^{-j},n}\rangle\rangle_{\Tate}\\
\notag&=(-1)^{j} \Omega^{ -j}j![[
%\mathfrak{C}_n^{-1}\circ\Upsilon^n
 {\Theta_{V^\ast(1)(\chi_{\rm LT}^{j}),n}^{-1}}
\left(x(\chi_{\LT}^{-j},n)  \otimes e_j\right),\widetilde{\exp}^*_{L_n,V(\chi_{\LT}^{-j}),\id}(y_{\chi_{\LT}^{-j},n})]]_{\dR,L_n}
\\
\notag&=[[ x(\chi_{\LT}^{-j},n)\otimes e_{j},(-1)^{j}\Omega^{ -j}j!
% (\pi_L/q)^{a(\cdot)}%(\sigma_{-1}\Upsilon^n)^{-1}\circ(\mathfrak{C}_{n,0}^*)^{-1}
{ \sigma_{-1}\Theta^\ast_{V(\tau^{-1}\chi_{\rm LT}^{-j}),n}}
\Big( \widetilde{\exp}^*_{L_n,V(\chi_{\LT}^{-j}),\id}(y_{\chi_{\LT}^{-j},n})\Big)]]'_{\cris,n}.
\end{align*}
Here we used \eqref{f:interjgeq1mod} in the fourth equation for the interpolation property of $\mathbf{\Omega}_{{V^*(1)},1}. $  The fifth equation is the defining equation for the dual exponential map resulting from \eqref{f:dualexpId}. Furthermore, for the last equality we use Lemma \ref{lem:dR-pairing-cris} (note that we assume $1-\pi_L^{-1}\varphi_L^{-1}$ and $ 1-\frac{\pi_L}{q}\varphi_L$ are invertible on  $D_{\cris,L}(V(\tau^{-1}\chi_{\LT}^{-j}))$, which means $D_{\cris,L}(V^\ast(1)(\chi_{\LT}^j))^{\varphi_L=q^{-1}}= D_{\cris,L}(V^\ast(1)(\chi_{\LT}^j))^{\varphi_L=1}=0$).
% and the fact that $ \pi_L^{-1}\varphi_L^{-1}$ is adjoint to $\varphi_L$ under the lower pairing.
The claim follows since the evaluation map is surjective and $[[\;,\;]]'_{\cris,n}$ is non-degenerate.
%\Footnote{Implicity we have used in the calculation that the operators $(1-q^{-1}\varphi_L^{-1}), \; (1-\varphi_L) $ are invertible, this can perhaps be avoided by a better arrangement!? }

Now assume that $j<0$. We similarly have the following:
\begin{align*}
\notag[[x(\chi_{\LT}^{-j},n)&\otimes e_{j} , (-1)^{j}{ \sigma_{-1}}\mathbf{L}_V(y)(\chi_{\LT}^{j},n)\otimes e_{-j}]]'_{\cris,n}\\
&={ \Omega}
[x,\frac{\sigma_{-1}\mathbf{L}_V(y)}{{ \Omega}}]^0(\chi_{\LT}^{-j},n)\\
&={ \Omega} { \frac{q-1}{q}}\{\mathbf{\Omega}_{{V^*(1)},1}(x),y\}_{\Iw}(\chi_{\LT}^{-j},n)\\
&={ \Omega} \langle\langle h^1_{L_n}\circ \notag \mathrm{tw}_{\chi_{\LT}^{j}}\left(\mathbf{\Omega}_{{V^*(1)},1}(x)\right),y_{\chi_{\LT}^{-j},n,}\rangle\rangle_{\Tate}\\
\notag&={ \Omega} \langle\langle \frac{1}{(-1-j)!} \Omega^{-j-1}\log_{L_n,V^\ast(1)(\chi_{\LT}^j)}^*\Big(
  %\mathfrak{C}_n^{-1}\circ\Upsilon^n
   {\Theta_{V^\ast(1)(\chi_{\rm LT}^{j}),n}^{-1}}
  \left(x(\chi_{\LT}^{-j},n)  \otimes e_j\right)\Big),y_{\chi_{\LT}^{-j},n}\rangle\rangle_{\Tate}\\
\notag&=\frac{\Omega^{ -j}}{(-1-j)!}[[
%\mathfrak{C}_n^{-1}\circ\Upsilon^n
 {\Theta_{V^\ast(1)(\chi_{\rm LT}^{j}),n}^{-1}}
\left(x(\chi_{\LT}^{-j},n)  \otimes e_j\right),  -\widetilde{\log}_{L_n,V(\chi_{\LT}^{-j}),\id}(y_{\chi_{\LT}^{-j},n})]]_{\dR,L_n}
\\
\notag&=[[ x(\chi_{\LT}^{-j},n)\otimes e_{j},\frac{-\Omega^{ -j}}{(-1-j)!}
%(\pi_L/q)^{a(\rho)}
%(\sigma_{-1}\Upsilon^n)^{-1}\circ(\mathfrak{C}_{n,0}^*)^{-1}
{\sigma_{-1}\Theta^\ast_{V(\tau^{-1}\chi_{\rm LT}^{-j}),n}}
\Big(\widetilde{\log}_{L_n,V(\chi_{\LT}^{-j}),\id}(y_{\chi_{\LT}^{-j},n})\Big)]]_{\cris,n}'.
\end{align*}
\end{proof}
\begin{remark}\label{rem:int}
The last condition in the theorem is also equivalent to the requirement that the operators $1-q\varphi_L$ and $1-\varphi_L$ are invertible on $D_{\cris,L}(V^\ast(1)(\chi_{\LT}^j))$ (or equivalently, $1-\frac{q}{\pi_L^j}\varphi_L$ and $1-\pi_L^{-j}\varphi_L$ are invertible on  $D_{\cris,L}(V^\ast(1))$).
\end{remark}

As a consequence we obtain the following interpolation property extended to Artin characters:

\begin{corollary}\label{cor:adjointformula}
Assume that $V^*(1)$ is $L$-analytic with $\mathrm{Fil}^{-1}D_{\cris,L}(V^*(1))=D_{\cris,L}(V^*(1))$  and such that $D_{\cris,L}(V^*(1))^{\varphi_L=\pi_L^{-1}}=D_{\cris,L}(V^*(1))^{\varphi_L=1}=0$.
 Let $j\in\mathbb{Z}$   and let $\rho$ be a finite order character of $\Gamma_L$ of conductor $n,$ i.e., $\rho$ factorizes over $G_n,$ but not over $G_{n-1}.$ Then for any $y\in H^1_{\rm Iw}(L_\infty/L,T)$ we have the following:
\begin{enumerate}
\item If $j\geq 0,$ we have \small
\begin{align*}
   \Omega^{j} \mathbf{L}_V(y)(\rho\chi_{\LT}^j)&=  j!
\left\{
  \begin{array}{ll}
\displaystyle    (1-\pi_L^{-1-j}\varphi_L^{-1})^{-1}(1-\frac{\pi_L^{j+1}}{q}\varphi_L)\Big(\widetilde{\exp}^*_{L,V(\chi_{\LT}^{-j}),\id}(y_{\chi_{\LT}^{-j}})\otimes e_{ j}\Big), & \hbox{if  $n=0$,} \\
  %{[L_n:L]  }
 \displaystyle   \pi_L^{n(1+j)}\tau(\rho)^{-1} \varphi_L^{n}  \sum_{\sigma \in G_n}\rho(\sigma)\Big( \widetilde{\exp}^*_{L_n,V(\chi_{\LT}^{-j}),\id}(\sigma^{-1} (y_{\chi_{\LT}^{-j},n}))  \otimes e_{ j}\Big) , & \hbox{if $n>0$.}
  \end{array}
\right.
\end{align*}\normalsize
\item If $j\leq -1,$ we have\small
\begin{align*}
 \Omega^{j} \mathbf{L}_V(y)(\rho\chi_{\LT}^j)&= \frac{{ (-1)^{j+1}}}{(-1-j)!}
\left\{
  \begin{array}{ll}
 \displaystyle    (1-\pi_L^{-1-j}\varphi_L^{-1})^{-1}(1-\frac{\pi_L^{j+1}}{q}\varphi_L)\Big(\widetilde{\log}_{L,V(\chi_{\LT}^{-j}),\id}(y_{\chi_{\LT}^{-j}})\otimes e_{ j}\Big), & \hbox{if  $n=0$,} \\
 %   {[L_n:L]  }
 \displaystyle     \pi_L^{n(1+j)}\tau(\rho)^{-1} \varphi_L^{n}  \sum_{\sigma \in G_n}\rho(\sigma)\Big(\widetilde{\log}_{L_n,V(\chi_{\LT}^{-j}),\id}(\sigma^{-1}(y_{\chi_{\LT}^{-j},n}))  \otimes e_{ j}\Big),  & \hbox{if $n>0$.}
  \end{array}
\right.
\end{align*}
\end{enumerate}\normalsize
In both cases, we assume that $1-\pi_L^{-1-j}\varphi_L^{-1}, 1-\frac{\pi_L^{j+1}}{q}\varphi_L $ are invertible on $D_{\cris,L}(V(\tau^{-1}))$ when $\rho$ is the trivial character, i.e., $n=0.$ In addition,  we have
\[\sum_{\sigma \in G_n}\rho(\sigma)\Big( \widetilde{\exp}^*_{L_n,V(\chi_{\LT}^{-j}),\id}(\sigma^{-1} (y_{\chi_{\LT}^{-j},n}))  \otimes e_{ j}\Big)=\widetilde{\exp}^*_{L,V(\chi_{\LT}^{-j}\rho^*),\id}( y_{\chi_{\LT}^{-j}\rho^*,n})  \otimes e_{ j}\]
and
\[\sum_{\sigma \in G_n}\rho(\sigma)\Big(\widetilde{\log}_{L_n,V(\chi_{\LT}^{-j}),\id}(\sigma^{-1}(y_{\chi_{\LT}^{-j},n}))  \otimes e_{ j}\Big)=\widetilde{\log}_{L,V(\chi_{\LT}^{-j}\rho^*),\id}(\sigma^{-1}(y_{\chi_{\LT}^{-j}\rho^*,n}))  \otimes e_{ j}\]
where $ y_{\chi_{\LT}^{-j}\rho^*,n}$ is defined analogously as in \eqref{f:twistprojection}.
\end{corollary}
\begin{proof}
Note that $\varphi_L(d)\otimes e_j=\pi_L^{j}\varphi_L(d\otimes e_j)$ for any $d\in D_{{\rm cris},L}(V(\tau^{-1}\chi_{\rm LT}^{-j}))$. Then the claim follows immediately from Theorem \ref{thm:adjointformulan} and the definition of $\Theta^\ast$ in (\ref{def Theta ast}). The addition follows analogously as in \cite[Lem.\ B.4]{LZ}.
%  The formulae follow from those in Theorem \ref{thm:adjointformulan} by applying the  idempotents corresponding to $\rho$. For $n=0$ there is nothing to show, so we assume $n>0$. By definition (see (\ref{def Theta})), the map
%$$\Theta_{V(\tau^{-1}),n}: K\otimes_L L_n \otimes_{L} D_{{\rm cris},L}(V(\tau^{-1})) \to K[G_n] \otimes_L D_{{\rm cris},L}(V(\tau^{-1}))$$
%satisfies
%$$({\rm pr}_\rho\otimes \id)\circ \Theta_{V(\tau^{-1}),n} (x)= \tau(\rho)^{-1}q^n\varphi_L^n \sum_{\sigma \in G_n}\rho(\sigma)\sigma^{-1} (x)$$
%for any $x\in K\otimes_L L_n \otimes_L D_{{\rm cris},L}(V(\tau^{-1}))$, where
%$${\rm pr}_\rho: K[G_n] \to K$$
%is the ring homomorphism induced by $\rho$. The claim follows immediately from this.
  % Then the contribution of the inverse of $\mathfrak{C}_{n,j}^* $ is the factor $[L_n:L]^{-1}$ {\color{red} here is a sign issue or other problem, which we discussed earlier, didn't we?}, while the inverse of $\Upsilon^n_{}$ contributes $\tau(\rho)^{-1}q^{n}\varphi^n$ according to \eqref{f:upsrho}. Now by the analogue of \cite[Lem.\ B.4.]{LZ} we have $\mathfrak{e}_\rho \widetilde{\exp}^*_{L_n,?}(?)=\frac{1}{[L_n:L]}\widetilde{\exp}^*_{L,?}(?)$ and similar for $\widetilde{\log}.$
\end{proof}

%Note that due to the relation $ \mathbf{L}_V(y)\otimes\mathbf{d}_1=\cL_V(y) $ the results in Theorem \ref{thm:adjointformulan} and Corollary \ref{cor:adjointformula} can also rephrased in terms of $\cL_V$. Moreover, we obtain the following interpolation property for $\mathcal{L}_{L(\chi_{\rm LT}^{-k})(1)}$ for $k\geq 1.$

If $V$ is of the form $L(\chi_{\rm LT}^{-k})(1)$ with $k\geq 1$, we can remove the assumption of Theorem \ref{thm:adjointformulan}.

\begin{corollary}\label{f:intchi-j}
Let $k\geq 1$ and $y \in H^1_{\rm Iw}(L_\infty/L, o_L(\chi_{\rm LT}^{-k})(1))$. For $j'\geq 0$ we have
\begin{align*}
\displaystyle  \mathbf{L}_{L(\chi_{\rm LT}^{-k})(1)}(y)(\chi_{\LT}^{j'},n)&=\frac{j'!}{\Omega^{j'}}\Theta^\ast_{L(\chi_{\rm LT}^{1-k-j'}),n}
\Big( \widetilde{\exp}^*_{L_n,L(\chi_{\LT}^{-k-j'})(1),\id}(y_{\chi_{\LT}^{-j'},n})\Big)\otimes e_{ j'} ,
\end{align*}
%\footnotesize
%\begin{align*}
%\displaystyle  \mathbf{L}_{L(\chi_{\rm LT}^{-k})(1)}(y)(\chi_{\LT}^{j'},n)&=\frac{j'!}{\Omega^{j'}}
%\left\{
%  \begin{array}{ll}
%    (1-\pi_L^{-1-j'}\varphi_L^{-1})^{-1}(1-\frac{\pi_L^{j'+1}}{q}\varphi_L) \Big(\widetilde{\exp}^*_{L,L(\chi_{\LT}^{-k-j'})(1),\id}(y_{\chi_{\LT}^{-j'}})\otimes e_{ j'} \Big), & \hbox{if  $n=0$,} \\
%\displaystyle    (\pi_L/q)^{a(\cdot)} \pi_L^{nj'}
%%( \Upsilon^n)^{-1}\circ(\mathfrak{C}_{n,j'}^*)^{-1}
%\Theta_{L(\chi_{\rm LT}^{1-k}),n}
% \Big( \widetilde{\exp}^*_{L_n,L(\chi_{\LT}^{-k-j'})(1),\id}(y_{\chi_{\LT}^{-j'},n})\otimes e_{ j'} \Big), & \hbox{if $n>0$,}
%  \end{array}
%\right.
%\end{align*}\normalsize
and, for any finite order character $\rho$ of $\Gamma_L$ of conductor $n,$\footnotesize
\begin{align*}
\mathbf{L}_{L(\chi_{\rm LT}^{-k})(1)}(y)(\rho\chi_{\LT}^{j'})&=\frac{j'!}{\Omega^{j'}}
\left\{
  \begin{array}{ll}
\displaystyle(1-\pi_L^{-1-j'}\varphi_L^{-1})^{-1}(1-\frac{\pi_L^{j'+1}}{q}\varphi_L) \Big(\widetilde{\exp}^*_{L,L(\chi_{\LT}^{-k-j'})(1),\id}(y_{\chi_{\LT}^{-j'}})\otimes e_{ j'} \Big), & \hbox{if  $n=0$,} \\
%  {[L_n:L]  }
\displaystyle \pi_L^{n(1+j')}\tau(\rho)^{-1} \varphi_L^{n} \sum_{\sigma\in G_n}\rho(\sigma)\Big( \widetilde{\exp}^*_{L_n,L(\chi_{\LT}^{-k-j'})(1),\id}(\sigma^{-1}(y_{\chi_{\LT}^{-j'},n})) \otimes  e_{ j'} \Big), & \hbox{if $n>0$.}
  \end{array}
\right.
\end{align*}\normalsize
\end{corollary}
\begin{proof}
Indeed, $V= L(\chi_{\rm LT})(1)=L(\chi_{\rm LT}^2\tau)$ does satisfy the assumptions in Theorem \ref{thm:adjointformulan} (for any $j \in \ZZ$, the representation $L(\chi_{\rm LT}^j)$ is $L$-analytic by \cite[Lem.\ 2 on page III.45]{Se0} or with an explicit construction of periods by \cite[Prop. 1 in \S 3.2]{fourquaux}% combined with Galois descent
), which we apply for $j=j'+k+1.$ Hence the claim   follows from  the commutativity of \eqref{f:regulatortwist} and of the following diagram for $j\geq k+1$ as a consequence of \eqref{f:evlfactor}:
\footnotesize
\begin{equation}\label{f:Gammafactor}
 \xymatrix{
  \frac{1}{\frak{l}'_{L(\chi_\LT^{k+1})}}D(\Gamma_L,K)\otimes_L D_{\mathrm{cris}}(L(\chi_{\rm LT}^{-k})(1))\ar[d]^-{\mathrm{pr}_{\Gamma_m}\circ \mathrm{Tw}_{\chi_\LT^{j-k-1}}\otimes\id}  \ar[rrr]^-{\frac{\frak{l}_{L(\chi_\LT^{k+1})}\mathrm{Tw}_{\chi_\LT^{-(k+1)}}}{\Omega^{k+1}}\otimes \id \otimes e_{k+1}}  &&  & D(\Gamma_L,K)\otimes_L D_{\mathrm{cris}}(L(\chi_{\rm LT} )(1))\ar[d]^-{\mathrm{pr}_{\Gamma_m}\circ \mathrm{Tw}_{\chi_\LT^{j}}\otimes\id}\\
 K[\Gamma_L/\Gamma_m]\otimes_L D_{\mathrm{cris}}(L(\chi_{\rm LT}^{-k})(1)) \ar[rrr]^{\frac{j!}{\Omega^{k+1}(j-k-1)!}\otimes \id \otimes e_{k+1}}   &&&K[\Gamma_L/\Gamma_m]\otimes_LD_{\mathrm{cris}}(L(\chi_{\rm LT} )(1)).  }
\end{equation}
%\begin{equation}\label{f:regulator}
% \xymatrix@C=0.1cm{
%   &  D(\Gamma_L,K)\otimes_L D_{\mathrm{cris}}(L(\chi_{\rm LT} )(1))\ar[dd]^-{\mathrm{pr}_{\Gamma_m}\circ \mathrm{Tw}_{\chi_\LT^{j+1}}\otimes\id}\\
%        \frac{1}{\frak{l}'_{L(\chi_\LT^{j+1})}}D(\Gamma_L,K)\otimes_L D_{\mathrm{cris}}(L(\chi_{\rm LT}^{-j})(1))\ar[dd]^-{\mathrm{pr}_{\Gamma_m}\otimes\id}  \ar[ru]^(0.4){{\Omega^{-(j+1)}{\frak{l}_{L(\chi_\LT^{j+1})}\mathrm{Tw}_{\chi_\LT^{-(j+1)}}}}\otimes \id \otimes e_{j+1}} &  \\
%     &K[\Gamma_L/\Gamma_m]\otimes_LD_{\mathrm{cris}}(L(\chi_{\rm LT} )(1)) \\
%      K[\Gamma_L/\Gamma_m]\otimes_L D_{\mathrm{cris}}(L(\chi_{\rm LT}^{-j})(1)) \ar[ru]^{\frac{(j+1)!}{\Omega^{j+1}}\otimes \id \otimes e_{j+1}} .&  }
%\end{equation}
\normalsize
Note that $1-q\varphi_L$ and $1-\varphi_L$ are invertible on  $D_\cris(L(\chi_\LT^k))$ whence the interpolation formula holds at all $j'\geq 0$ by Remark \ref{rem:int}.
\end{proof}

%{\color{purple}
%
%\begin{lemma}\label{lem:Lan}
%For any $j \in \ZZ$, the representation $L(\chi_{\rm LT}^j)$ is $L$-analytic.
%\end{lemma}
%
%\begin{proof}
%Let $\sigma: L\hookrightarrow \CC_p$ be an embedding with $\sigma\neq \id$. We check that $\CC_p\otimes_{L,\sigma} L(\chi_{\rm LT}^j)$ is a trivial semilinear $\CC_p$-representation of $G_L$.
%
%By \cite[Prop. 1 in \S 3.2]{fourquaux}, there exists $\xi_\sigma \in \CC_p^\times$ such that
%\begin{equation}\label{xi fou}
%\frac{\xi_\sigma}{g(\xi_\sigma)} =\sigma(\chi_{\LT}(g)) \text{ for any $g\in G_L$}
%\end{equation}
%(note that $L/\QQ_p$ is a Galois extension). We claim that
%$$\CC_p\to \CC_p\otimes_{L,\sigma}L(\chi_{\rm LT}^j) ; \ a\mapsto \xi_\sigma^j a \otimes \eta^{\otimes j}$$
%is an isomorphism of $\CC_p$-representations of $G_L$. In fact, this is an isomorphism of $\CC_p$-vector spaces, so it is sufficient to show that the map is $G_L$-equivariant. For any $g\in G_L$, we have
%\begin{align*}
%g\cdot (\xi_{\sigma}^j a \otimes \eta^{\otimes j}) &= g(\xi_{\sigma}^j a)\otimes \chi_{\LT}^j(g)\eta^{\otimes j}\\
%&= g(\xi_{\sigma}^j) g(a) \sigma(\chi_\LT^j(g)) \otimes \eta^{\otimes j}\\
%&= \xi_\sigma^j g(a)\otimes \eta^{\otimes j},
%\end{align*}
%where the last equality follows from (\ref{xi fou}). This proves the claim.
%\end{proof}
%}

\section{Kato's explicit reciprocity law}

In this section, we prove Kato's explicit reciprocity law, i.e., Theorem \ref{kato erl}. As explained in Introduction, it is sufficient to prove Theorem \ref{key claim}. We give a proof of this theorem in \S \ref{sec:claim}. Our argument actually proves a variant for a more general $p$-adic representation (``Colmez' reciprocity law in the Lubin-Tate setting", see Theorem \ref{thm:Colmez}), which we prove in \S \ref{sec colmez}.

In \S \ref{sec second}, we prove Kato's explicit reciprocity law by a different argument, which does not rely on Theorem \ref{key claim}.
%In this argument, we apply Corollary \ref{f:intchi-j} and use Theorem \ref{thm sv}.

\subsection{The first proof}\label{sec:claim}

In the following we often abbreviate $D_{\cris, L}$ to $D_\cris$.
Theorem \ref{key claim} is equivalent to the commutativity of the following two diagrams for any $j,m\geq 1$:
\begin{equation}\label{f:twist}
   \xymatrix{
     H^1_{\Iw}(L_\infty{/L},\TLT^{\otimes -1}(1)) \ar[d]_{\mathrm{Exp}^*} \ar[r]^{\mathrm{tw}_{j-1}} & H^1_{\Iw}(L_\infty{/L},\TLT^{\otimes -j}(1)) \ar[d]^{\mathrm{Exp}^*} \\
     \mathbf{A}_L^{\psi=1} \ar[d]_{ \pi_L^{-m}\varphi_L^{-m}\otimes \mathbf{d}_1} \ar[r]^{\otimes \eta^{\otimes 1-j}} &  \mathbf{A}_L^{\psi=1}(\chi_{\rm LT}^{1-j}) \ar[d]^{ \pi_L^{-m}\varphi_L^{-m}\otimes \mathbf{d}_1} \\
    L_m((t_{\rm LT}))\otimes D_{\mathrm{cris}}(L(\tau)) \ar[d]_{c_{t_{\rm LT}^{j-1}}\otimes\id} \ar[r]^{\cdot t_{\rm LT}^{1-j}\otimes e_{1-j}} & L_m((t_{\rm LT}))  \otimes D_{\mathrm{cris}}(L(\chi_{\rm LT}^{-j})(1)) \ar[d]^{c_{t_{\rm LT}^{0}}\otimes \id} \\
     L_m \otimes D_{\mathrm{cris}}(L(\tau)) \ar[r]^{\id\otimes  e_{1-j}} & L_m   \otimes D_{\mathrm{cris}}(L(\chi_{\rm LT}^{-j})(1)), }
\end{equation}
where  $c_{t_{\rm LT}^{j-1}} $ denotes taking the coefficient of $t_{\rm LT}^{j-1} $ (note that $\mathrm{ev}_{m,j}=c_{t_{\rm LT}^{j-1}}\circ\pi_L^{-m}\varphi_L^{-m})$ and $  \varphi_L^{-m}:\mathbf{A}_L^{\psi=1}(\chi_{\rm LT}^{1-j})\to L_m((t_{\rm LT}))\otimes_L D_{\mathrm{cris}}(L(\chi_{\rm LT}^{1-j})) $ is defined via the  inclusions
\[(\mathbf{A}_L(\chi_{\rm LT}^{1-j}))^{\psi=1}=(Z^{-2}\mathbf{A}_L^+(\chi_{\rm LT}^{1-j}))^{\psi=1}\]
 - by \cite[Lem.\ 3.3.4, 3.1.11/12]{SV24} -, \footnote{ $j=1$ is also covered by the next inclusion because $\frac{1}{\omega_{\rm LT}}=\frac{?}{t_{\rm LT}}$ with $?\in \cR^+_L$!}
\[(\mathbf{A}_L^+(\chi_{\rm LT}^{1-j}))^{\psi=1}\subseteq D_{\mathrm{rig}}^\dagger(L(\chi_{\rm LT}^{1-j}))^{\psi=1} \subseteq \cR_L^+[\frac{1}{t_{\rm LT}}]\otimes_L D_\mathrm{cris}(L(\chi_{\rm LT}^{1-j})),\]
where the second arrow stems from \cite[Thm.\ 3.1.1]{BF} involving  Kisin's and Ren's
comparison isomorphism,  and the map $\varphi_q^{-m}$ defined before \cite[Lem.\ 2.4.3]{BF}.
The second diagram is the following:
\begin{equation}\label{f:interpolation}
 \xymatrix{
   H^1_{\Iw}(L_\infty{/L},\TLT^{\otimes -j}(1)) \ar[dd]_{\mathrm{pr}_m} \ar[r]^-{\mathrm{Exp}^*} &  \mathbf{A}_L^{\psi=1}(\chi_{\rm LT}^{1-j})  \ar[d]^{\pi_L^{-m}\varphi_L^{-m}\otimes \mathbf{d}_1}   \\
   &   L_m((t_{\rm LT}))\otimes_L D_{\mathrm{cris}}(L(\chi_{\rm LT}^{-j})(1)) \ar[d]^{c_{t_{\rm LT}^{0}}\otimes \id} \\
   H^1(L_m,\TLT^{\otimes -j}(1))\ar[r]^{\exp_j^*}&   L_m \otimes_L D_{\mathrm{cris}}(L(\chi_{\rm LT}^{-j})(1)).  }
\end{equation}

The commutativity of the first one is easily checked using the commutativity
\begin{equation*}
  \xymatrix{
    \cR^+_K[\frac{1}{t_\LT}]\otimes_L D_{\mathrm{cris}}(V) \ar[d]_{t_\LT\otimes (\id\otimes t_\LT^{-1}\otimes \eta)} \ar[r]^{\simeq} &  \cR^+_K[\frac{1}{t_\LT}]\otimes_{\mathbf{A}^+_L} N(V) \ar[d]^{ \id\otimes({\rm incl}\otimes\eta)} \\
    \cR^+_K[\frac{1}{t_\LT}]\otimes_L D_{\mathrm{cris}}(V(\chi_\LT))\ar[r]^{\simeq} &    \cR^+_K[\frac{1}{t_\LT}]\otimes_{\mathbf{A}^+_L} N(V(\chi_\LT))   }
\end{equation*}
for the Kisin-Ren module $N(o_L(\chi_\LT^k))=Z^{-k}\mathbf{A}^+_L\otimes \eta^k$ and its comparison isomorphism attached to the presentation $V=o_L(\chi_\LT^k)$ as in \cite[Rem.\ 3.1.11, Lem.\ 3.1.12]{SV24}. Hence, {\it it is sufficient to show the commutativity of the second diagram (\ref{f:interpolation})}. To this aim we shall apply the interpolation property of an attached regulator map as a consequence of  Theorem \ref{thm:adjointformulan}.

\begin{proof}[Proof of commutativity of \eqref{f:interpolation}]
Instead of proving the commutativity of \eqref{f:interpolation} directly we extend it to the following larger diagram, in which the upper line is just the regulator map $\mathcal{L}_{L(\chi_{\rm LT}^{-j})(1)} $ (see \S \ref{sec:reg}):
\begin{equation}\label{f:regulatordescent}\small
 \xymatrix{
   H^1_{\Iw}(L_\infty{/L},\TLT^{\otimes -j}(1)) \ar[dd]_{\mathrm{pr}_m} \ar[r]^-{\mathrm{Exp}^*} &  \mathbf{A}_L^{\psi=1}(\chi_{\rm LT}^{1-j})  \ar[d]^{ \pi_L^{-m}\varphi_L^{-m}\otimes \mathbf{d}_1} \ar[r]^-{\Xi_{L(\chi_{\rm LT}^{1-j})}} & \frac{1}{\frak{l}'_{L(\chi_\LT^{j+1})}}D(\Gamma_L,K)\otimes_L D_{\mathrm{cris}}(L(\chi_{\rm LT}^{-j})(1))\ar[dd]^-{\mathrm{pr}_{\Gamma_m}\otimes\id}\\
   &   L_m((t_{\rm LT})) \otimes D_{\mathrm{cris}}(L(\chi_{\rm LT}^{-j})(1))  \ar[d]^{c_{t_{\rm LT}^{0}}\otimes \id} & \\
   H^1(L_m,\TLT^{\otimes -j}(1))\ar[r]^{  \exp_j^*}&   L_m \otimes D_{\mathrm{cris}}(L(\chi_{\rm LT}^{-j})(1)) \ar[r]^-{\Theta_{L(\chi_\mathrm{\rm LT}^{1-j}),D,m}} & K[\Gamma_L/\Gamma_m]\otimes_L D_{\mathrm{cris}}(L(\chi_{\rm LT}^{-j})(1)) .  }
\end{equation}
%Here $\sigma_{-1}\in\Gamma_L$ acts on $L_m$ via its image in $G(L_m/L).$
Hence the commutativity of the outer diagram   follows from Corollary \ref{f:intchi-j} with $j'=0$ together with the commutativity of the following diagram for $j\geq 1:$ \small
\begin{equation*}
  \xymatrix{
    H^1(L_m,\TLT^{\otimes -j}(1))\ar[rr]^-{  \exp_j^*}& &  L_m \otimes D_{\mathrm{cris}}(L(\chi_{\rm LT}^{-j})(1)) \ar[rrr]^-{\Theta_{L(\chi_\mathrm{\rm LT}^{1-j}),D,m}} & && K[\Gamma_L/\Gamma_m]\otimes_L D_{\mathrm{cris}}(L(\chi_{\rm LT}^{-j})(1))  \\
   H^1(L_m,\TLT^{\otimes -j}(1))\ar[rr]^-{ \widetilde{\exp}^*_{L_m,L(\chi_{\LT}^{-j})(1),\id}}\ar@{=}[u]& &  L_m \otimes D_{\mathrm{cris}}(L(\chi_{\rm LT}^{1-j})) \ar[rrr]^-{ \Theta^\ast_{L(\chi_{\rm LT}^{1-j}),m}
   %(\frac{\pi_L}{q})^{a(\cdot)}(\Upsilon^m)^{-1}\circ(\mathfrak{C}_{m,0}^*)^{-1}
    } \ar[u]^{\cdot\otimes\mathbf{d}_1}&&& K[\Gamma_L/\Gamma_m]\otimes_L D_{\mathrm{cris}}(L(\chi_{\rm LT}^{1-j})) . \ar[u]^{\cdot\otimes\mathbf{d}_1}  }
\end{equation*}\normalsize
Indeed,  the first square is commutative by the definition of $\widetilde{\exp}^*_{L_m,L(\chi_{\LT}^{-j})(1),\id}$ while the second one due to (\ref{Thetad1}) in Remark \ref{rem:EvDn} (with $W=L(\chi_\LT^{1-j}),$  $D:=D_\cris(L(\tau))$).

The commutativity of the right inner diagram of \eqref{f:regulatordescent} is the content of   Remark \ref{rem:EvDn} providing the descent of  \eqref{f:Xiexplicit}  combined with the following commutative diagram:\footnotesize
\begin{equation}\label{f:Dtwist}
\xymatrix{
  \mathbf{A}_L^{\psi=1}(\chi_{\rm LT}^{1-j}) \ar[d]_{\pi_L^{-m}\varphi_L^{-m}} \ar[r]^-{\id \otimes \mathbf{d}_1} &  \mathbf{A}_L^{\psi=1}(\chi_{\rm LT}^{1-j})\otimes D \ar[d]^{q^{-m}\varphi_L^{-m} \otimes \varphi_L^{-m}}\ar[r]^-\subseteq & \bigg(\frac{1}{t^{j+1}}\mathcal{O}_K\otimes D_{\cris}(L(\chi_\LT^{-j})(1))\bigg)^{\Psi_L=1}\ar[dl]^{q^{-m}\varphi_L^{-m} \otimes \varphi_L^{-m}}\\
  K_m((t_{\rm LT}))\otimes_L D_{\mathrm{cris}}(L(\chi_{\rm LT}^{1-j})) \ar[r]^-{\id \otimes \mathbf{d}_1} & K_m((t_{\rm LT}))\otimes_L D_{\mathrm{cris}}(L(\chi_{\rm LT}^{-j})(1)),   }
\end{equation}\normalsize
which stems immediately from the definition of the map $\varphi_L^{-m}.$ Note that the descent of \eqref{f:Gammainverse} and \eqref{f:XiexplicitGamma} both give rise to the factor $(j+1)!$ (whence cancel each other) by \eqref{f:evlfactor} and Lemma \ref{lem:descent}, respectively.

Note in general that $\Theta_{W,D,m}$ is an isomorphism if $D_{\cris,L}(W)^{\varphi_L=\pi_L^{-1}}= D_{\cris,L}(W)^{\varphi_L=\frac{q}{\pi_L}}=0$.
Since $\Theta_{L(\chi_\mathrm{\rm LT}^{1-j}),D,m}$ is an isomorphism (by our assumption $q \neq  \pi_L^j$), the claimed commutativity of \eqref{f:interpolation} follows.
%
%
%
%
%The following diagram for $k\geq 1$ \small
%\begin{equation*}
%  \xymatrix{
%    H^1(L_m,\TLT^{\otimes -k}(1))\ar[rr]^-{\sigma_{-1} \exp_k^*}& &  L_m \otimes D_{\mathrm{cris}}(L(\chi_{\rm LT}^{-k})(1)) \ar[rrr]^-{\Theta_{L(\chi_\mathrm{\rm LT}^{1-k}),D,m}} & && K[\Gamma_L/\Gamma_m]\otimes_L D_{\mathrm{cris}}(L(\chi_{\rm LT}^{-k})(1))  \\
%   H^1(L_m,\TLT^{\otimes -k}(1))\ar[rr]^-{\sigma_{-1}\widetilde{\exp}^*_{L_m,V,\id}}\ar@{=}[u]& &  L_m \otimes D_{\mathrm{cris}}(L(\chi_{\rm LT}^{1-k})) \ar[rrr]^-{(\frac{\pi_L}{q})^{a(\cdot)}(\Upsilon^m)^{-1}\circ(\mathfrak{C}_{m,0}^*)^{-1} } \ar[u]^{\cdot\otimes\mathbf{d}_1}&&& K[\Gamma_L/\Gamma_m]\otimes_L D_{\mathrm{cris}}(L(\chi_{\rm LT}^{1-k})) . \ar[u]^{\cdot\otimes\mathbf{d}_1}  }
%\end{equation*}\normalsize
%is commutative, because the first square is so by definition of $\widetilde{\exp}^*_{L_m,V,\id}$ while the second one due to Remark %\ref{rem:EvDn}.  {\color{red} but with ceratin $\mathrm{pr}_{m,-j}$ instead of $\mathrm{pr}$????which we hopefully can untwist again!}
\end{proof}
\begin{lemma}\label{lem:descent}
The following diagram is commutative:
\[\xymatrix{
  \bigg(\frac{1}{t^{j+1}}\mathcal{O}_K\otimes D_{\cris,L}(L(\chi_\LT^{-j})(1))\bigg)^{\Psi_L=1} \ar[d]_{q^{-m}\mathrm{Ev}_{L(\chi_{\LT}^{1-j}),D,m}} \ar[rrr]^{\partial_\mathrm{inv}^{j+1}t_\LT^{j+1}\otimes\id} &&& \bigg(\mathcal{O}_K\otimes D_{\cris,L}(L(\chi_\LT^{-j})(1))\bigg)^{\Psi_L=1} \ar[d]^{q^{-m}\mathrm{Ev}_{L(\chi_{\LT}^{1-j}),D,m}} \\
 K_m\otimes_L  D_{\cris,L}(L(\chi_\LT^{-j})(1)) \ar[rrr]^{(j+1)!} &&& K_m\otimes_L  D_{\cris,L}(L(\chi_\LT^{-j})(1)) . }\]
\end{lemma}
\begin{proof}
It is sufficient to check that
$$\text{the constant term of }\varphi_L^{-m}\circ \partial_{\rm inv}^{j+1}(t_{\rm LT}^{j+1} f) = \text{the constant term of }(j+1)!\cdot \varphi_L^{-m}(f)$$
for any $f \in o_{\CC_p}((t_{\rm LT}))$.  We write $f=\sum_{n\in \ZZ}a_nt_{\rm LT}^n$. Since $\partial_{\rm inv}=\frac{d}{dt_{\rm LT}}$ by Lemma \ref{lemder}, we have
$$ \partial_{\rm inv}^{j+1}(t_{\rm LT}^{j+1} f)  =\sum_{n\in \ZZ}(n+j+1)(n+j)\cdots(n+1)a_nt_{\rm LT}^n.$$
Also, since $\varphi_L^{-m}(t_{\rm LT})=\pi_L^{-m} t_{\rm LT}$ by Lemma \ref{lemtlt}, we have
$$\varphi_L^{-m}\circ \partial_{\rm inv}^{j+1}(t_{\rm LT}^{j+1} f) = \sum_{n\in \ZZ}(n+j+1)(n+j)\cdots(n+1)a_n \pi_L^{-mn}t_{\rm LT}^n.$$
Hence we have
$$\text{the constant term of }\varphi_L^{-m}\circ \partial_{\rm inv}^{j+1}(t_{\rm LT}^{j+1} f) = (j+1)!a_0.$$
On the other hand, we have $\varphi_L^{-m}(f)=\sum_{n\in \ZZ}a_n \pi_L^{-mn}t_{\rm LT}^n$ and so
$$\text{the constant term of }(j+1)!\cdot \varphi_L^{-m}(f)= (j+1)! a_0.$$
This completes the proof.
%By Lemma \ref{lem:invphi}, we have
%$$\varphi_L^{-m}\circ \partial_{\rm inv}^{j+1}(t_{\rm LT}^{j+1}f)= \pi_L^{m(j+1)}\partial_{\rm inv}^{j+1}\circ \varphi_L^{-m}(t_{\rm LT}^{j+1}f).$$
%Also, by Lemma \ref{lemtlt}, we have
%$$\varphi_L^{-m}(t_{\rm LT}^{j+1}f)=\pi_L^{-m(j+1)}t_{\rm LT}^{j+1}\varphi_L^{-m}(f).$$
%Hence we have
%$$\varphi_L^{-m}\circ \partial_{\rm inv}^{j+1}(t_{\rm LT}^{j+1}f)=\partial_{\rm inv}^{j+1}(t_{\rm LT}^{j+1} \varphi_L^{-m}(f)).$$
%Note that this is an equality in $\CC_p[[t_{\rm LT}]]$ {\color{purple}since $\varphi_L^{-m}$ is a map from $o_{\CC_p}((t_{\rm LT}))=o_{\CC_p}((Z))$ to $\CC_p[[t_{\rm LT}]]$ (see the proof of Corollary \ref{corpsi}).} Since $\partial_{\rm inv}=\frac{d}{dt_{\rm LT}}$ by Lemma \ref{lemder}, the claim follows.
%  Idea: check compatibility of  $\partial_\mathrm{inv}=\frac{d}{dt_\LT}$ with $\varphi_L^{-m}$ and use for $f\in t^{-k}_\LT L_m[[t_\LT]]$ that $c_{t^0_\LT}(f)= k!c_{t^0_\LT}((\frac{d}{dt_\LT})^k(t_\LT^k f)).$
\end{proof}

\subsection{Colmez' reciprocity law}\label{sec colmez}

In this subsection, we give a variant of Theorem \ref{key claim} for a more general $p$-adic representation.

Let   $T$ be in $\Rep_{o_L,f}^{\cris}(G_L)$    such that $T(\tau^{-1})$ belongs to  $\Rep_{o_L,f}^{\cris,\mathrm{an}}(G_L)$ with all   Hodge--Tate weights   in $[1,r],$ and such that $V:=L\otimes_{o_L} T$ does not have any quotient isomorphic to $L(\tau).$ Then the assumptions of \cite[Lem.\ 3.3.6]{SV24}  are satisfied and we may define the  maps
$  \varphi_L^{-m}:D_{\LT}(T(\tau^{-1}))^{\psi=1}\to L_m[[t_{\rm LT}]]\otimes_L D_{\mathrm{cris}}(V(\tau^{-1})) $   via the  inclusions
%\[(\mathbf{A}_L(\chi_{\rm LT}^{1-j}))^{\psi=1}=(\omega_{\rm LT}^{j-2}\mathbf{A}_L^+(\chi_{\rm LT}^{1-j}))^{\psi=1}\]
 - by \cite[(3.16), Cor.\ 3.1.14]{SV24} involving  Kisin's and Ren's
comparison isomorphism -,
\[D_{\LT}(T(\tau^{-1}))^{\psi=1}=(N(T(\tau^{-1})))^{\psi=1} \subseteq \cR_L^+\otimes_L D_\mathrm{cris}(V(\tau^{-1}))\]
and  the map $\varphi_q^{-m}$ defined before \cite[Lem.\ 2.4.3]{BF}.

\begin{theorem}[Colmez' reciprocity law in the Lubin-Tate setting]\label{thm:Colmez}
  Assume  that $V$ satisfies $D_{\cris,L}(V(\tau^{-1}))^{\varphi_L=1}=D_{\cris,L}(V(\tau^{-1}))^{\varphi_L=\pi_L^{-1}}=0$. Then, for all $j\geq 1$   such that the operators  $1-\pi_L^{-1-j}\varphi_L^{-1}, 1-\frac{\pi_L^{j+1}}{q}\varphi_L $ are invertible on $D_{\cris,L}(V(\tau^{-1}))$, the following diagram commutes:
\begin{equation}\label{f:interpolationgeneral}
 \xymatrix{
   H^1_{\Iw}(L_\infty{/L},T) \ar[dd]_{\mathrm{pr}_{m,-j}} \ar[rr]^-{\mathrm{Exp}^*} &&  D_{\LT}(T(\tau^{-1}))^{\psi_L=1}  \ar[d]^{\pi_L^{-m}\varphi_L^{-m}\otimes \mathbf{d}_1}   \\
   &&   L_m((t_{\rm LT}))\otimes_L D_{\mathrm{cris}}(V) \ar[d]^{c_{t_{\rm LT}^{j}}\otimes (\id\otimes e_{-j})} \\
   H^1(L_m,V(\chi_\mathrm{\rm LT}^{-j}))\ar[rr]_-{\exp_{L_m,V(\chi_\LT^{-j})}^*}&&   L_m \otimes_L D_{\mathrm{cris}}(V(\chi_\LT^{-j})).  }
\end{equation}
In other words, in $L_m((t_\LT))\otimes_L D_{\mathrm{dR}}(V) \subseteq B_\dR \otimes_L V$ we have
\[\pi_L^{-m}\varphi_L^{-m}(\mathrm{Exp}^*)\otimes \mathbf{d}_1=\sum_{j\geq 1 } \exp^*_{L_m,V(\chi_\LT^{-j})}\circ \mathrm{pr}_{m,-j}=\sum_{j\in\mathbb{Z} }\exp^*_{L_m,V(\chi_\LT^{-j})}\circ \mathrm{pr}_{m,-j},\]
if the above conditions hold for all $j\geq 1,$ i.e., if $\varphi_L$ acting on $D_{\cris,L}(V(\tau^{-1}))$ does not have any eigenvalue in $\pi_L^{\mathbb{N}_0}\cup q\pi_L^{\mathbb{N}_0};$ to this end we identify the target $L_m \otimes_L D_{\mathrm{cris}}(V(\chi_\LT^{-j}))$ with $L_mt_\LT^{j}\otimes_L D_{\cris,L}(V)\subseteq L_m((t_\LT))\otimes_L D_{\mathrm{dR}}(V).$ Note that $c_{t^i}(\pi_L^{-m}\varphi_L^{-m}(\mathrm{Exp}^*)\otimes \mathbf{d}_1)=0$ for $i\leq 0$, because ${\rm Fil}^0 D_\dR(V)=0$ as the Hodge-Tate weights of $V$ are at least $1.$
\end{theorem}

\begin{proof}
  The reasoning is the same as  for the proof of the commutativity of the outer diagram \eqref{f:regulatordescent}: Apply Theorem \ref{thm:adjointformulan} for the given $V$ and $j \geq 1$ together with Remark \ref{rem:EvDn} in combination with \eqref{f:Dtwist}. The assumptions guarantee that $\Theta_{V(\tau^{-1}\chi_{\LT}^{-j}),D,m}$ is defined and injective.
\end{proof}

While our method of proof is not able to show it, we are very optimistic that there should be a positive answer to the following\\

\noindent
{\bf Question:} Is the statement of the theorem also valid without the assumption that $\varphi_L$ acting on $D_{\cris,L}(V(\tau^{-1}))$ does not have any eigenvalue in $\pi_L^{\mathbb{N}_0}\cup q\pi_L^{\mathbb{N}_0}?$

\subsection{The second proof}\label{sec second}

In this subsection, we give a proof of Kato's explicit reciprocity law by a different argument, without assuming $\pi_L^j\neq q$. The idea is to prove each character component of the equality (see (\ref{katorho}) below). The key ingredients of the proof are Theorem \ref{thm sv} and Corollary \ref{f:intchi-j}.

Let $\lambda_{m,j}, \psi_{{\rm CW},m}^j : \varprojlim_n L_n^\times \to L_m$ be the maps in Introduction.

\begin{lemma}\label{lem:trace}
For any $ j,m \geq 1$ we have
\begin{itemize}
\item[(i)] ${\rm Tr}_{L_{m+1}/L_{m}}\circ \lambda_{m+1,j} = \lambda_{m,j}$,
\item[(ii)] ${\rm Tr}_{L_{m+1}/L_{m}}\circ \psi_{{\rm CW},m+1}^j = \psi_{{\rm CW},m}^j.$
\end{itemize}
\end{lemma}

\begin{proof}
(i) is straightforward using the compatibility of the dual exponential map under corestriction. To prove (ii), we first note that $\psi_L \partial_{\rm inv}= \pi_L \partial_{\rm inv} \psi_L$ and $\partial_{\rm inv}\log (g_{u,\eta})\in \bA_L^{\psi_L=1}$ for any $u\in \varprojlim_n L_n^\times$. These facts imply
$$\psi_L \partial_{\rm inv}^j \log (g_{u,\eta})=\pi_L^{j-1}\partial_{\rm inv}^j \log (g_{u,\eta}).$$
Applying $\varphi_L$ to both sides, we obtain
$$\pi_L^{-1}\sum_{a\in {\rm LT}_1}\partial_{\rm inv}^j \log (g_{u,\eta} )(a+_{\rm LT} Z) = \pi_L^{j-1}\varphi_L\partial_{\rm inv}^j \log (g_{u,\eta}).$$
We then take evaluation at $Z=\eta_{m+1}$:
$$\pi_L^{-1} {\rm Tr}_{L_{m+1}/L_m}(\partial_{\rm inv}^j \log (g_{u,\eta}) |_{Z=\eta_{m+1}})=\pi_L^{j-1} \partial_{\rm inv}^j \log (g_{u,\eta})|_{Z=\eta_m}.$$
(Note that $\sum_{a \in {\rm LT}_1}f(a+_{\rm LT}\eta_{m+1}) = {\rm Tr}_{L_{m+1}/L_m}(f(\eta_{m+1}))$ since $m\geq 1$.)
Hence we have
$$\pi_L^{-(m+1)j} {\rm Tr}_{L_{m+1}/L_m}(\partial_{\rm inv}^j \log (g_{u,\eta} )|_{Z=\eta_{m+1}})=\pi_L^{-mj} \partial_{\rm inv}^j \log (g_{u,\eta})|_{Z=\eta_m}.$$
Since $\psi_{{\rm CW},m}^j(u) = j!^{-1}\pi_L^{-mj} \partial_{\rm inv}^j \log (g_{u,\eta})|_{Z=\eta_m}$ by definition, this proves (ii).
\end{proof}

\begin{remark}\label{rem:m0}
For $u \in \varprojlim_n o_{L_n}^\times$ one can check that
$${\rm Tr}_{L_1/L}(\psi_{{\rm CW},1}^j(u)) = j!^{-1} (1-\pi_L^{-j})\partial_{\rm inv}^j \log (g_{u,\eta})|_{Z=0}.$$
(Note that we can take evaluation at $Z=0$ since $g_{u,\eta} \in o_L[[Z]]^\times$ in this case.)
\end{remark}

\begin{lemma}\label{lem:evaluate}
Let $\lambda \in D(\Gamma_L, \CC_p)$ and set $f:=\mathfrak{M}(\lambda)=\lambda\cdot \eta(1,Z)$. Let $\rho$ be the character of $G_m:=\Gal(L_m/L)$ of conductor $m >0$. Then we have
$$\lambda(\rho) =[L_m:L]\tau(\rho)^{-1}\mathfrak{e}_\rho \cdot f(\eta_m),$$
where $\mathfrak{e}_\rho:=[L_m:L]^{-1}\sum_{\sigma \in G_m}\rho^{-1}(\sigma)\sigma$ is the usual idempotent.
\end{lemma}

\begin{proof}
The proof is similar to that of Lemma \ref{lem:Evn}. We compute
\begin{align*}
\mathfrak{e}_\rho \cdot  f(\eta_m)& = \mathfrak{e}_\rho\cdot  ( \lambda \cdot \eta(1,Z))|_{Z=\eta_m}\\
&= \mathfrak{e}_\rho \cdot \lambda\cdot \eta(1,\eta_m) \\
&= \lambda(\rho)\mathfrak{e}_\rho \cdot \eta(1,\eta_m) \\
&= [L_m: L]^{-1}\tau(\rho)\lambda(\rho),
\end{align*}
where the second equality follows from the fact that evaluating $Z=\eta_m$ is a $D(\Gamma_L,\CC_p)$-homomorphism, and the last from the definition of the Gauss sum $\tau(\rho)$. Upon multiplying with $[L_m:L]\tau(\rho)^{-1}$ on both sides, we obtain the desired equality.
\end{proof}

Recall that we set $\tau:=\chi_{\rm cyc}\chi_{\rm LT}^{-1}$. For $u\in \varprojlim_n L_n^\times$, let
$$\mathbf{u} \in H^1_{\rm Iw}(L_\infty/L, L(\tau))$$
be the image of $u$ under the composite map
$$\varprojlim_n L_n^\times \xrightarrow{{\rm Kum}} H^1_{\rm Iw}(L_\infty/L, \ZZ_p(1)) \xrightarrow{{\rm tw}_1} H^1_{\rm Iw}(L_\infty/L, L(\tau)).$$
Note that one can define the regulator map $\mathbf{L}_{L(\tau)}$ as a twist of $\mathbf{L}_{L(\tau \chi_{\rm LT})}$ (see \cite[\S 5.1.1]{SV24}). According to Figure 5.2 in (loc.\ cit.) we have the following commutative diagram:
$$
\xymatrix{
H^1_{\rm Iw}(L_\infty/L, L(\tau)) \ar[rr]^-{\mathbf{L}_{L(\tau)}} \ar[d]_-{\simeq}&& D(\Gamma_L,\CC_p)\otimes_L D_{{\rm cris},L}(L) \ar[d]^-{\frac{\nabla {\rm Tw}_{\chi_{\rm LT}^{-1}}}{\Omega} \otimes t_{\rm LT}^{-1}} \\
H^1_{\rm Iw}(L_\infty/L, L(\tau \chi_{\rm LT})) \otimes_L L(\chi_{\rm LT}^{-1})  \ar[rr]_-{\mathbf{L}_{L(\tau\chi_{\rm LT})} \otimes \id }&& D(\Gamma_L,\CC_p)\otimes_L D_{{\rm cris},L}(L(\chi_{\rm LT})) \otimes_L L(\chi_{\rm LT}^{-1}).
}
$$
It is shown in \cite[\S 5.1]{SV24} that the map
$$H^1_{\rm Iw}(L_\infty/L,L(\tau))\simeq H^1_{\rm Iw}(L_\infty/L, L(\tau \chi_{\rm LT})) \otimes_L L(\chi_{\rm LT}^{-1}) \xrightarrow{\mathbf{L}_{L(\tau \chi_{\rm LT})} \otimes \id \otimes t_{\rm LT}} D(\Gamma_L,\CC_p)\otimes_L D_{{\rm cris},L}(L)$$
sends $\mathbf{u}$ to
$$\mathfrak{M}^{-1}\left(t_{\rm LT} \left(1-\frac{\pi_L}{q}\varphi_L\right) \partial_{\rm inv} \log (g_{u,\eta})\right).$$
By \cite[Lem. 5.1.2]{SV24}, we have
$$\frac{\nabla {\rm Tw}_{\chi_{\rm LT}^{-1}}}{\Omega} \circ \mathfrak{M}^{-1} = \mathfrak{M}^{-1}\circ t_{\rm LT},$$
so the commutative diagram above implies that $\mathbf{L}_{L(\tau)}$ satisfies
\begin{equation}\label{explicit regulator}
\mathbf{L}_{L(\tau)}(\mathbf{u})=\mathfrak{M}^{-1}\left( \left(1-\frac{\pi_L}{q}\varphi_L \right)\partial_{\rm inv}\log(g_{u,\eta}) \right).
\end{equation}

\begin{lemma}\label{lem:regtwist}
Let $\rho$ be a character of $\Gamma_L$ of conductor $m$. If $m>0$ (i.e., $\rho$ is a non-trivial character), then for any $u\in \varprojlim_n L_n^\times$ and $j\geq 1$ we have
$$\Omega^{j-1}\mathbf{L}_{L(\tau)}(\mathbf{u}) (\rho \chi_{\rm LT}^{j-1}) = [L_m:L]
\tau(\rho)^{-1}\mathfrak{e}_\rho\cdot \left. \left(\partial_{\mathrm{inv}}^{j}\log(g_{u,\eta})\right)\right|_{Z=\eta_m} .
$$
If $m=0$ (i.e., $\rho$ is the trivial character), then for any $u\in \varprojlim_n o_{L_n}^\times$ and $j\geq 1$ we have
$$\Omega^{j-1} \mathbf{L}_{L(\tau)}(\mathbf{u})(\rho\chi_{\rm LT}^{j-1}) = \left(1-\frac{\pi_L^j}{q}\right)\left.\left(\partial_{\rm inv}^j \log (g_{u,\eta})\right)\right|_{Z=0}.$$
\end{lemma}

\begin{proof}
By (\ref{explicit regulator}) and Lemma \ref{lem:Tw}, we have
\begin{align}\label{omegamellin}
\Omega^{j-1}\mathfrak{M}\left({\rm Tw}_{\chi_{\rm LT}^{j-1}}\circ \mathbf{L}_{L(\tau)}(\mathbf{u})\right)&=\partial_{\rm inv}^{j-1} \left(1-\frac{\pi_L}{q}\varphi_L\right)\partial_{\rm inv}\log(g_{u,\eta})\\
&= \partial_{\rm inv}^j \log(g_{u,\eta})- \frac{\pi_L^j}{q}\varphi_L \partial_{\rm inv}^j \log(g_{u,\eta}),\nonumber
\end{align}
where the second equality follows by noting $\partial_{\rm inv}\varphi_L=\pi_L\varphi_L \partial_{\rm inv}$. If $m>0$, note that
$$\left.\left(\varphi_L \partial_{\rm inv}^j \log(g_{u,\eta}) \right)\right|_{Z=\eta_m} = \left.\left( \partial_{\rm inv}^j \log(g_{u,\eta}) \right)\right|_{Z=\eta_{m-1}} $$
and this is annihilated by $\mathfrak{e}_\rho$ (since $\rho$ does not factor through $G_{m-1}$). Hence, by Lemma \ref{lem:evaluate}, we have
$$\Omega^{j-1}\mathbf{L}_{L(\tau)}(\mathbf{u}) (\rho \chi_{\rm LT}^{j-1}) =[L_m:L]\tau(\rho)^{-1}\mathfrak{e}_\rho\cdot \left. \left(\partial_{\mathrm{inv}}^{j}\log(g_{u,\eta})\right)\right|_{Z=\eta_m}. $$
This proves the first claim. To prove the second claim, we let $m=0$ and $u \in \varprojlim_n o_{L_n}^\times$. In this case, we can evaluate (\ref{omegamellin}) at $Z=0$ and we obtain the desired formula
$$\Omega^{j-1} \mathbf{L}_{L(\tau)}(\mathbf{u})(\chi_{\rm LT}^{j-1}) = \left(1-\frac{\pi_L^j}{q}\right)\left.\left(\partial_{\rm inv}^j \log (g_{u,\eta})\right)\right|_{Z=0}.$$

\end{proof}

We now give a proof of Kato's explicit reciprocity law.

\begin{proof}[Proof of Theorem \ref{kato erl}]
Let $j,m\geq 1$ and $u \in \varprojlim_n L_n^\times$. It is sufficient to show that for any character $\rho$ of $G_m$ we have
\begin{equation}\label{katorho}
\mathfrak{e}_\rho \cdot \lambda_{m,j}(u)= \mathfrak{e}_\rho\cdot    \frac{1}{(j-1)!\pi_L^{mj}} \left. \left(\partial_{\mathrm{inv}}^{j}\log(g_{u,\eta})\right)\right|_{Z=\eta_m}.
\end{equation}

If $\rho$ is non-trivial, then by Lemma \ref{lem:trace} we may assume that $m$ is the conductor of $\rho$. In this case, Corollary \ref{f:intchi-j} with $k=1$ and $j'=j-1$ gives the formula
$$\Omega^{j-1}\mathbf{L}_{L(\tau)}(\mathbf{u})(\rho \chi_{\rm LT}^{j-1}) = (j-1)! [L_m:L]\pi_L^{mj} \tau(\rho)^{-1}\mathfrak{e}_\rho\cdot \lambda_{m,j}(u)$$
 (note that $\varphi_L$ acts trivially on $D_{{\rm cris},L}(L)$). On the other hand, Lemma \ref{lem:regtwist} gives
$$\Omega^{j-1}\mathbf{L}_{L(\tau)}(\mathbf{u}) (\rho \chi_{\rm LT}^{j-1}) = [L_m:L]
\tau(\rho)^{-1}\mathfrak{e}_\rho\cdot\left. \left(\partial_{\mathrm{inv}}^{j}\log(g_{u,\eta})\right)\right|_{Z=\eta_m}.
$$
Comparing these formulas, we obtain the desired equality (\ref{katorho}) for non-trivial characters.

It remains to prove (\ref{katorho}) when $\rho$ is the trivial character.  Replacing $u$ by $u/\sigma(u)$ with $\sigma\in G_L$ such that $\chi_{\rm LT}^j(\sigma) \neq 1$, we may assume $u\in \varprojlim_n o_{L_n}^\times$. (Note that $\lambda_{m,j}$ satisfies $\lambda_{m,j}(\sigma u)= \chi_{\rm LT}^j(\sigma)\lambda_{m,j}(u)$ and similarly for $\psi_{{\rm CW},m}^j$.) By Remark \ref{rem:m0}, we are reduced to proving the equality
\begin{equation*}\label{specialkato}
\lambda_{0,j}(u) = \frac{1-\pi_L^{-j}}{(j-1)!} \left. \left(\partial_{\rm inv}^j \log (g_{u,\eta})\right)\right|_{Z=0}.
\end{equation*}
This is exactly the ``special case of Kato's explicit reciprocity law" proved in \cite[Cor. 8.7]{SV15}. (Note that the dual exponential map in loc. cit. is defined by using (a) in \S \ref{rem sign} (see the diagram on \cite[p. 462]{SV15}), so it is different from ours by sign.)  %We omit the proof.
%Corollary \ref{cor:adjointformula} gives the formula
%$$\Omega^{j-1} \mathbf{L}_{L(\tau)}(\mathbf{u})(\chi_{\rm LT}^{j-1}) = (j-1)! (1-\pi_L^{-j})^{-1}\left(1-\frac{\pi_L^j}{q}\right)\lambda_{0,j}(u).$$
%On the other hand, {\color{purple}by Lemma \ref{lem:regtwist} we have}
%$$\Omega^{j-1} \mathbf{L}_{L(\tau)}(\mathbf{u})(\chi_{\rm LT}^{j-1}) = \left(1-\frac{\pi_L^j}{q}\right)\left.\left(\partial_{\rm inv}^j \log (g_{u,\eta})\right)\right|_{Z=0}.$$
%These formulas immediately imply (\ref{specialkato}). Thus we have completed the proof. {\color{red}This argument does not work if $\pi_L^j = q$...}
\end{proof}

\appendix

\numberwithin{equation}{section} %   remove section zeros in numbering of (equations)

\section{The action of the Lie algebra and interpolation of $\Gamma$-factors}\label{sec:Lie}

The group isomorphism $\Gamma_L\xrightarrow{\chi_\LT} o_L^\times$ of locally $L$-analytic Lie groups induces an isomorphism of Lie algebras
\[\mathrm{Lie}(\Gamma_L)\xrightarrow{{\rm Lie}(\chi_\LT)}\mathrm{Lie}(o_L^\times)=L\] and the exponential map for $o_L^\times$ as $L$-analytic group is just given by the usual $p$-adic exponential map $\exp: L --> o_L^\times$ (The broken arrow indicates that this map is only defined locally around zero). We write $\nabla\in  \mathrm{Lie}(\Gamma_L)$ for the element corresponding to $1\in L=\mathrm{Lie}(o_L^\times).$ If $\log_{\Gamma_L}:\Gamma_L \to \mathrm{Lie}(\Gamma_L)$ denotes the logarithm map of $\Gamma_L$, then, for tautological reasons, we have the identity
\begin{equation*}
  \nabla= \frac{\log_{\Gamma_L} \gamma}{\log \chi_\LT(\gamma)} \in \mathrm{Lie}(\Gamma_L)
\end{equation*}
for any element $\gamma\in\Gamma_L$ close to $1.$
Following  \cite[\S 2]{KR} or \cite[\S 2.3, especially Cor.\ 2.15]{BSX} the derived action of $\nabla$ on a $K$-Banach spaces $B$ is similarly given as
\begin{equation*}
   \frac{\log \gamma}{\log \chi_\LT(\gamma)}\in A
\end{equation*}
for any element $\gamma\in\Gamma_L$ close to $1,$ where now $\log\gamma$ denotes the series $\log(1+(\gamma-1))=-\sum_{n\geq 1}\frac{1}{n}(1-\gamma)^n$ within the $K$-Banach space $A$ of continuous linear endomorphisms of $B$ endowed with the operator norm.
%This  point of view (which is often given as a (heuristic) definition) is in particular true, when we are dealing with the derived action of a $K$-linear action on a finite dimensional $K$-vector space.
By \cite[\S 2.3]{ST1} there is a natural embedding of $\mathrm{Lie}(\Gamma_L)$ into $D(\Gamma_L,K)$ which is  compatible with the derived action on it in the following sense: the action of $\nabla$  on $D(\Gamma,K)$ coincides with the multiplication by $\nabla$ considered as an element of $D(\Gamma_L,K)$ in the mentioned way. Since one can write the distribution algebra as inverse limit of Banach spaces we deduce that the image of $\nabla$ in $D(\Gamma_L,K)$ coincides with
\begin{equation}\label{f:loggamma}
   \frac{\log \gamma}{\log \chi_\LT(\gamma)}\in D(\Gamma_L,K)
\end{equation}
for any element $\gamma\in\Gamma_L$ close to $1$ (in particular, such that the above $\log$-series converges with respect to the Fr\'{e}chet-Topology of the distribution algebra).

\begin{lemma} Let $\eta=\rho\chi_\LT^j$ be a character of $\Gamma_L$ with $\rho$ a character of finite order and $j\in \mathbb{Z}.$
\begin{enumerate}
\item $\mathrm{Tw}_{\eta}(\nabla)=\nabla +j$
\item $\nabla(\eta)=j$
\item $\nabla(\eta,n)=j$
\end{enumerate}
\end{lemma}

\begin{proof} (i) Using \eqref{f:loggamma} and the fact that $\log$ is a group homomorphism we calculate %(as identity of operators)
  \begin{align*}
    \mathrm{Tw}_{\eta}(\nabla) & = \frac{\log \eta(\gamma)\gamma}{\log \chi_\LT(\gamma)} \\
    & =  \frac{\log \chi_\LT^j(\gamma) }{\log \chi_\LT(\gamma)}+ \frac{\log \rho(\gamma)}{\log \chi_\LT(\gamma)}+ \frac{\log \gamma}{\log \chi_\LT(\gamma)}\\
      & =j+\nabla,
  \end{align*}
since $\rho(\gamma)$ is a torsion element. For (ii) we have
similarly
\begin{align*}
     {\eta}(\nabla) & = \frac{\log \eta(\gamma) }{\log \chi_\LT(\gamma)} \\
    & =  \frac{\log \chi_\LT^j(\gamma) }{\log \chi_\LT(\gamma)}+ \frac{\log \rho(\gamma)}{\log \chi_\LT(\gamma)} \\
      & =j .
  \end{align*}
(iii) follows from (ii) by considering $\rho'$-components of $K[G_n]$ or from (i) by noting that $\nabla$ belongs  to the kernel of  $\mathrm{pr}_{G_n}$ (represent $\nabla$ as \eqref{f:loggamma} with $\gamma\in\Gamma_n$).% These lines turn into a rigourous proof regarding at least (ii) and (iii) by the following reasoning: the evaluation of $\nabla$ at $\eta$ coincides with the element in $\mathrm{End}_K(K)$ describing the derived action of $\Gamma$ on $K(\eta)$. But within this endomorphism algebra the description \eqref{f:loggamma} holds true as mentioned above. Regarding (i) one can write the distribution algebra as inverse limit of Banach spaces and deduce the claim from this.
\end{proof}

We set $\frak{l}_i:=\nabla-i$,  $\frak{l}_{L(\chi_\LT^i)}:=\prod_{k=0}^{i-1}\frak{l}_k$ and   $\frak{l}'_{L(\chi_\LT^i)}:=\mathrm{Tw}_{\chi^i}(\frak{l}_{L(\chi_\LT^i)})=\prod_{k=0}^{i-1}\frak{l}_{k-i}$.  It follows immediately from the lemma that for $k\geq 0$ we have for all $n\geq 0$
\begin{align}\label{f:evlfactor}
 \frak{l}_{L(\chi_\LT^k)}(\chi_\LT^j,n) & =\left\{
                                             \begin{array}{ll}
                                              \frac{j!}{(j-k)!}, & \hbox{if $j\geq k$;} \\
                                               0, & \hbox{if $0\leq j<k;$} \\
                                                (-1)^{k}\frac{(k-1-j)!}{(-j-1)!} , & \hbox{otherwise.}
                                             \end{array}
                                           \right.
 %j! =\frak{l}'_{L(\chi_\LT^i)}(\chi^0_\LT,n).
\end{align}
Note that these factors are not zero-divisors because they are non-zero in any component of $D(\Gamma_L,K)=\prod D(\Gamma_1,K)$, where $\Gamma_1\simeq 1+\pi_L o_L$ via the $\chi_{\mathrm{LT}}.$

 \section{Formal Laurent series}\label{sec calcul}

In this appendix, we prove some basic properties of the map $\varphi_L^{-m}$ in Introduction.

\subsection{Formal power series}

Fix a formal power series $h=\sum_{i=0}^\infty b_i t^i \in \CC_p[[t]]$ such that $0< |b_0|_p <1$. For $n\geq 0$, we write
$$h^n=\sum_{i=0}^\infty c_{i,n} t^i.$$

\begin{lemma}\label{cnformula}
We have
$$c_{0,n}=b_0^n\text{ and }c_{i,n}=\frac{1}{ib_0}\sum_{k=1}^i(nk-i+k)b_k c_{i-k,n} \text{ for $i\geq 1$}.$$
\end{lemma}

\begin{proof}
Obviously we have $c_{0,n}=b_0^n$. We shall prove the formula for $i\geq 1$. We fix $n $ and write $c_i:=c_{i,n}$ for simplicity.
From the equality $h\cdot (h^n)'=nh'\cdot h^n$, we have
$$\left(\sum_{i=0}^\infty b_i t^i\right) \left(\sum_{i=0}^\infty (i+1)c_{i+1}t^i\right)=n\left(\sum_{i=0}^\infty (i+1)b_{i+1}t^i\right)\left(\sum_{i=0}^\infty c_i t^i\right).$$
Comparing the coefficients of $t^{i-1}$, we have
$$\sum_{k=0}^i (i-k)b_kc_{i-k} = n\sum_{k=1}^i k b_k c_{i-k}.$$
From this, we obtain
$$ib_0 c_i = \sum_{k=1}^i (nk-i+k)b_kc_{i-k}.$$
Since $b_0\neq 0$, we obtain the formula.
\end{proof}

\begin{lemma}\label{lemcin}
For each $i\geq 0$, there is a constant $C_i>0$ such that $|c_{i,n}|_p \leq C_i |b_0^n|_p$ for all $n\geq 0$.
\end{lemma}

\begin{proof}
We prove the claim by induction on $i$. When $i=0$, the claim is obvious. When $i\geq 1$, by Lemma \ref{cnformula} we have
$$|c_{i,n}|_p \leq |i b_0|_p^{-1}  \max \{|b_k c_{i-k,n}|_p \mid 1\leq k\leq i\}.$$
By induction hypothesis there are $C_{i-k}$ such that $|c_{i-k,n}|_p\leq C_{i-k} |b_0^n|_p$ for any $1\leq k\leq i$. If we set
$$C_i:= |ib_0|_p^{-1} \max \{ |b_k|_p C_{i-k} \mid 1\leq k \leq i\},$$
then we have $|c_{i,n}|_p\leq C_i |b_0^n|_p$. This completes the proof.
\end{proof}

Lemma \ref{lemcin} immediately implies the following.

\begin{proposition}\label{propconvh}
 If $f =\sum_{n=0}^\infty a_nZ^n \in \CC_p[[Z]]$ satisfies $\underset{n\to \infty}{\lim} a_n b_0^n=0$, then
$$f(h)= \sum_{n=0}^\infty a_n h^n := \sum_{i=0}^\infty \left(\sum_{n=0}^\infty a_n c_{i,n}\right)t^i$$
is well-defined in $\CC_p[[t]]$, i.e., $\sum_{n=0}^\infty a_n c_{i,n}$ converges for each $i$.

\end{proposition}

\subsection{The operator $\psi_L$}

Recall that $\bA_L$ denotes the $\pi_L$-adic completion of the ring of formal Laurent series $o_L((Z))$, i.e.,
$$\bA_L:=\left\{ \sum_{n \in \ZZ} a_n Z^n \ \middle| \  a_n\in o_L, \ \underset{n\to -\infty}{\lim} a_n =0\right\}.$$
Let $\psi_L: \bA_L\to \bA_L$ denote the trace operator defined in \cite[\S 3]{SV15} (see also \cite[Rem. 3.2(i)]{SV15}). We set
$$\bA_L^{\psi_L=1}:=\{f\in \bA_L \mid \psi_L(f)=f\}.$$

In this subsection, we prove

\begin{proposition}\label{prop:psisub}
We have
$$\bA_L^{\psi_L=1} \subseteq Z^{-1} o_L[[Z]].$$
\end{proposition}

\begin{remark}
In the cyclotomic case, Proposition \ref{prop:psisub} is proved in \cite[Prop. 7.1.1(i)]{colmez}.
\end{remark}

Proposition \ref{prop:psisub} immediately implies the following.

\begin{corollary}\label{corpsi}
For any $m\geq 1$, the map
$$\varphi_L^{-m}: \bA_L^{\psi_L=1} \to L_m[[t_{\rm LT}]]; \ f\mapsto f\left( \eta_m +_{\rm LT} \exp_{\rm LT}\left(\frac{t_{\rm LT}}{\pi_L^m}\right)\right)$$
is well-defined.
\end{corollary}

\begin{proof}
Recall the following basic facts: we have
$$t_{\rm LT}=\log_{\rm LT}(Z)= \sum_{n=1}^\infty \frac{a_n}{n}Z^n \quad (a_n \in o_L, \ a_1=1),$$
$$\exp_{\rm LT}(Z)=\sum_{n=1}^\infty \frac{b_n}{n!} Z^n \quad (b_n \in o_L, \ b_1=1),$$
$$X+_{\rm LT} Y = X+Y + \sum_{i,j\geq 1} c_{ij} X^iY^j \quad (c_{ij} \in o_L).$$
The first fact implies that $L_m[[Z]]=L_m[[t_{\rm LT}]]$. The second and the third facts imply that the constant term of $h:=\eta_m +_{\rm LT} \exp_{\rm LT}\left(\frac{t_{\rm LT}}{\pi_L^m}\right) \in L_m[[t_{\rm LT}]]$ is $\eta_m$, which is non-zero. Hence $h$ is a unit in $L_m[[t_{\rm LT}]]$ and, for any $f\in o_L[Z,Z^{-1}]$, we see that $ f(h)$ is a well-defined element in $L_m[[t_{\rm LT}]]$. Since $|\eta_m|_p<1$, we see  by Proposition \ref{propconvh} that $f(h)$ is also well-defined for any $f \in o_L((Z))$. Proposition \ref{prop:psisub} in particular implies $\bA_L^{\psi_L=1} \subset o_L((Z))$, so the claim follows.
\end{proof}

\begin{remark}\label{rem:philog}
By Proposition \ref{propconvh} and the proof of Corollary \ref{corpsi}, one can actually define $\varphi_L^{-m}$ on
$$\left\{f=\sum_{n\in \ZZ} a_n Z^n \in L((Z)) \ \middle| \  \underset{n\to +\infty}{\lim} a_n \eta_m^n =0\right\}.$$
In particular, $\varphi_L^{-m}(t_{\rm LT})$ is defined. (It is equal to $\pi_L^{-m}t_{\rm LT}$: see Lemma \ref{lemtlt} below.)
\end{remark}

We shall prove Proposition \ref{prop:psisub}. We need some lemmas.

\begin{lemma}\label{lempsi0}
We have
$$\psi_L( o_L[[Z]]) \subseteq o_L[[Z]].$$
\end{lemma}

\begin{proof}
Let $\psi_{\rm Col}$ be the Coleman trace operator defined in \cite[\S 2]{SV15}: it is characterized by
$$\varphi_L\circ \psi_{\rm Col}(f)= \sum_{a\in {\rm LT}_1}f(a+_{\rm LT} Z)$$
for any $f\in o_L[[Z]]$.
By \cite[Rem. 3.2(ii)]{SV15}, we have $\psi_L = \pi_L^{-1}\cdot \psi_{\rm Col}$ on $o_L[[Z]]$, so it is sufficient to prove $\psi_{\rm Col}(o_L[[Z]]) \subseteq \pi_L o_L[[Z]]$. Let $\mathfrak{m}_{L_1}$ be the maximal ideal of $L_1=L({\rm LT}_1)$. Then, noting ${\rm LT}_1\subset \mathfrak{m}_{L_1}$ and $\# {\rm LT}_1 =q \in \mathfrak{m}_{L_1}$, we have
$$\varphi_L\circ \psi_{\rm Col}(f)=\sum_{a\in {\rm LT}_1}f(a+_{\rm LT} Z)\equiv q f(Z) \equiv 0\text{ (mod $\mathfrak{m}_{L_1}$)}.$$
This implies that each coefficient of $\varphi_L\circ \psi_{\rm Col}(f)$ lies in $o_L\cap \mathfrak{m}_{L_1}=\pi_Lo_L$. Hence we have $\varphi_L\circ \psi_{\rm Col}(o_L[[Z]]) \subseteq \pi_L o_L[[Z]]$, which immediately implies $\psi_{\rm Col}(o_L[[Z]]) \subseteq \pi_L o_L[[Z]]$. This completes the proof.
\end{proof}

\begin{lemma}\label{lempsi01}
We have
$$\psi_L(Z^{-1}o_L[[Z]]) \subseteq Z^{-1}o_L[[Z]].$$
\end{lemma}

\begin{proof}
By Lemma \ref{lempsi0} and the fact that $g_{\rm LT}(Z)\in o_L[[Z]]^\times$, it is sufficient to prove
$$\psi_L(Z^{-1} g_{\rm LT}(Z)^{-1}) \in Z^{-1}o_L[[Z]].$$
By \cite[Lem. 2.5 and Rem. 3.2(ii)]{SV15}, we have
$$\psi_L(Z^{-1}g_{\rm LT}(Z)^{-1}) = g_{\rm LT}(Z)^{-1} \frac{\cN'(Z)}{\cN(Z)},$$
where $\cN:o_L[[Z]]\to o_L[[Z]]$ denotes the Coleman norm operator and $\cN'(Z)$ is the derivative of $\cN(Z)$. Since we have $\cN(Z) \in Z\cdot o_L[[Z]]^\times$ (see \cite[Lem. 5.8(iii)]{iwasawa} for example), we see that
$$\frac{\cN'(Z)}{\cN(Z)} \in Z^{-1}o_L[[Z]].$$
This completes the proof.
\end{proof}

%\begin{remark}
%The proof of Lemma \ref{lempsi01} shows that, if $\cN(Z)=Z$, then
%$$\frac{\partial_{\rm inv}(Z)}{Z}=Z^{-1}g_{\rm LT}(Z)^{-1} \in \bA_L^{\psi_L=1}.$$
%Note also that $\cN(Z)=Z$ is equivalent to $(\eta_n)_n \in \varprojlim_n L_n^\times$ (see \cite[Lem. 8.3]{iwasawa}).
%\end{remark}

\begin{lemma}\label{lempsi1}
For any $f\in \bA_L$ and $k\in \ZZ$, we have
$$\psi_L(f)\equiv Z^k \psi_L(Z^{-qk}f) \text{ (mod $\pi_L$)}.$$
\end{lemma}

\begin{proof}
First, note that $\varphi_L(Z)\equiv Z^q$ (mod $\pi_L$). This implies
$$\psi_L(Z^q f) \equiv \psi_L(\varphi_L(Z)f) \text{ (mod $\pi_L$)}.$$
By the projection formula (\ref{proj formula}), the right hand side is equal to $Z\psi_L(f)$. Hence we have
$$\psi_L(Z^q f) \equiv Z\psi_L(f) \text{ (mod $\pi_L$)}.$$
This proves the claim for $k=-1$. If we replace $f$ by $Z^{-q}f$, the claim for $k=1$ follows. The general case follows easily by induction.
\end{proof}

We set
$$\bE_L:=k_L((Z)) = \bA_L/\pi_L \bA_L \text{ and }\bE_L^+:= k_L[[Z]].$$
Let $v= v_{\bE_L}: \bE_L \to \ZZ\cup \{\infty\}$ be the usual discrete valuation such that $v(Z)=1$. Let
$$\overline \psi_L: \bE_L \to \bE_L$$
be the map induced by $\psi_L$.

\begin{lemma}\label{lempsi2}
For any $f \in \bE_L$, we have
$$v(\overline \psi_L(f))\geq \left[ \frac{v(f)}{q}\right].$$
Here for a real number $a$ we write $[a] $ for the integer such that $0\leq a-[a]<1$.
\end{lemma}

\begin{proof}
We set $k:= \left[ v(f)/q\right]$. Then we have $qk \leq v(f)$, i.e., $Z^{-qk}f \in \bE_L^+$. So by Lemma \ref{lempsi0} we have $v(\overline \psi_L(Z^{-qk}f))\geq 0$. Hence we have
$$v(\overline \psi_L(f)) = k + v(\overline \psi_L(Z^{-qk}f)) \geq k,$$
where the first equality follows from Lemma \ref{lempsi1}. This proves the lemma.
\end{proof}

We now prove Proposition \ref{prop:psisub}.
\begin{proof}[Proof of Proposition \ref{prop:psisub}]
By Lemma \ref{lempsi01}, we have a well-defined map
$$\psi_L-1: \bA_L/ Z^{-1}o_L[[Z]] \to \bA_L/ Z^{-1}o_L[[Z]] .$$
It is sufficient to show that this map is injective. To show this, it is sufficient to prove the injectivity of the induced map
$$\overline \psi_L- 1 : \bE_L/Z^{-1}\bE_L^+ \to \bE_L/Z^{-1}\bE_L^+.$$
In fact, consider the commutative diagram with exact rows:
$$\xymatrix{
0\ar[r] & \bA_L/ Z^{-1}o_L[[Z]] \ar[r]^{\pi_L} \ar[d]^{\psi_L-1}& \bA_L/ Z^{-1}o_L[[Z]] \ar[r] \ar[d]^{\psi_L-1}&\bE_L/Z^{-1}\bE_L^+ \ar[r] \ar[d]^{\overline \psi_L-1}& 0\\
0\ar[r] & \bA_L/ Z^{-1}o_L[[Z]] \ar[r]^{\pi_L}& \bA_L/ Z^{-1}o_L[[Z]] \ar[r]&\bE_L/Z^{-1}\bE_L^+ \ar[r]& 0.
}$$
By Snake Lemma, if $\overline \psi_L-1$ is injective, then multiplication by $\pi_L$ induces an automorphism of $\ker(\psi_L-1)$, and so $\ker(\psi_L-1)= \bigcap_n \pi_L^n \ker(\psi_L-1)=0$.

Hence we are reduced to proving the injectivity of $\overline \psi_L-1$. Take an element $f\in \bE_L$ whose image in $\bE_L/Z^{-1}\bE_L^+$ lies in the kernel of $\overline \psi_L-1$. Then we have $\overline \psi_L(f) -f \in Z^{-1}\bE_L^+$, i.e., $v(\overline \psi_L(f)-f)\geq -1$. If $v(f)<-1$, then one easily sees that $[v( f)/q] > v( f)$ (since $q\geq 2$). So by Lemma \ref{lempsi2} we have the following implications:
$$v( f)<-1 \Rightarrow v(\overline \psi_L( f)) > v( f) \Rightarrow v(\overline \psi_L( f) -  f) = v( f)<-1 .$$
The last condition contradicts $v(\overline \psi_L(f)-f)\geq -1$. Thus we have $v( f)\geq -1$, i.e., $f$ is zero in $\bE_L/Z^{-1}\bE_L^+$. This proves the injectivity of the map above. Hence we have proved the proposition.
\end{proof}

\subsection{Some calculations}

By the proof of Corollary \ref{corpsi}, we know that $\varphi_L^{-m}$ is well-defined on $o_L((Z))$. For $m,j\geq 1$, we define the evaluation map
$${\rm ev}_{m,j}: o_L((Z)) \to L_m$$
by assigning $f \in o_L((Z))$ to the coefficient of $t_{\rm LT}^{j-1}$ in $\pi_L^{-m}\varphi_L^{-m}(f) \in L_m[[t_{\rm LT}]]$.

In this section, we give an explicit description of this map:

\begin{proposition}\label{prop:explicit}
For any $f\in o_L((Z))$, we have
\begin{equation*}
{\rm ev}_{m,j}(f) = \frac{1}{(j-1)! \pi_L^{mj}} \bigg(\partial_{\rm inv}^{j-1} f\bigg)|_{Z=\eta_m}.
\end{equation*}
\end{proposition}

In what follows, to simplify the notation, we often omit the subscript ``${\rm LT}$". In particular, we write $t$ for $t_{\rm LT}$.

%The following two lemmas are well-known (and easily verified).

\begin{lemma}\label{lemder}
Let $K$ be a subfield of $\CC_p$ containing the coefficients of $t=\log_{\rm LT}(Z) \in L[[Z]]$.
Under the identification $K((Z))=K((t))$, the map $\partial_{\rm inv}: K((Z))\to K((Z))$ coincides with the derivative map $\frac{d}{dt}: K((t)) \to K((t)). $
\end{lemma}

\begin{proof}
Recall that $\partial_{\rm inv}:=g(Z)^{-1}\frac{d}{dZ}$, where $g(Z):=\frac{d}{dZ}\log_{\rm LT}(Z) \in o_L[[Z]]^\times$. For $n\in \ZZ$, we compute
$$\partial_{\rm inv}(t^n)= g(Z)^{-1}\frac{d}{dZ}(\log_{\rm LT}(Z)^n) = g(Z)^{-1}\cdot n \log_{\rm LT}(Z)^{n-1}\cdot g(Z)= nt^{n-1}. $$
Thus $\partial_{\rm inv}$ coincides with the derivative $\frac{d}{dt}$.
\end{proof}

\begin{lemma}[{\cite[Rem. 3.14]{MSVW}}]\label{lemtlt}
We have
$$\varphi_L^{-m}(t_{\rm LT})=\pi_L^{-m} t_{\rm LT}.$$
 (Note that $\varphi_L^{-m}(t_{\rm LT})$ is defined by Remark \ref{rem:philog}.)
\end{lemma}

\begin{proof}
We compute
$$\varphi_L^{-m}(t_{\rm LT})=\log_{\rm LT}(\varphi_L^{-m}(Z))= \log_{\rm LT}\left(\eta_m +_{\rm LT} \exp_{\rm LT}\left(\frac{t_{\rm LT}}{\pi_L^m}\right)\right)= \log_{\rm LT}(\eta_m)+ \frac{t_{\rm LT}}{\pi_L^m}.$$
So it is sufficient to prove $\log_{\rm LT}(\eta_m)=0$. But this follows by noting
$$\pi_L^m \cdot \log_{\rm LT}(\eta_m) = \log_{\rm LT}([\pi_L^m](\eta_m))= \log_{\rm LT}(0)=0.$$
\end{proof}

The following result is proved in the cyclotomic case in \cite[Lem. III.2.3(ii)]{CC}.

\begin{lemma}\label{lem:invphi}
We have
$$\partial_{\rm inv}\circ \varphi_L^{-m} = \pi_L^{-m}\varphi_L^{-m}\circ \partial_{\rm inv}$$
on $\{f=\sum_{n\in \ZZ} a_n Z^n \in L((Z)) \mid \lim_{n\to +\infty} a_n \eta_m^n =0\}$. (See Remark \ref{rem:philog}.)
\end{lemma}

\begin{proof}
We set $\pi:=\pi_L$ and $\varphi^{-m}:=\varphi_L^{-m}$.
By Lemma \ref{lemder}, it is sufficient to prove
$$\frac{d}{dt}\varphi^{-m}(Z^n) = \pi^{-m}\varphi^{-m}(\partial_{\rm inv}(Z^n))$$
for any $n\in \ZZ$. Since by computation we have
$$\frac{d}{dt}\varphi^{-m}(Z^n) =\frac{d}{dt}\varphi^{-m}(Z)^n= n\varphi^{-m}(Z)^{n-1}\cdot \frac{d}{dt}\varphi^{-m}(Z) $$
and
$$\pi^{-m}\varphi^{-m}(\partial_{\rm inv}(Z^n))= \pi^{-m}\varphi^{-m}(g(Z)^{-1}\cdot nZ^{n-1})= n\varphi^{-m}(Z)^{n-1}\cdot \pi^{-m}\varphi^{-m}(g(Z)^{-1}),$$
it is sufficient to prove the $n=1$ case:
\begin{equation}\label{ddtphi}
\frac{d}{dt}\varphi^{-m}(Z)=\pi^{-m} \varphi^{-m}(g(Z)^{-1}).
\end{equation}

We shall prove (\ref{ddtphi}). By Lemma \ref{lemtlt}, we have
\begin{equation}\label{ddt1}
\frac{d}{dt} \varphi^{-m}(t)=\pi^{-m}.
\end{equation}
On the other hand, noting $g(Z)=\frac{d}{dZ}\log_{\rm LT}(Z)$, we compute
\begin{equation}\label{ddt2}
\frac{d}{dt}\varphi^{-m}(t)=\frac{d}{dt}\log_{\rm LT}(\varphi^{-m}(Z)) = g(\varphi^{-m}(Z))\frac{d}{dt}\varphi^{-m}(Z).
\end{equation}
By (\ref{ddt1}) and (\ref{ddt2}), we obtain the desired equality (\ref{ddtphi}).
\end{proof}

Now we prove Proposition \ref{prop:explicit}.

\begin{proof}
In general, the coefficient of $t^{j-1}$ in $h \in L_m[[t]]$ is
$$\left. \frac{1}{(j-1)!}\left(\frac{d}{dt}\right)^{j-1}h\right|_{t=0}.$$
Hence, by the definition of ${\rm ev}_{m,j}$, we have
\begin{equation}\label{ev def eq}
{\rm ev}_{m,j}(f)=\left. \frac{1}{(j-1)! \pi_L^m} \left(\frac{d}{dt}\right)^{j-1}\varphi^{-m}_{L}(f)\right|_{t=0}
\end{equation}
for any $f\in o_L((Z))$. By Lemmas \ref{lemder} and \ref{lem:invphi}, we have
$$\left(\frac{d}{dt}\right)^{j-1}\varphi^{-m}_{L}(f)= \pi_L^{-m(j-1)}\varphi_L^{-m}(\partial_{\rm inv}^{j-1}f).$$
By the definition of $\varphi_L^{-m}$, we have $\varphi_L^{-m}(F)|_{t=0}=F(\eta_m)$ for any $F \in o_L((Z))$, so we have
\begin{equation}\label{t0eq}
\left. \left(\frac{d}{dt}\right)^{j-1}\varphi^{-m}_{L}(f)\right|_{t=0}= \pi_L^{-m(j-1)}\partial_{\rm inv}^{j-1}f|_{Z=\eta_m}.
\end{equation}
By (\ref{ev def eq}) and (\ref{t0eq}), we obtain the desired formula
$${\rm ev}_{m,j}(f) = \frac{1}{(j-1)! \pi_L^{mj}} \partial_{\rm inv}^{j-1} f|_{Z=\eta_m}.$$
\end{proof}

\section{The cyclotomic case}\label{sec:bigexp}

The aim of this appendix is to show that, in the cyclotomic case, two definitions of the map ${\rm Exp}^\ast$ due to Cherbonnier-Colmez \cite{CC} and Schneider-Venjakob \cite{SV15} coincide (see Proposition \ref{prop comparison}). In the proof, we use a result of Benois \cite{benois} on an explicit description of the invariant map in local class field theory in terms of $(\varphi,\Gamma)$-modules. In \S \ref{sec:invariant}, we give a proof of this result by a different argument from loc. cit. (see Theorem \ref{thm:invariant}).

Furthermore, for the reader's convenience, we provide a direct proof of Theorem \ref{key claim} in the cyclotomic case in \S \ref{sec cyc erl}, following the argument in \cite{CC} and \cite{B}.

Throughout this appendix, we let $L=\QQ_p$ and consider the case ${\rm LT}=\widehat \GG_m:=X+Y+XY$, which is the Lubin-Tate formal group associated with $(1+Z)^p-1$. In this case, the $p^n$-division points are $\widehat \GG_m[p^n]=\{\zeta -1 \mid \zeta \in \mu_{p^n}\}$ and so $L_n=\QQ_p(\mu_{p^n})$. We identify the Tate module of $\widehat \GG_m$ with $\ZZ_p(1)$ and fix its generator $\epsilon=(\zeta_{p^n})_n$.

\subsection{The invariant map}\label{sec:invariant}

Fix a positive integer $m$ and let
$${\rm inv}: H^1(L_m,\ZZ_p(1)) \xrightarrow{\sim} \ZZ_p$$
be the invariant map in local class field theory. An explicit description of this map in terms of $(\varphi,\Gamma)$-modules was given by Benois in \cite[Thm. 2.2.6]{benois}. We shall give a proof of this result by a different argument (see Theorem \ref{thm:invariant} below). The main ingredients in our proof are Theorem \ref{thm sv} and the classical explicit reciprocity law due to Iwasawa.

We first recall some basic facts on $(\varphi,\Gamma)$-modules. We basically follow notations in \cite{CC}. Let $V$ be a $\ZZ_p$-representation of $G_{\QQ_p}$. Fix a topological generator $\gamma_m \in \Gamma_{L_m}:=\Gal(L_\infty/L_m)$. Let $D(V):=(\mathbf{A}\otimes_{\ZZ_p} V)^{G_{L_\infty}}$ be the $(\varphi,\Gamma_{L_m})$-module associated with $V$. Let $C_{\varphi,\gamma_m}(V)=C_{\varphi,\gamma_m}(L_m,V)$ be the Herr complex defined for $\varphi$ and $\gamma_m$:
$$C_{\varphi, \gamma_m}( V) :  D(V)\xrightarrow{x \mapsto ((\varphi-1)x, (\gamma_m-1)x)} D(V)\oplus D(V)\xrightarrow{(x,y)\mapsto (\gamma_m-1)x -(\varphi-1)y} D(V).$$
Then we have canonical isomorphisms:
$$\iota^i: H^i(C_{\varphi,\gamma_m}(V))\xrightarrow{\sim} H^i(L_m,V) $$
(see \cite[Thm. 2.1]{herr1}). We know that $\iota^1$ is explicitly given by
\begin{equation}\label{iota1}
 [(x,y)]\mapsto \left[g\mapsto \frac{g-1}{\gamma_m-1}y - (g-1)b \right] \quad (x,y\in D(V), \ g\in G_{L_m}),
 \end{equation}
where $b\in \bA\otimes_{\ZZ_p}V$ is an element satisfying $(\varphi-1)b=x$ (see \cite[Prop. I.4.1]{CC} or \cite[Prop. 1.3.2]{benois}). When $V=\ZZ_p(1)$, we have a natural isomorphism
$$D(\ZZ_p(1)) = \bA_{\QQ_p}(1) \xrightarrow{\sim} \Omega^1; \ f\otimes \epsilon \mapsto f \frac{dZ}{1+Z},$$
where $\Omega^1:=\Omega_{\bA_{\QQ_p}/\ZZ_p}^1$ is the free $\bA_{\QQ_p}$-module of differential forms (see \cite[Thm. 3.7]{herr} or \cite[Lem. 3.14]{SV15}). Also, we have a residue map
$${\rm res}: \Omega^1\to \ZZ_p; \ \left(\sum_{i\in \ZZ}a_i Z^i\right)dZ\mapsto a_{-1}.$$
Identifying $D(\ZZ_p(1))$ with $\Omega^1$, we regard
$${\rm res}: D(\ZZ_p(1)) \to \ZZ_p.$$
This map factors through
$${\rm res}: H^2(C_{\varphi,\gamma_m}(\ZZ_p(1))) \to \ZZ_p$$
(see \cite[Lem. 5.1]{herr} or \cite[Prop. 3.17]{SV15}).

\begin{theorem}[{Benois \cite[Thm. 2.2.6]{benois}}]\label{thm:invariant}
The diagram
$$\xymatrix{
H^2(C_{\varphi,\gamma_m}(\ZZ_p(1))) \ar[d]_{\iota^2}^-\simeq\ar[r]^-{\rm res}& \ZZ_p\ar[d]_\simeq^-{\frac{p^m}{\log \chi_{\rm cyc}(\gamma_m)}} \\
H^2(L_m,\ZZ_p(1)) \ar[r]^-\simeq_-{\rm inv}& \ZZ_p
}$$
is commutative. In particular, ${\rm res}$ is an isomorphism (note that $\frac{p^m}{\log \chi_{\rm cyc}(\gamma_m)} \in \ZZ_p^\times$).
%The map
%$${\rm inv}\circ \iota^2: H^2(C_{\varphi,\gamma_m}(\ZZ_p(1))) \to \ZZ_p$$
%coincides with
%$$\frac{p^m}{\log \chi_{\rm cyc}(\gamma_m)}{\rm res}.$$
\end{theorem}

\begin{remark}\label{rem benois sign}
The sign of the formula in \cite[\S 2.2.3]{benois} is different from (\ref{kummer normalize}) in \S \ref{rem sign}, and this is the reason why our statement of Theorem \ref{thm:invariant} differs by a sign from \cite[Thm. 2.2.6]{benois}.
%
%It seems to the authors that the proof of \cite[Thm. 2.2.6]{benois} has a sign error: the same formula as \cite[Lem. II.1.4.5]{katolecture} is used in \cite[\S 2.2.3]{benois}, which is not correct as we explained in \S \ref{rem sign}. This is the reason why our statement of Theorem \ref{thm:invariant} differs by a sign from \cite[Thm. 2.2.6]{benois}.
\end{remark}

We shall give a proof of Theorem \ref{thm:invariant} by an argument different from \cite{benois}.

For $a\in L_\infty$, choose $n\geq 1$ such that $a\in L_n$ and set
\begin{equation}\label{defT}
{\rm T}(a):=\frac{1}{p^n} {\rm Tr}_{L_n/\QQ_p}(a).
\end{equation}
One sees easily that this is independent of the choice of $n$.

Fix $k\geq 1$. In the following, we abbreviate $X/p^k X$ to $X/p^k$ for a $\ZZ_p$-module $X$. Also, we abbreviate $M/(\varphi-1)M$ to $\frac{M}{\varphi-1}$ for a $(\varphi,\Gamma)$-module $M$.

Recall that we have an exact sequence
$$0\to \ZZ/p^k(1)\to \bA(1)/p^k\xrightarrow{\varphi-1} \bA(1)/p^k\to 0$$
(see \cite[Rem. 5.1]{SV15}). We know that $H^1(L_\infty, \bA(1)/p^k)=0$ and let
$$\delta_k: \dfrac{\bA_{\QQ_p}(1)/p^k}{\varphi-1} \xrightarrow{\sim} H^1(L_\infty,\ZZ/p^k(1))$$
be the isomorphism induced by the connecting homomorphism (see \cite[Lem. 5.2]{SV15}). We define a pairing
$$\langle \cdot, \cdot \rangle: \dfrac{\bA_{\QQ_p}(1)/p^k}{\varphi-1} \times (\bA_{\QQ_p}/p^k)^{\psi=1}\to \ZZ/p^k$$
by setting
$$\langle f, g \rangle := {\rm res}\left(fg \right).$$
(We regard ${\rm res}$ as a map $\bA_{\QQ_p}(1)/p^k \to \ZZ/p^k$ by the identification $\bA_{\QQ_p}(1)/p^k=\Omega^1/p^k$.) This pairing is perfect (see \cite[Lem. 3.6 and Prop. 3.19]{SV15}). By definition, the map
%{\color{blue} I am hesitating to call this map also big dual exponential map, as this name is already used for (generalisations of) Perrin Riou's map $\mathbf{\Omega}$; of course Colmez' reciprocity law might justify this name, too.}
$${\rm Exp}^\ast: H^1_{\rm Iw}(L_\infty/\QQ_p, \ZZ/p^k) \xrightarrow{\sim} (\bA_{\QQ_p}/p^k)^{\psi=1}$$
fits into the commutative diagram
$$\xymatrix{
\dfrac{\bA_{\QQ_p}(1)/p^k}{\varphi-1} \ar[d]_{\delta_k}^\simeq & \times & (\bA_{\QQ_p}/p^k)^{\psi=1}  \ar[r]^-{\langle \cdot,\cdot\rangle}& \ZZ/p^k\ar@{=}[d] \\
H^1(L_\infty,\ZZ/p^k(1)) &\times & H^1_{\rm Iw}(L_\infty/\QQ_p, \ZZ/p^k)\ar[u]^\simeq_{{\rm Exp}^\ast} \ar[r]& \ZZ/p^k,
}$$
where the bottom pairing is the composite of the cup product and the invariant map. (See the diagram on \cite[p.441]{SV15}.) Let
$$(\cdot,\cdot)_k: L_\infty^\times \times \varprojlim_n L_n^\times \to  \ZZ/p^k$$
be the classical Hilbert symbol, which fits into the commutative diagram
$$\xymatrix{
H^1(L_\infty,\ZZ/p^k(1)) &\times & H^1_{\rm Iw}(L_\infty/\QQ_p, \ZZ/p^k) \ar[r]& \ZZ/p^k \ar@{=}[d]\\
L_\infty^\times \ar[u]^{\kappa_k}&\times &\varprojlim_n L_n^\times \ar[u]_-{{\rm Kum}\otimes\epsilon^{\otimes(-1)}} \ar[r]_-{(\cdot,\cdot)_k}& \ZZ/p^k.
}$$
Here $\kappa_k : L_\infty^\times \to H^1(L_\infty, \ZZ/p^k(1))$ denotes the Kummer map.
By Theorem \ref{thm sv} and the commutative diagrams above, we have
$$(a,u)_k = \langle \delta_k^{-1}(\kappa_k(a)), D\log (g_{u})\rangle  = {\rm res}\left(\delta_k^{-1}(\kappa_k(a)) D\log(g_u)\right)$$
for any $a\in L_\infty^\times$ and $u\in \varprojlim_n L_n^\times$, where $g_u=g_{u,\epsilon}$ denotes the Coleman power series of $u $ and we set $D\log (f):=(1+Z)f'/f$ (see also \cite[Prop. 2.4.3]{fontaine}). Let $\mathcal{U}_n:=U_{L_n}^{(1)}$ be the group of principal units of $L_n$ and set $\mathcal{U}_\infty:=\bigcup_n \mathcal{U}_n$. Let $\log: \mathcal{U}_\infty\to L_\infty$ be the usual logarithm map. Then, by the classical explicit reciprocity law of Iwasawa (see \cite[Thm. 8.16]{iwasawa} for example), we have
$$(a,u)_k\equiv \frac{1}{p^m} {\rm Tr}_{L_m/\QQ_p}\left(\log(a) D\log(g_u)|_{Z=\zeta_{p^m}-1}\right) \text{ (mod $p^k$)}$$
for any $a\in \mathcal{U}_\infty$ and $u\in \varprojlim_n L_n^\times$, where $m$ is a sufficiently large integer so that $a\in \mathcal{U}_m$ and $m\geq k$. (The right hand side actually lies in $\ZZ_p$.) Hence we have
$${\rm res}\left(\delta_k^{-1}(\kappa_k(a)) D\log(g_u)\right) \equiv \frac{1}{p^m} {\rm Tr}_{L_m/\QQ_p}\left(\log(a) D\log(g_u)|_{Z=\zeta_{p^m}-1}\right) \text{ (mod $p^k$)}.$$
We now take $u:= \epsilon$ so that $g_u=1+Z$ and $D\log(g_u)= 1$. Then for any $a\in \mathcal{U}_\infty$ we have
\begin{equation*}
{\rm res}\left(\delta_k^{-1}(\kappa_k(a)) \right) \equiv \frac{1}{p^m} {\rm Tr}_{L_m/\QQ_p}\left(\log(a) \right) = {\rm T}(\log(a)) \text{ (mod $p^k$)}
\end{equation*}
(recall that ${\rm T}$ is defined in (\ref{defT})). Since this equality holds for any $k\geq 1$, we obtain the following.

\begin{proposition}\label{prop:kummer residue}
For any $a\in \mathcal{U}_\infty$, we have
\begin{equation*}
{\rm res}\left(\delta^{-1}(\kappa(a)) \right)  = {\rm T}(\log(a)) \text{ in }\ZZ_p,
\end{equation*}
where $\kappa: \mathcal{U}_\infty\subset L_\infty^\times \to H^1(L_\infty,\ZZ_p(1))$ denotes the Kummer map and
$$\delta: \frac{\bA_{\QQ_p}(1)}{\varphi-1} \to H^1(L_\infty, \ZZ_p(1))$$
is the connecting homomorphism associated with the short exact sequence
$$0\to \ZZ_p(1)\to \bA(1)\xrightarrow{\varphi-1}\bA(1)\to 0.$$
Note that $\delta$ is an isomorphism by the argument of \cite[Lem. 5.2]{SV15}, see also  \cite[\S 2.4.1]{fontaine}.
\end{proposition}

\begin{lemma}\label{lem:iotadelta}
The following diagram is commutative:
$$\xymatrix{
H^1(C_{\varphi,\gamma_m}(\ZZ_p(1))) \ar[d]_-{-\iota^1}^\simeq\ar[rr]^-{[(x,y)]\mapsto [x]} && \dfrac{\bA_{\QQ_p}(1)}{\varphi-1} \ar[d]_\simeq^\delta \\
H^1(L_m,\ZZ_p(1)) \ar[rr]_{{\rm Res}_{L_\infty/L_m}}&& H^1(L_\infty,\ZZ_p(1)) .
}$$
\end{lemma}

\begin{proof}
Take any $[(x,y)] \in H^1(C_{\varphi,\gamma_m}(\ZZ_p(1))) $. By (\ref{iota1}), we have
$$-\iota^{1}([(x,y)]) = \left[ g \mapsto (g-1)b - \frac{g-1}{\gamma_m-1}y \right]$$
with $b\in \bA(1)$ such that $(\varphi-1)b=x$. Since $G_{L_\infty}$ acts trivially on $\bA_{\QQ_p}(1)$, we have
$$- {\rm Res}_{L_\infty/L_m}\circ \iota^1 ([(x,y)]) = [g\mapsto (g-1)b] .$$
On the other hand, by the definition of connecting homomorphisms, we have
$$\delta([x])= [g \mapsto (g-1)b].$$
Thus we have
$$- {\rm Res}_{L_\infty/L_m}\circ \iota^1 ([(x,y)])=\delta([x]).$$
\end{proof}

\begin{lemma}[{\cite[Prop. 4.4(f)]{herr}}]\label{lem:iotacup}
The following diagram is commutative:
$$\xymatrix{
H^1(C_{\varphi,\gamma_m}(\ZZ_p(1))) \ar[d]^{\iota^1}& \times & H^1(C_{\varphi,\gamma_m}(\ZZ_p)) \ar[d]^{\iota^1}\ar[r]& H^2(C_{\varphi,\gamma_m}(\ZZ_p(1))) \ar[d]^{\iota^2}\\
H^1(L_m, \ZZ_p(1)) &\times &H^1(L_m,\ZZ_p) \ar[r]^\cup& H^2(L_m,\ZZ_p(1)) ,
}$$
where the upper pairing is given by
$$([(x,y)], [(z,t)]) \mapsto [y\gamma_m(z) - x \varphi(t)].$$
\end{lemma}

We now prove Theorem \ref{thm:invariant}.

\begin{proof}[Proof of Theorem \ref{thm:invariant}]
%It is sufficient to prove that the diagram
%$$\xymatrix{
%H^2(C_{\varphi,\gamma_m}(\ZZ/p^k(1))) \ar[d]^{\iota^2}\ar[r]^-{\rm res}& \ZZ/p^k \ar[d]^{\frac{p^m}{\log \chi_{\rm cyc}(\gamma_m)}} \\
%H^2(L_m,\ZZ/p^k(1)) \ar[r]^-{\rm inv}& \ZZ/p^k
%}$$
%is commutative for any $k\geq 1$.
We regard $\log \chi_{\rm cyc}$ as an element of $H^1(L_m,\ZZ_p)=\Hom_{\rm cont}(G_{L_m}, \ZZ_p)$. By (\ref{iota1}), we see that
$$\iota^1([0,\log \chi_{\rm cyc}(\gamma_m)]) = \log \chi_{\rm cyc}.$$
Take any $a \in \mathcal{U}_m:=U_{L_m}^{(1)}$ and let $[(x,y)] \in H^1(C_{\varphi,\gamma_m}(\ZZ_p(1)))$ be the element such that $\iota^1([(x,y)]) = \kappa(a)$ in $H^1(L_m,\ZZ_p(1))$. Then Lemma \ref{lem:iotacup} implies
$$-\log \chi_{\rm cyc}(\gamma_m) \iota^2([x]) = \kappa(a)\cup \log \chi_{\rm cyc}.$$
Hence we have
\begin{equation}\label{inviota2}
{\rm inv}\circ \iota^2([x]) = -\frac{1}{\log \chi_{\rm cyc}(\gamma_m)} {\rm inv}(\kappa(a)\cup \log \chi_{\rm cyc}) = -\frac{p^m}{\log \chi_{\rm cyc}(\gamma_m)} {\rm T}(\log(a)),
\end{equation}
where the second equality follows from the fact that
$${\rm inv}(\kappa(a)\cup \log \chi_{\rm cyc}) = {\rm Tr}_{L_m/\QQ_p}(\log(a))$$
(see \cite[Lem. II.1.4.5]{katolecture} and \S \ref{rem sign} for the sign) and the definition of ${\rm T}$ (see (\ref{defT})). On the other hand, by Lemma \ref{lem:iotadelta}, we have
$$[x] = -\delta^{-1}(\kappa(a)) \text{ in }\dfrac{\bA_{\QQ_p}(1)}{\varphi-1}.$$
Using this equality and Proposition \ref{prop:kummer residue}, we obtain
\begin{equation}\label{resx}
{\rm res}([x]) = - {\rm T}(\log(a)).
\end{equation}
By (\ref{inviota2}) and (\ref{resx}), we see that the composition map
$$\ZZ_p \xrightarrow{{\rm inv}^{-1}} H^2(L_m,\ZZ_p(1)) \xrightarrow{(\iota^2)^{-1}} H^2(C_{\varphi,\gamma_m}(\ZZ_p(1))) \xrightarrow{{\rm res}} \ZZ_p$$
sends ${\rm T}(\log(a))$ to $\frac{\log\chi_{\rm cyc}(\gamma_m)}{p^m} {\rm T}(\log (a))$. Since one can choose $a$ so that ${\rm T}(\log(a))\neq 0$, we see that this composition map is given by multiplication by $\frac{\log \chi_{\rm cyc}(\gamma_m)}{p^m}$, i.e., the diagram
$$\xymatrix{
H^2(C_{\varphi,\gamma_m}(\ZZ_p(1))) \ar[d]_{\iota^2}\ar[r]^-{\rm res}& \ZZ_p\ar[d]^-{\frac{p^m}{\log \chi_{\rm cyc}(\gamma_m)}} \\
H^2(L_m,\ZZ_p(1)) \ar[r]_-{\rm inv}& \ZZ_p
}$$
is commutative. This proves the theorem.
\end{proof}

\subsection{Comparison}
%In this case, we check that the big dual exponential map of Schneider-Venjakob \cite{SV15} coincides with that of Cherbonnier-Colmez \cite{CC}.
Let $V$ be a $\ZZ_p$-representation of $G_{\QQ_p}$. There are two definitions of ${\rm Exp}^\ast$ due to Cherbonnier-Colmez \cite{CC} and Schneider-Venjakob \cite{SV15}:
$${\rm Exp}^\ast_{\rm CC}, {\rm Exp}^\ast_{\rm SV}  : H^1_{\rm Iw}(L_\infty/\QQ_p,V)\xrightarrow{\sim} D(V)^{\psi=1}.$$
We show that these maps coincide.

\begin{proposition}\label{prop comparison}
We have $ {\rm Exp}_{\rm CC}^\ast = {\rm Exp}_{\rm SV}^\ast . $
\end{proposition}

We first recall the definition due to Schneider-Venjakob. In the following, we assume that $V$ is of finite length (i.e., $p$-torsion): the general case follows by taking limits. We set $V^\vee(1):=\Hom_{\ZZ_p}(V,\QQ_p/\ZZ_p(1))$.

By \cite[Rem. 5.1]{SV15}, we have an exact sequence
$$0\to V^\vee(1)\to \bA \otimes_{\ZZ_p}V^\vee(1) \xrightarrow{\varphi-1} \bA\otimes_{\ZZ_p}V^\vee(1)\to 0.$$
Then the connecting homomorphism associated with this sequence induces an isomorphism
\begin{equation}\label{def delta}
\delta: \frac{D(V^\vee(1))}{\varphi-1} \xrightarrow{\sim} H^1(L_\infty, V^\vee(1))
\end{equation}
(see \cite[Lem. 5.2]{SV15}).

Also, there is a canonical pairing
\begin{equation}\label{Dpairing}
\langle \cdot,\cdot \rangle: D(V^\vee(1))\times D(V)\to \QQ_p/\ZZ_p
\end{equation}
%which induces an isomorphism
%$$\tau_D: D(V)^{\psi=1} \xrightarrow{\sim} \left(\frac{D(V^\vee(1))}{\varphi-1}\right)^\vee$$
(see \cite[Rem. 5.6]{SV15}). This is defined as follows. Suppose that $p^k V=0$ and we regard $V^\vee(1)=\Hom_{\ZZ_p}(V, \ZZ/p^k(1))$. Then we have a natural map
$$D(V^\vee(1)) \times D(V) \to D(\ZZ/p^k(1))=\bA_{\QQ_p}(1)/p^k$$
induced by $V^\vee(1)\times V \to \ZZ/p^k(1); \ (f,x)\mapsto f(x)$. Composing this map with the residue map ${\rm res}: \bA_{\QQ_p}(1)/p^k \to \ZZ/p^k$ and the natural injection $\ZZ/p^k \hookrightarrow \QQ_p/\ZZ_p$, we obtain the pairing (\ref{Dpairing}). Note that it induces a pairing
$$\langle \cdot,\cdot \rangle: \frac{D(V^\vee(1))}{\varphi-1}\times D(V)^{\psi=1}\to \QQ_p/\ZZ_p$$
(see \cite[Prop. 3.19]{SV15}).

Let
$$H^1(L_m, V^\vee(1))\times H^1(L_m,V) \xrightarrow{\cup} H^2(L_m, \QQ_p/\ZZ_p(1))\xrightarrow{\rm inv} \QQ_p/\ZZ_p$$
be the local Tate duality pairing.
%we have a canonical isomorphism
%$$\tau_H: H^1_{\rm Iw}(L_\infty/L, V)=\varprojlim_m H^1(L_m,V)\xrightarrow{\sim} \left(\varinjlim H^1(L_m,V^\vee(1))\right)^\vee= H^1(L_\infty,V^\vee(1))^\vee. $$
Schneider-Venjakob's map
$${\rm Exp}^\ast_{\rm SV}: H^1_{\rm Iw}(L_\infty/L,V)\xrightarrow{\sim} D(V)^{\psi=1} $$
is defined to be the map which makes the following diagram commutative:
%given by
%$${\rm Exp}^\ast_{\rm SV}:=\tau_D^{-1}\circ \delta^\vee \circ \tau_H: H^1_{\rm Iw}(L_\infty/L,V)\xrightarrow{\sim} D(V)^{\psi=1}. $$
%By definition, we have the following commutative diagram:
\begin{equation}\label{SVcomm}\xymatrix{
H^1(L_\infty,V^\vee(1)) &\times & H^1_{\rm Iw}(L_\infty/\QQ_p, V)\ar[d]_\simeq^{{\rm Exp}_{\rm SV}^\ast} \ar[r]& \QQ_p/\ZZ_p \ar@{=}[d]\\
\dfrac{D(V^\vee(1))}{\varphi-1} \ar[u]^{\delta}_\simeq & \times & D(V)^{\psi=1}  \ar[r]^-{\langle\cdot,\cdot\rangle}& \QQ_p/\ZZ_p .
}
\end{equation}

Next, we recall the definition due to Cherbonnier-Colmez. Let $\gamma_m$ be a topological generator of $\Gal(L_\infty/L_m)$. Let $C_{\psi,\gamma_m}(L_m,V)$ be the Herr complex defined for $\psi$ and $\gamma_m$ (see \cite[\S I.4]{CC}):
$$C_{\psi, \gamma_m}(L_m, V) :  D(V)\xrightarrow{x \mapsto ((\psi-1)x, (\gamma_m-1)x)} D(V)\oplus D(V)\xrightarrow{(x,y)\mapsto (\gamma_m-1)x -(\psi-1)y} D(V).$$
By definition, we have
$$H^1(C_{\psi,\gamma_m}(L_m,V)) = \frac{\{(x,y) \in D(V)\oplus D(V) \mid (\gamma_m-1)x=(\psi-1)y\}}{\{((\psi-1)z, (\gamma_m-1)z) \mid z\in D(V)\}}.$$
Then by \cite[Prop. I.4.1, Lem. I.4.2, I.5.2]{CC} we have a canonical isomorphism
\begin{equation*}\label{herr isom}
\ell_{\psi,m}:=\ell_{L_m}(\gamma_m) \iota_{\psi,\gamma_m}: H^1(C_{\psi,\gamma_m}(L_m,V))\xrightarrow{\sim}H^1(L_m,V),
\end{equation*}
where $\ell_{L_m}(\gamma_m)$ is defined to be $\frac{\log \chi_{\rm cyc}(\gamma_m)}{p^m}$. (See (\ref{defellpsi}) below for details.)
Also, we have a natural injection
$$\iota_m: \frac{D(V)^{\psi=1}}{\gamma_m-1} \hookrightarrow H^1(C_{\psi,\gamma_m}(L_m,V)); \ [y]\mapsto [(0,y)].$$
(See \cite[Lem. I.5.5]{CC}.) Then one can show that the map
$${\rm Log}^\ast:= \varprojlim_m \ell_{\psi,m}\circ \iota_m: D(V)^{\psi=1}\to H^1_{\rm Iw}(L_\infty/L,V)$$
is an isomorphism (see \cite[Thm. II.1.3]{CC}). The map of Cherbonnier-Colmez ${\rm Exp}_{\rm CC}^\ast$ is defined to be the inverse of this isomorphism.

\begin{remark}
As explained in \cite[\S II.1]{CC}, the isomorphism ${\rm Log}^\ast$ is explicitly given by
$$y \mapsto \left(\left[ g \mapsto \ell_{L_m}(\gamma_m) \cdot\left(\frac{g-1}{\gamma_m-1} y - (g-1)b_m\right)\right] \right)_m, $$
where $b_m\in \bA\otimes_{\ZZ_p} V$ is an element satisfying $(\varphi-1)b_m=(\gamma_m-1)^{-1}(\varphi-1)y$. (Note that $(\varphi-1)y \in D(V)^{\psi=0}$ and $\gamma_m-1$ is invertible on $D(V)^{\psi=0}$.)
\end{remark}

Let $C_{\varphi,\gamma_m}(L_m,V)$ be the Herr complex defined for $\varphi$ and $\gamma_m$:
$$C_{\varphi, \gamma_m}(L_m, V) :  D(V)\xrightarrow{x \mapsto ((\varphi-1)x, (\gamma_m-1)x)} D(V)\oplus D(V)\xrightarrow{(x,y)\mapsto (\gamma_m-1)x -(\varphi-1)y} D(V).$$
By \cite[Prop. I.4.1]{CC}, we have a canonical isomorphism
$$-\iota_{\varphi,\gamma_m}: H^1(C_{\varphi,\gamma_m}(L_m,V^\vee(1))) \xrightarrow{\sim} H^1(L_m,V^\vee(1)); \ [(x,y)]\mapsto \left[ g\mapsto (g-1)b -\frac{g-1}{\gamma_m-1}y\right],$$
where $b\in \bA\otimes_{\ZZ_p}V^\vee(1)$ is an element satisfying $(\varphi-1)b=x$. (Note that $\iota_{\varphi,\gamma_m}$ is denoted by $\iota^1$ in the previous subsection.)
Also, we have a natural map
$$\rho_m: H^1(C_{\varphi,\gamma_m}(L_m,V^\vee(1))) \to \left(\frac{D(V^\vee(1))}{\varphi-1}\right)^{\Gamma_{L_m}}; \ [(x,y)]\mapsto [x].$$
This is surjective: in fact, if $[x] \in D(V^\vee(1))/\varphi-1$ is $\Gamma_{L_m}$-invariant, then there is $y \in D(V^\vee(1))$ such that $(\gamma_m-1)x = (\varphi -1)y$, i.e., $[(x,y)]$ lies in $H^1(C_{\varphi,\gamma_m}(L_m,V^\vee(1)))$.
Consider the composition map
$$\lambda_m: H^1(L_m,V^\vee(1)) \xrightarrow{-\iota_{\varphi,\gamma_m}^{-1}} H^1(C_{\varphi,\gamma_m}(L_m,V^\vee(1))) \xrightarrow{\rho_m} \left(\frac{D(V^\vee(1))}{\varphi-1}\right)^{\Gamma_{L_m}}.$$

\begin{lemma}\label{tate lemma}\
\begin{itemize}
\item[(i)] The map
$$\varinjlim_m \lambda_m: H^1(L_\infty, V^\vee(1)) \to \frac{D(V^\vee(1))}{\varphi-1}$$
coincides with the inverse of $\delta$ in (\ref{def delta}).
\item[(ii)]
%The following diagram is commutative:
%$$
%\xymatrix{
%\dfrac{D(V)^{\psi=1}}{\gamma_m-1} \ar[r]^{\ell_{\psi,m}\circ \iota_m} \ar[d]_{\tau_{D,m}}& H^1(L_m,V) \ar[d]^{\tau_{H,m}} \\
%\left(\left(\dfrac{D(V^\vee(1))}{\varphi-1}\right)^{\Gamma_{L_m}}\right)^\vee \ar[r]_-{\lambda_m^\vee} & H^1(L_m,V^\vee(1))^\vee,
%}
%$$
%where $\tau_{D,m}$ is the isomorphism induced by $\tau_D: D(V)^{\psi=1} \xrightarrow{\sim} \left(\frac{D(V^\vee(1))}{\varphi-1}\right)^\vee$ and $\tau_{H,m}$ is the local Tate duality isomorphism.
%
%In other words, t
The following diagram is commutative:
$$
\xymatrix{
 H^1(L_m, V^\vee(1)) \ar@/_70pt/[dd]_-{\lambda_m}&\times & H^1(L_m,V)\ar[r]& \QQ_p/\ZZ_p \ar@{=}[dd]\\
H^1(C_{\varphi,\gamma_m}(L_m,V^\vee(1))) \ar[u]^{-\iota_{\varphi,\gamma_m}}_\simeq \ar[d]_-{\rho_m} &\times &H^1(C_{\psi,\gamma_m}(L_m,V)) \ar[u]_{\ell_{\psi,m}}^\simeq& \\
\dfrac{D(V^\vee(1))}{\varphi-1} &\times & D(V)^{\psi=1} \ar[u]_{\iota_m}\ar[r]^{\langle\cdot ,\cdot \rangle} &\QQ_p/\ZZ_p.
}$$
%where the upper pairing is the local Tate duality pairing and the bottom pairing is induced by  (\ref{Dpairing}).
\end{itemize}
\end{lemma}
\begin{proof}
(i) This is proved in the same way as Lemma \ref{lem:iotadelta}. Take an element $[x]\in \frac{D(V^\vee(1))}{\varphi-1}$. Then there is a sufficiently large $m$ such that $[x] $ is $\Gamma_m$-invariant. We then take $[(x,y)] \in H^1(C_{\varphi,\gamma_m}(L_m,V^\vee(1)))$ such that $\rho_m([(x,y)])=[x]$. Then we have
$$-\iota_{\varphi,\gamma_m}([(x,y)]) = \left[ g\mapsto (g-1)b -\frac{g-1}{\gamma_m-1}y\right],$$
where $b\in \bA\otimes_{\ZZ_p}V^\vee(1)$ is an element satisfying $(\varphi-1)b=x$. Since $(g-1)y=0$ for any $g\in G_{L_\infty}$, we have
$${\rm Res}_{L_\infty/L_m}\left(-\iota_{\varphi,\gamma_m}([(x,y)])\right)= \left[g\mapsto (g-1)b\right] \in H^1(L_\infty,V^\vee(1)).$$
By the definition of the connecting homomorphism $\delta$, the right hand side is $\delta([x])$. On the other hand, by the definition of $\lambda_m$, we have
$$\varinjlim_m \lambda_m \left( {\rm Res}_{L_\infty/L_m}\left(-\iota_{\varphi,\gamma_m}([(x,y)])\right)  \right)= [x].$$
Hence we have
$$\varinjlim_m \lambda_m \left( \delta([x])\right) = [x].$$
This completes the proof of (i).

(ii) By Theorem \ref{thm:invariant} and \cite[Prop. 4.4(f)]{herr}, we have a commutative diagram
$$\xymatrix{
H^1(L_m, V^\vee(1)) &\times & H^1(L_m,V) \ar[r]^\cup& H^2(L_m,\QQ_p/\ZZ_p(1)) \ar[d]^{\rm inv}\\
 H^1(C_{\varphi,\gamma_m}(L_m,V^\vee(1))) \ar[u]^{\iota_{\varphi,\gamma_m}}_\simeq  &\times & H^1(C_{\varphi,\gamma_m}(L_m,V)) \ar[u]_{\ell_{L_m}(\gamma_m) \iota_{\varphi,\gamma_m}}^\simeq \ar[r]& \QQ_p/\ZZ_p,
}$$
where we set $\ell_{L_m}(\gamma_m):=\frac{\log \chi_{\rm cyc}(\gamma_m)}{p^m}$ and the bottom pairing is defined by
\begin{equation}\label{cup explicit}
([(x,y)], [(z,t)]) \mapsto \langle \varphi y,z\rangle -\langle \gamma_m x,t\rangle.
\end{equation}
(Note that $\langle \varphi y,z\rangle -\langle \gamma_m x,t\rangle = \langle y,\gamma_m z \rangle - \langle x , \varphi t \rangle$ by the skew commutativity of the product in \cite[Prop. 4.4(f)]{herr}.)
We have an isomorphism
$$\iota: H^1(C_{\varphi,\gamma_m}(L_m,V)) \xrightarrow{\sim} H^1(C_{\psi,\gamma_m}(L_m,V)); \ [(x,y)] \mapsto [(-\psi(x),y)]$$
(see \cite[Lem. I.5.2]{CC}) and recall that $\ell_{\psi,m}$ is defined to be the composition map
\begin{equation}\label{defellpsi}\ell_{\psi,m}: H^1(C_{\psi,\gamma_m}(L_m,V)) \xrightarrow{\iota^{-1}} H^1(C_{\varphi,\gamma_m}(L_m,V)) \xrightarrow{\ell_{L_m}(\gamma_m) \iota_{\varphi,\gamma_m}} H^1(L_m,V).
\end{equation}
Therefore, it is sufficient to show the commutativity of the following diagram:
\begin{equation}\label{D comm diag}
\xymatrix{
 H^1(C_{\varphi,\gamma_m}(L_m,V^\vee(1)))\ar[d]_-{-\rho_m}  & \times & H^1(C_{\varphi,\gamma_m}(L_m,V))  \ar[r]^-{(\ref{cup explicit})}& \QQ_p/\ZZ_p\ar@{=}[d] \\
 \dfrac{D(V^\vee(1))}{\varphi-1} &\times & D(V)^{\psi=1} \ar[u]^-{\iota^{-1}\circ \iota_m}\ar[r]_{\langle\cdot ,\cdot \rangle}&\QQ_p/\ZZ_p .
}
\end{equation}
Take any $[(x,y)] \in H^1(C_{\varphi,\gamma_m}(L_m,V^\vee(1)))$ and $t\in D(V)^{\psi=1}$. Since $(\varphi-1)t \in D(V)^{\psi=0}$ and $\gamma_m-1$ is invertible on $D(V)^{\psi=0}$, there exists $z \in D(V)^{\psi=0}$ such that $(\gamma_m-1)z=(\varphi-1)t$, i.e., $[(z,t)] \in H^1(C_{\varphi,\gamma_m}(L_m,V))$. It is easy to see that $\iota^{-1}\circ \iota_m(t)=[(z,t)]$. To prove the commutativity of (\ref{D comm diag}), it is sufficient to show that
\begin{equation}\label{xyz}
\langle \varphi y, z \rangle -\langle \gamma_m x,  t \rangle = -\langle x,t \rangle.
\end{equation}
We compute
$$\langle \varphi y,z\rangle = \langle y, \psi z \rangle = \langle y, 0 \rangle = 0.$$
Using $(\gamma_m-1)x = (\varphi-1)y$, we see that $\gamma_m x = x $ in $\frac{D(V^\vee(1))}{\varphi-1}$. Thus we obtain the desired equality (\ref{xyz}).
\end{proof}

\begin{proof}[Proof of Proposition \ref{prop comparison}]
Recall that ${\rm Exp}_{\rm CC}^\ast$ is defined to be the inverse of $\varprojlim_m \ell_{\psi,m}\circ \iota_m$.
By Lemma \ref{tate lemma}, we have a commutative diagram
$$\xymatrix{
H^1(L_\infty,V^\vee(1)) &\times & H^1_{\rm Iw}(L_\infty/\QQ_p, V) \ar[r] \ar[d]_\simeq^{{\rm Exp}_{\rm CC}^\ast}& \QQ_p/\ZZ_p \ar@{=}[d]\\
\dfrac{D(V^\vee(1))}{\varphi-1} \ar[u]^{\delta}_\simeq & \times & D(V)^{\psi=1}  \ar[r]^-{\langle\cdot,\cdot\rangle}& \QQ_p/\ZZ_p .
}$$
Comparing this diagram with (\ref{SVcomm}), we obtain ${\rm Exp}^\ast_{\rm CC}={\rm Exp}_{\rm SV}^\ast$.
%It is sufficient to show $({\rm Exp}_{\rm SV}^\ast)^{-1} = ({\rm Exp}_{\rm CC}^\ast)^{-1}$, i.e.,
%$$\tau_H^{-1}\circ (\delta^\vee)^{-1}\circ \tau_D =  \varprojlim_m \ell_{\psi,m}\circ \iota_m.$$
%By Lemma \ref{tate lemma}(i), we have $(\delta^{-1})^\vee = \varprojlim_m (\lambda_m^\vee)$. It is easy to check $(\delta^{-1})^\vee = (\delta^\vee)^{-1}$, so we have $(\delta^\vee)^{-1} = \varprojlim_m (\lambda_m^\vee)$. This implies
%$$\tau_H^{-1}\circ (\delta^\vee)^{-1}\circ \tau_D = \varprojlim_m (\tau_{H,m}^{-1}\circ \lambda_m^\vee\circ \tau_{D,m}).$$
%By Lemma \ref{tate lemma}(ii), the right hand side is equal to $\varprojlim_m \ell_{\psi,m}\circ \iota_m$, so we have proved the claim.
\end{proof}

\subsection{The cyclotomic explicit reciprocity law}\label{sec cyc erl}

In the cyclotomic case, we can give a direct proof of Theorem \ref{key claim} without using regulators. More generally, we have the following.

\begin{theorem}[{Colmez' reciprocity law, see \cite[Thm. IV.2.1]{CC} and \cite[Thm. II.6]{B}}]\label{thm:cyc colmez}
Let $V$ be a crystalline representation of $G_{\QQ_p}$ and $T\subset V$ be its stable lattice. Then for any $m\geq 1$ and $j\in \ZZ$ we have a commutative diagram
\begin{equation*}
 \xymatrix{
   H^1_{\Iw}(L_\infty{/\QQ_p},T) \ar[d]_{{\rm tw}_{j}} \ar[r]^-{\mathrm{Exp}^*} &  D(T)^{\psi=1}  \ar[d]^{p^{-m}\varphi^{-m}}   \\
  H^1_{\Iw}(L_\infty{/\QQ_p},T(-j)) \ar[d]_{{\rm pr}_m}  &   L_m((t))\otimes_{\QQ_p} D_{\mathrm{cris}}(V) \ar[d]^{c_{m,j} \otimes e_{-j}} \\
   H^1(L_m,V(-j))\ar[r]_-{\exp_{V(-j)}^*}&   L_m \otimes_{\QQ_p} D_{\mathrm{cris}}(V(-j)),  }
\end{equation*}
where $\varphi^{-m}: D(T)^{\psi=1} \to L_m((t)) \otimes_{\QQ_p} D_{\rm cris}(V)$ is defined via $D(T)^{\psi=1}\subseteq \mathcal{R}_{\QQ_p}^+[1/t] \otimes_{\QQ_p}D_{\rm cris}(V)$ (see \cite[Thm. A.3]{B} or \cite[Thm. 3.1.1]{BF}) and $c_{m,j}$ is the map sending an element of $ L_m((t))$ to its coefficient of $t^{j}$.
\end{theorem}

\begin{proof}
By the argument in \S \ref{sec:claim}, it is sufficient to prove the case $j=0$.

By Proposition \ref{prop comparison}, we can use the definition of ${\rm Exp}^\ast$ due to Cherbonnier-Colmez, i.e., the inverse of $\varprojlim_m \ell_{\psi,m}\circ \iota_m$. Take any $y \in D(T)^{\psi=1}$ and set
$$\alpha:= \ell_{\psi,m}\circ \iota_m(y)={\rm pr}_m\circ ({\rm Exp}^\ast)^{-1}(y) \in H^1(L_m,V).$$
Then by definition we have
$$\alpha= \frac{\log \chi_{\rm cyc}(\gamma_m)}{p^m}\left[ g\mapsto \frac{g-1}{\gamma_m-1}y -(g-1)b\right],$$
where $b\in \mathbf{B} \otimes_{\QQ_p} V$ is an element satisfying $(\gamma_m-1)(\varphi-1)b=(\varphi-1)y$ in $D(T)^{\psi=0}$.
One can check that $\varphi^{-n}(b)$ converges in $ B_{\rm dR} \otimes_{\QQ_p}V$ for a sufficiently large $n$ (see \cite[Prop. I.8]{B}). We set
$$\mathfrak{e}_1:=\frac{1}{p^{n-m}}{\rm Tr}_{L_n/L_m}=\frac{1}{p^{n-m}}\sum_{\sigma \in \Gal(L_n/L_m)}\sigma \in \QQ_p[\Gal(L_n/L_m)] .$$
Then, in $H^1(L_m, B_{\rm dR}\otimes_{\QQ_p}V)$ we have
\begin{align*}
\alpha&= \frac{\log \chi_{\rm cyc}(\gamma_m)}{p^m}\left[ g\mapsto \frac{g-1}{\gamma_m-1}\varphi^{-n}(y) \right] \\
&=\frac{\log \chi_{\rm cyc}(\gamma_m)}{p^m}\left[ g\mapsto \frac{g-1}{\gamma_m-1}c_{n,0}(\varphi^{-n}(y)) \right] \\
&=\frac{\log \chi_{\rm cyc}(\gamma_m)}{p^m}\left[ g\mapsto \frac{g-1}{\gamma_m-1}\mathfrak{e}_1 c_{n,0}(\varphi^{-n}(y)) \right] \\
&= \left[ g\mapsto \log \chi_{\rm cyc}(g) c_{m,0}(p^{-m}\varphi^{-m}(y))\right],
\end{align*}
where the first equality follows from the fact that $V= ( \mathbf{B}\otimes_{\QQ_p}V)^{\varphi=1}$ is invariant under $\varphi^{-n}$, the second because $\gamma_m-1$ is invertible on $t^k L_n \otimes_{\QQ_p} D_{\rm cris}(V)$ if $k\neq 0$, the third from the fact that $\gamma_m-1$ is invertible on $(1-\mathfrak{e}_1)L_n$, and the last from
$${\rm Tr}_{L_n/L_m}\left(c_{n,0}(p^{-n}\varphi^{-n}(y))\right) = c_{m,0}(p^{-m}\varphi^{-m}(y))$$
(see \cite[Lem. II.1]{B}). Since Kato's definition of $\exp_V^\ast$ is given by the inverse of
$$L_m\otimes_{\QQ_p}D_{\rm cris}(V)=L_m\otimes_{\QQ_p}D_{\rm dR}(V) = H^0(L_m,B_{\rm dR}\otimes_{\QQ_p}V) \xrightarrow{ d \mapsto d\cup \log \chi_{\rm cyc}} H^1(L_m, B_{\rm dR}\otimes_{\QQ_p}V)$$
(see \S \ref{rem sign}), we have
$$\exp_V^\ast(\alpha)= c_{m,0}(p^{-m}\varphi^{-m}(y)).$$
This completes the proof.
\end{proof}

%\addcontentsline{toc}{part}{\large References}

%{\bf Acknowledgements:}
\begin{acknowledgement} We thank Rustam Steingart for discussions about Appendix  \ref{sec:Lie}. We also thank Laurent Berger for kindly answering our questions regarding \cite{BF}.
The project was funded by the Deutsche Forschungsgemeinschaft (DFG, German Research Foundation) under   TRR 326, {\it Geometry and Arithmetic of Uniformized Structures}, project-ID 444845124, and JSPS KAKENHI Grant Number 22K13896.
\end{acknowledgement}

%\noindent
%Otmar Venjakob\\
%Universit\"{a}t Heidelberg,  Mathematisches Institut,\\  Im Neuenheimer Feld 288,  69120
%Heidelberg,  Germany,\\
% http://www.mathi.uni-heidelberg.de/$\,\tilde{}\,$venjakob/\\
%venjakob@mathi.uni-heidelberg.de

%\printendnotes[custom]

\end{document}